\documentclass[10pt,reqno]{amsart}
\usepackage{amssymb,mathrsfs,color}
\usepackage{cite}

\usepackage{wrapfig}
\usepackage{tikz}
\usetikzlibrary{arrows,calc,decorations.pathreplacing}
\definecolor{light-gray1}{gray}{0.90}
\definecolor{light-gray2}{gray}{0.80}

\usepackage{mathtools}

\newcommand{\q}[1]{\mathcal{#1}}

\newcommand{\mr}[1]{\mathrm{#1}}

\newcommand{\ee}{\end{gather}}

\newcommand{\tendf}{\rightharpoonup}

\DeclareMathOperator{\flux}{\mathrm{Flux}}

\DeclareMathOperator{\Id}{\mathrm{Id}}

\definecolor{deepgreen}{cmyk}{1,0,1,0.5}

\newcommand{\A}{\mathcal{A}}
\newcommand{\B}{\mathcal{B}}

\newcommand{\E}{\mathcal{E}}

\newcommand{\HH}{\mathcal{H}}

\newcommand{\cS}{\mathcal{S}}
\newcommand{\Sp}{\mathbb{S}}
\newcommand{\EE}{\mathscr{E}}

\newcommand{\T}{\mathcal{T}}

\newcommand{\N}{\mathbb{N}}
\newcommand{\R}{\mathbb{R}}

\newcommand{\al}{\alpha}

\newcommand{\de}{\delta}
\newcommand{\e}{\varepsilon}
\newcommand{\fy}{\varphi}
\newcommand{\om}{\omega}
\newcommand{\la}{\lambda}

\newcommand{\s}{\sigma}

\newcommand{\io}{\iota}

\newcommand{\Ga}{\Gamma}

\newcommand{\p}{\partial}
\newcommand{\na}{\nabla}

\newcommand{\supp}{\operatorname{supp}}

\newcommand{\ext}{\operatorname{ext}}

\newcommand{\precc}{\preccurlyeq}

\makeatletter

\newcommand{\Rmnum}[1]{\expandafter\@slowromancap\romannumeral #1@}
\makeatother

\newcommand{\I}{\infty}

\newcommand{\ds}{\displaystyle}

\newcommand{\abs}[1]{\left\lvert{#1}\right\rvert}

\newcommand{\ant}[1]{\begin{align*}\begin{split} #1 \end{split}\end{align*}}

\newcommand{\EQ}[1]{\begin{equation}\begin{split} #1 \end{split}\end{equation}}
\newcommand{\Del}[1]{}

\numberwithin{equation}{section}

\newtheorem{thm}{Theorem}[section]
\newtheorem{cor}[thm]{Corollary}
\newtheorem{lem}[thm]{Lemma}
\newtheorem{prop}[thm]{Proposition}

\newtheorem{claim}[thm]{Claim}
\theoremstyle{remark}
\newtheorem{rem}{Remark}
\newtheorem{defn}{Definition}

\newcommand{\mand}{{\ \ \text{and} \ \  }}
\newcommand{\mor}{{\ \ \text{or} \ \ }}

\newcommand{\mas}{{\ \ \text{as} \ \ }}

 \def\Id{\mathrm{Id}}

\begin{document}

\title[Profiles for energy critical waves]{Profiles for the radial focusing $4d$ energy-critical wave equation}

\author{R.~C\^{o}te}
\author{C.~E.~Kenig}
\author{A.~Lawrie}
\author{W.~Schlag}

\begin{abstract} Consider a finite energy radial solution to  the focusing energy critical semilinear wave equation in $1+4$ dimensions. Assume that this solution exhibits type-II behavior, by which we mean that the critical Sobolev norm of the evolution stays bounded on the maximal interval of existence. We prove that along a sequence of times tending to the maximal forward time of existence, the solution decomposes  into a sum of dynamically rescaled solitons, a free radiation term, and an error tending to zero in the energy space. If, in addition, we assume that the critical norm of the evolution localized to the light cone (the forward light cone in the case of global solutions and the backwards cone in the case of finite time blow-up)  is less than $2$ times the critical norm of the ground state solution $W$, then the decomposition holds without a restriction to a subsequence. 

\end{abstract}

\thanks{The first author gratefully acknowledges Support of the European Research Council through the ERC Advanced Grant no. 291214, BLOWDISOL. Support of the National Science Foundation DMS-0968472 and DMS-1265249 for the second author,  DMS-1302782 for the third author, and  DMS-0617854, DMS-1160817 for the fourth author is gratefully acknowledged. }
\maketitle

\section{Introduction}\label{intro}

\subsection{History and setting of the problem}

Consider the Cauchy problem for the energy-critical, focusing wave equation in $\R^{1+4}$, namely
\EQ{\label{u gen eq}
&u_{tt}- \Delta u - u^3 = 0,\\
&\vec u(0) = (u_0, u_1), 
}
restricted to the radial setting. We study  solutions $\vec u(t)$ to~\eqref{u gen eq} in the energy space 
 \EQ{
 \vec u(t) :=(u(t), u_t(t)) \in  \HH:= \dot{H}^1 \times L^2( \R^4).
 } 
The conserved energy for solutions to~\eqref{u gen eq} is given by 
\ant{
E( \vec u)(t) := \int_{\R^3}  \left[\frac{1}{2}( \abs{u_t(t)}^2 + \abs{\nabla u(t)}^2) - \frac{1}{4} \abs{u(t)}^4\right] \, dx = \textrm{constant}.
}
As we will only be considering radial solutions to~\eqref{u gen eq}, we will slightly abuse notation by writing $u(t, x) = u(t, r)$ where here  $(r, \om)$ are polar coordinates on $\R^4$, i.e. $x= r  \om$, $r= \abs{x}$,  $\om \in  \Sp^3$. In this setting we can rewrite the equation~\eqref{u gen eq} as
\EQ{ \label{u eq}
&u_{tt}- u_{rr}- \frac{3}{r} u_r - u^3 = 0,\\
&\vec u(0) = (u_0, u_1),
}
 and the conserved energy (up to a constant multiple) by
 \EQ{\label{E}
E( \vec u)(t) = \int_0^{\infty} \left[ \frac{1}{2}( u_t^2(t) + u_r^2(t)) - \frac{1}{4} u^4(t) \right] \, r^3 \, dr.
}
We also define the local energy and localized $\HH$-norm  by 
\EQ{
&E_a^b( \vec u(t))  := \int_a^{b} \left[ \frac{1}{2}( u_t^2(t) + u_r^2(t)) - \frac{1}{4} u^4(t) \right] \, r^3 \, dr, \\
&\|\vec u(t)\|_{ \HH( a < r<b)}^2:=\int_a^{b} \left[  u_t^2(t) + u_r^2(t) \right] \, r^3 \, dr.
}
The Cauchy problem~\eqref{u eq} is invariant under the scaling
\EQ{
\vec u(t, r) \mapsto   \vec u_{\la}(t) := (\la^{-1} u( t/ \la, r/ \la), \la^{-2} u_t( t/ \la, r/ \la)).
}
One can also check that this scaling leaves unchanged the energy $E ( \vec u)$, as well as the $\HH$-norm of the initial data. It is for this reason that~\eqref{u eq} is called {\em energy-critical}.

\bigskip

This equation is locally well-posed in $\HH = \dot{H}^1 \times L^2( \R^3)$, which means that for all initial data,  $\vec u(0) = (u_0, u_1) \in \HH$ there exists a unique solution $\vec u(t)\in \HH$ to~\eqref{u eq} defined on a maximal interval of existence, $0 \in I_{\max}= I_{\max}( \vec u):=(T_-( \vec u), T_+( \vec u))$, with $\vec u \in C( I_{\max};  \HH)$ and for every compact $J \subset I_{\max}$ we have  $u  \in L^3_t( J; L^6_x(  \R^3))$.  

The Strichartz norm 
\EQ{\label{SI}
S(I):= L^3_t(I; L^6_x( \R^4))
} determines a criteria for both scattering and finite time blow-up. In particular, a solution $\vec u(t)$ globally defined for $t\in [0, \infty)$ scatters as $ t \to \infty$ to a free wave, i.e., a solution $\vec u_L(t) \in \HH$ of 
\[ \Box u_L =0 \]
if and only if  $ \|u \|_{S([0, \infty))}< \infty$. The local well-posedness theory gives the  existence of  a constant $\de>0$ so that  
\EQ{ \label{global small}
 \| \vec u(0) \|_{\HH} < \de \Longrightarrow  \| u\|_{S(\R)} \lesssim \|\vec u(0) \|_{\HH} \lesssim \de
 }
and hence $\vec u(t)$  scatters to free waves as $t \to \pm \infty$. Moreover,  we have the standard finite time blow-up criterion: 
\EQ{ \label{ftbuc}
T_+( \vec u)< \infty \Longrightarrow \|u \|_{S([0, T_+( \vec u)))} = + \infty.
}
A similar statement holds if $- \infty< T_-( \vec u)$. We also note that the same statements hold with $S(I)$ replaced with  $L^{5}_t(I; L^5_x( \R^4))$ as well, see for example~\cite{KM08}.

\bigskip

Here we will study the dynamics of solutions to~\eqref{u eq} that are bounded in the $\HH$-norm for positive times, i.e., 
\EQ{\label{type ii}
 \sup_{t\in [0, T_+( \vec u))} \| \vec u(t) \|_{ \HH}^2 :=\sup_{t\in [0, T_+( \vec u))} \| \nabla u(t)\|_{ L^2}^2 + \| u_t(t) \|_{L^2}^2 < \infty
 }
In general we will refer to such solutions as {\em type-II}, as the case with $T_+( \vec u) < \infty$ is called finite-time type-II blow-up. Type-I finite-time blow-up, also called ode blow-up, refers to solutions with, say $T_+( \vec u) < \infty$, and with the property that
\[ \liminf_{t \uparrow T_+(\vec u)} \| \vec u(t)\|_{\HH} = \infty. \]
Both type-I and type-II blow-up solutions were constructed for  \eqref{u eq} (see respectively \cite[Section 6.2]{DKM3} and \cite{KST3,HR}). In the study of long time dynamics, a crucial role is played by the stationary Aubin-Talenti solutions defined explicitly by
\EQ{\label{W}
W_{\la}(x) = \la^{-1} W( x/ \la),\quad W(x):= \left(1+\frac{\abs{x}^2}{8}\right)^{-1}.
}
$W=W(r)$ is a positive radial solution to the stationary elliptic equation 
\EQ{\label{W eq}
- \Delta W - \abs{W}^2 W =0.
}
$W$ is the unique (up to sign, dilation, and translation), amongst nonnegative nontrivial (not necessarily radial), $C^2$ solutions to~\eqref{W eq} and is unique (up to sign and dilation) amongst radial $\dot{H}^1$ solutions. $W$ is also the unique (up to translation and scaling) extremizer for the Sobolev inequality 
\[ \| f\|_{L^4( \R^4)} \le K(4, 2) \| \na f\|_{L^2( \R^4)} \]
in $\R^4$ where $K(4, 2)$ is the best constant, see \cite{Ta}. Because of this variational characterization, and its importance in variational estimates, (such as those found in \cite{DKM1, KM08}), $W$ is referred to as the ``ground state."

The second author and Merle, \cite{KM08}, gave a characterization  of the possible dynamics for~\eqref{u gen eq} for solutions with energy below the threshold formed by the ground state energy, i.e.
\[ E(\vec u) < E(W, 0). \]
For such sub-threshold solutions, the decisive factor is the size of the gradient of $u_0$ in $L^2$. Indeed, the following trichotomy holds: 
\begin{itemize}
\item If $\| \na u_0\|^2_{L^2}>\| \na W\|_{L^2}^2$ then $T_+( \vec u)< \infty$ and $T_-( \vec u)> - \infty$. In other words, $\vec u(t)$ blows up in finite time in both directions. 
\item If $\| \na u_0\|_{L^2}^2<\| \na W\|_{L^2}^2$ then $I_{\max}( \vec u) = \R$ and $\| u\|_{S( \R)}< \infty$, where $S(I)$ is a suitable Strichartz norm as in~\eqref{SI}. In other words, $\vec u(t)$ exists globally in time and scatters in both time directions. 
\item The case $\|\na u_0\|^2_{L^2} = \| \na W\|_{L^2}^2$ is impossible for sub-threshold solutions. 
\end{itemize}
Threshold solutions, namely those with energy $E( \vec u) = E(W, 0)$ were also classified by Duyckaerts, Merle \cite{DM}, see also~\cite{NKS12}.  

\bigskip

Let us now restrict to type-II solutions, i.e., those satisfying~\eqref{type ii}. 
It is known that $\| \na W\|_{L^2}^2$ is a sharp threshold for finite time blow-up and scattering. Indeed, the following generalization of the scattering part of the Kenig-Merle result in~\cite{KM08} was established in~\cite{DKM2} for $d=3, 4, 5$: If $\vec u(t)$ verifies~\eqref{type ii} and 
\ant{
&\sup_{0<t<T_+( \vec u)} \| \na u(t) \|^2_{L^2} + \frac{d-2}{2} \| \p_t u(t)\|_{L^2}^2 < \|\na W\|_{L^2}^2  \, \, \textrm{(non-radial case)} \\
&\quad \mor\\
& \sup_{0<t<T_+( \vec u)} \| \na u(t) \|^2_{L^2}< \|\na W\|_{L^2}^2  \, \, \textrm{(radial case)}
}
then $T_+( \vec u) = + \infty$ and $\vec u(t)$ scatters forward in time. 

When $d=3$, the fourth author, together with Krieger and Tataru \cite{KST3}, showed, by construction, that for every $\de>0$ there exists a type-II radial blow-up solution $\vec u(t)$ so that 
\EQ{\label{<2W de}
\sup_{t \in[0, T_+( \vec u))} \| \na u(t)\|_{L^2}^2 \le  \| \na W\|_{L^2}^2 + \de.
}
Moreover, the blow-up, say at time $T_+( \vec u) = 1$, occurs via the bubbling off of an elliptic solution $W$. In particular $\vec u(t)$ exhibits a decomposition of the form 
\EQ{ \label{kst}
 \vec u(t) = \la(t)^{-1/2}(W(r/ \la(t)), 0) + \vec  \eta(t)
 }
with $\la(t) =(1-t)^{1+ \nu}$, for $\nu>0$ (the case $0 < \nu \le 1/2$ is due to the Krieger and the fourth author, \cite{KS12}). Here the error $\vec \eta(t)$ is a regular function whose local energy inside the backwards light  cone $\{r \le 1-t\}$  vanishes as $t \nearrow 1$. 

In the $d=4$ case, Hillairet and Raphael,~\cite{HR}, exhibit $C^{\infty}$ type-II blow-up solutions $\vec u(t)$ so that~\eqref{<2W de} holds and again the blow-up at $T_+( \vec u) = 1$ occurs via the bubbling off of a $W$, with the decomposition
\ant{
\vec u(t) = \la(t)^{-1}(W(r/ \la(t)), 0) + \vec  \eta(t)
} 
where $\vec \eta (t)$ is as above and $\la(t) = (1-t) \exp(- \sqrt{\log\abs{1-t}}(1+o(1)))$ as $t \to 1$. 

It is believed that this type of bubbling behavior is characteristic of all radial type-II solutions, in the sense that all solutions $\vec u$ satisfying~\eqref{type ii}, for which $T_+(\vec u)< ~\infty$ or for which $T_+(\vec u) = + \infty$, but $\vec u$ does not scatter to zero,  exhibit a decomposition of the form~\eqref{kst} as $t \to T_+( \vec u)$, or more precisely~\eqref{bu dec} or~\eqref{global dec}, with possibly multiple profiles given by dynamic rescalings of  $W$ appearing on the right-hand side. This soliton-resolution type result was established for the radial case in $3$ space dimensions in the papers by the second author, Duykaerts, and Merle, \cite{DKM1, DKM3, DKM4}. The non-radial case, restricted to energies slightly above  the ground state energy for $d=3, 5$, was  treated in~\cite{DKM2}.

\subsection{Statements of the main results}

In this paper, we  treat the case of $4$-space dimensions by giving a characterization of the possible dynamics for radial type-II solutions to~\eqref{u eq}.
The following are the $1+4$ dimensional analogs of the main results for the $1+3$ dimensional energy critical wave equation in~\cite{DKM3}. 

We will use the notation $a_n \ll b_n$ to mean $a_n/b_n \to 0$ as $n \to \infty$, where $a_n$ and $b_n$ are two sequences of positive numbers.

Let us start with the blow up case.

\begin{thm}[Type-II blow-up solutions]\label{bu main} Let $\vec u(t)$ be a smooth solution to~\eqref{u eq} which satisfies~\eqref{type ii}, and blows-up, without loss of generality, at $T_+( \vec u)=1$. Then there exists $(v_0, v_1) \in \HH$,  a sequence of times $t_n \to 1$, an integer $J_0 \ge 1$, $J_0$ sequences  $\{\la_{j,n}\}_{n \in \N}$,  $j = 1, \dots J_0$ of positive numbers,  and signs $\io_j \in \{\pm 1\}$, such that 
\EQ{\label{bu dec}
 \vec u(t_n) =  \sum_{j=1}^{J_0} \left( \frac{\io_j}{  \la_{j,n}}W \left(\frac{\cdot}{\la_{j,n}} \right), 0\right) + (v_0, v_1) + o_{\HH}(1) \mas n \to  \infty,
 }
 with 
\EQ{
\la_{1, n} \ll \dots \ll\la_{ J_0,n} \ll 1-t_n.
}
 Furthermore, the local energy inside the light-cone is quantized: 
 \EQ{
 \lim_{t \to 1} E_0^{1-t} ( \vec u(t)) = J_0 E(W, 0),
 }
 and globally in space, we  have 
 \EQ{
 E(\vec u) = J_0 E( W, 0) + E(v_0, v_1).
 }
\end{thm}

Note that the above theorem holds only along a sequence of times. If we make an additional assumption regarding the size of the local $\dot{H}^1$-norm of $u(t)$ inside the backwards light cone, then we can prove a classification of type-II blow-up solutions which holds along all times $t \to 1$. 

\begin{thm}[Type-II blow-up below $2 \| \na W\|_{L^2}^2$]\label{bu 2W} Let $\vec u(t)$ be a smooth solution to~\eqref{u eq} which satisfies~\eqref{type ii}, and blows-up, without loss of generality, at $T_+( \vec u)=1$. Suppose in addition, that 
\EQ{
\sup_{0 \le t < 1-t} \|  u(t) \|_{\dot{H}^1( 0 < r < 1-t)}^2 < 2 \|  W\|_{\dot{H}^1}^2.
 }
 Then there exists $(v_0, v_1) \in \HH$ and a positive function $\la(t)$ with $\la(t) =o( 1-t)$ as $t \to 1$ so that 
 \EQ{ \label{2W dec}
 \vec u(t) =  \pm \left(\frac{1}{ \la(t)} W \left( \frac{\cdot}{ \la(t)} \right), 0 \right) + (v_0, v_1) + o_{\HH}(1) \mas t \to 0.
 }
\end{thm}

Next we move to the case of globally defined solutions. Here we show that at least along a sequence of times, any  global solution $\vec u(t)$ satisfying~\eqref{type ii}, asymptotically decouples into a sum of dynamically rescaled $W$'s plus free radiation, i.e., a finite energy solution $\vec v(t)$ to the free radial wave equation 
\EQ{\label{free wave}
&v_{tt} - v_{rr} - \frac{3}{r} v_r =0, \\
&\vec v(0) = (v_0, v_1) \in \HH.
}

\begin{thm}[Type-II global solutions]\label{global main} Let $\vec u(t)$ be a smooth solution to~\eqref{u eq} satisfying~\eqref{type ii}, and which is global in positive time, i.e., $T_+( \vec u) = + \infty$. Then there exists a free wave, i.e., a solution $\vec v_L(t) \in \HH$ to~\eqref{free wave},  a sequence of times $t_n \to  \infty$, an integer $J_0 \ge 0 $, $J_0$ sequences  $\{\la_{j,n}\}_{n \in \N}$,  $j = 1, \dots J_0$ of positive numbers,  and signs $\io_j \in \{\pm 1\}$, such that 
\EQ{ \label{global dec}
 \vec u(t_n) =  \sum_{j=1}^{J_0} \left( \frac{\io_j}{  \la_{j,n}}W \left( \frac{\cdot}{\la_{j,n}} \right), 0\right) + \vec v_L(t_n) + o_{\HH}(1) \mas n \to  \infty,
 }
 with 
\EQ{
\la_{1, n} \ll \dots \ll\la_{J_0, n} \ll t_n.
}
 Furthermore, for all $A>0$ the limit as $t\to \infty$ of the localized energy $E_0^{t-A} ( \vec u(t))$ exists and satisfies 
 \EQ{
 \lim_{t \to \infty} E_0^{t-A} ( \vec u(t)) = J_0 E(W, 0),
 }

\end{thm}

As in the finite time blow-up case, we can prove the global-in-time decomposition along all times $t \to \infty$ if we assume a bound on the local $\dot{H}^1$-norm of $ u(t)$ which prevents there from being more than one profile $W$ in~\eqref{global dec}.

\begin{thm}[Type-II global solutions below $2 \| \na W\|_{L^2}^2$]\label{global 2W} Let $\vec u(t)$ be a smooth solution to~\eqref{u eq} satisfying~\eqref{type ii}, and which is global in positive time, i.e., $T_+( \vec u) = + \infty$. Suppose in addition that there exists an $A>0$ so that 
\EQ{ \label{2W bound}
\limsup_{t \to \infty}  \| u(t) \|_{ \dot{H}^1(0\le r \le t-A)}^2 < 2 \| W\|_{\dot{H}^1}^2.
}
Then, there exists a solution $\vec v_L(t)\in \HH$ to~\eqref{free wave}  so that one of the following holds:
\begin{itemize}
\item[$(i)$] $\vec u(t)$ scatters to the free wave $\vec v_L(t)$ as $t \to \infty$. 
\item[$(ii)$] There exists a positive function $\la(t)$ with $\la(t) =o(t)$ as $t \to \infty$ so that
\EQ{\label{2W glob dec}
 \vec u(t) =  \pm \left(\frac{1}{ \la(t)} W \left( \frac{\cdot }{ \la(t)} \right), 0 \right) + \vec v_{L}(t) + o_{\HH}(1) \mas t \to  \infty.
 }
 \end{itemize}
\end{thm}

\subsection{Comments on the proofs} 

While many of the techniques introduced in the series of papers~\cite{DKM1, DKM3, DKM2} carry over to the even dimensional setting, several key elements of the argument are quite different when one moves away from $3$ space dimensions.  In particular, the missing ingredients in even dimensions were:
\begin{itemize}
\item[(1)] Exterior energy estimates for the underlying free radial wave equation. 
\item[(2)] A proof that that the energy of a smooth solution cannot concentrate in the self-similar region of the light-cone. 
\end{itemize}

The first of these ingredients (1) was studied in~\cite{CKS}. In fact, the main argument of~\cite{DKM3} is the proof that (2) holds for the $3d$ radial energy critical wave equation, using the exterior energy estimates for the $3d$ linear, radial wave equation proved in~\cite{DKM1}. However, in~\cite{CKS}, it is proved that the crucial exterior energy estimates established in~\cite{DKM1, DKM2} are false in even dimensions, thus rendering the use of the channel of energy method of~\cite{DKM1, DKM3, DKM2, DKM4} in doubt for the case of even dimensions. In~\cite{CKS} it is proved that the exterior energy estimate established in~\cite{DKM2} fails for radial data of the form $(0, g)$, but does hold for radial data of the form $(f, 0)$. This was used for energy critical equivariant wave maps into $\Sp^2$, to prove a classification of degree one below $3$ times the energy of the harmonic map  in~\cite{CKLS1, CKLS2}, and the soliton resolution along a sequence of times in~
\cite{Cote13}, in the spirit of~\cite{DKM3}. 

In the case of equivariant wave maps, (2) is classical and was established by Christodoulou, Tahvildar-Zadeh, \cite{CTZduke, CTZcpam} and Shatah, Tahvildar-Zadeh, \cite{STZ92, STZ94}. The classical arguments rely crucially on multiplier identities, the monotonicity of the local energy, and on the positivity of the flux -- both of which appear to be absent in the semilinear wave equation set-up. In~\cite{CKLS1, CKLS2} and later in~\cite{Cote13}, one uses (2) as in the works mentioned above to show that, along a sequence of times,  the time derivative of the solution, restricted to a suitable cone, tends to $0$, thus making it possible to apply the $d=4$ exterior energy lower bound from~\cite{CKS}, for data of the form $(f, 0)$. 


The main new ingredient in this paper is the proof of (2) for solutions to the $4d$ equation~\eqref{u eq}. In fact, the proof uses a reduction to a $2d$ equation that bears many similarities to a wave map type equation. This is the opposite of what is usually done, when equivariant wave maps are transformed to look like an energy critical nonlinear wave equation.

The crucial monotonicity of the localized energy and the positivity of the flux are established in the relevant regions after the regular part of the solution is considered separately from the singular part. One can then follow the classical techniques for wave maps to prove $(2)$ for radial solutions to~\eqref{u eq}. With the weakened version of (1) proved in~\cite{CKS} for data $(f, 0)$, and (2) in hand, one can then follow the arguments in~\cite{DKM1,DKM2,DKM3}, and \cite{CKLS1,CKLS2} to establish the main results. New refined techniques from~\cite{DKM6} are also used to prove Theorem~\ref{bu 2W} and Theorem~\ref{global 2W}.

The vanishing of the energy in the self-similar regions proved in the previous sections allows one to deduce a  vanishing of the $L^2$ norm of the time derivative of the singular part of the solution along a sequence of times. The vanishing time derivative then allows one to conclude that all the profiles in the Bahouri-Gerard profile decomposition of the solution along this sequence must be either $0$ or $\pm W$. The error term in the profile decomposition is then shown to vanish in the energy space using the exterior energy estimates for the underlying free equation as in \cite{CKLS1, CKLS2}. One main difference with~\cite{CKLS1, CKLS2} in the argument is that there the harmonic map must be extracted before the machinery of profile decompositions can be applied due to the geometric nature of wave maps. Here one can work directly with a profile decomposition for $\vec u(t_n)$. 

\bigskip

In Section~\ref{prelim} we recall various preliminary results including the linear and nonlinear profile decompositions from~\cite{BG}, the exterior linear estimates for the free equation from~\cite{CKS},  and the rigidity of radial compact trajectories proved in~\cite{DKM1}. 

In Section~\ref{self sim bu} we show that no energy can concentrate in the self-similar region of the backwards light cone for type-II solutions that blow up in finite time, i.e., we prove (2) in the finite time blow-up case. In Section~\ref{self sim global} we prove the vanishing of energy in the self-similar region of the forward light cone for solutions that exist for all positive times, proving (2) for global solutions. These two sections contain the main technical novelties in this paper as the classical $2d$ geometric arguments from~\cite{CTZduke, CTZcpam, STZ92, STZ94} are adapted to a focusing  $4d$ semilinear equation once crucial positivity properties are revealed.  

In Section~\ref{mains} we prove Theorem~\ref{bu main} and Theorem~\ref{global main} using the arguments from~\cite{CKLS1, CKLS2} which in turn were based on the channel of energy methods introduced in~\cite{DKM1, DKM3, DKM2}, which we also rely on here.

Finally, in  Section~\ref{2W bu} and Section~\ref{2W global} we prove Theorem~\ref{bu 2W} and Theorem~\ref{global 2W}. Here the argument has its foundations in the techniques from~\cite{DKM1, DKM3} but also requires new methods recently developed in~\cite{DKM6}.

\subsection{Notation} As we are dealing strictly with radial functions, we will often abuse notation by writing $f(x) = f(\abs{x}) = f(r)$. For a space-time function $f(t, r)$ we will sometimes use the notation $\abs{\nabla_{t, x} f(t, r)}^2 = f_t^2(t, r) + f_r^2(t, r)$. For spacial integrals of radial functions we will ignore  a dimensional constant by writing
\ant{
\int_{\R^4} f(x) dx := \int_0^{\infty} f(r) \, r^3 \, dr.
}

\section{Preliminaries}\label{prelim}

 \subsection{Energy trapping}
 
 We recall a few variational results from~\cite{DKM1, KM08} which give a useful characterization of the threshold energy $E(W, 0)$.  The key point here is that $W$ is the unique minimizer, up to translation, scaling and constant multiplication of the Sobolev embedding: 
\ant{
 \|f \|_{L^4( \R^4)} \le K(4, 2) \| \na f\|_{L^2( \R^4)},
}
 where $K(4, 2)$ is the optimal Aubin-Talenti constant,~\cite{Au, Ta}. Using the equation~\eqref{W}, one can show that in fact, 
\EQ{
\frac{1}{4} \| \na W\|_{L^2}^2 = E(W, 0),
}
and a variational argument yields the following useful result from~\cite{DKM1,DKM2, KM08}. 

\begin{lem}[{\cite[Claim~$2.3$]{DKM1},\cite[Claim~2.4]{DKM2}}] \label{E<2W} 
Let $f \in \dot{H}^1( \R^4)$. Then 
\EQ{
\| \na f\|^2_{L^2} \le \|\na W\|_{L^2}^2 \mand E(f, 0) \le E(W, 0) \Longrightarrow \frac{1}{4} \| \na f\|_{L^2}^2 \le E(f, 0).
}
Moreover,  there exists $c>0$ such that if $\| \na f\|_{L^2}^2 \le 2 \| \na W\|_{L^2}^2$ then 
\EQ{
 E(f, 0) \ge c \min\{ \|\na f\|_{L^2}^2, \, 2\|\na W\|_{L^2}^2 - \|\na f\|_{L^2}^2\} \ge 0
}
\end{lem} 
 \subsection{Exterior energy estimates and linear theory}
 Exterior energy estimates for the free radial wave equation established by the first, second, and fourth authors in \cite{CKS} will play a crucial role.   In particular, we will use the fact that free radial waves $\vec v(t)$ in $4$ space dimensions with zero initial velocity, i.e., with data $(f, 0)$, maintain a fixed percentage of their energy on the exterior of the forward light cone emanating from the origin. 
 
 We will denote a solution $\vec v(t)$ to the free wave equation~\eqref{free wave}, with initial data $(f, g) \in \HH$, by 
\ant{ 
\vec v(t) = S(t)( f, g).
}

\begin{prop}[{\cite[Corollary $5$]{CKS}}] \label{lin ext estimate}  
There exists $\al_0>0$ such that for all $t  \in \R$ we have
\begin{align} \label{ext en est}
\|S(t)(f, 0)\|_{\HH( r \ge\abs{t})} \ge  \al_0 \|f\|_{\dot{H}^1}
\end{align}
for all radial data $(f, 0) \in \HH$. 
\end{prop}

\begin{rem}  We note that estimates~\eqref{ext en est} with data $(0, g)$ or $(f, g)$ with $g \neq 0$ are false, see~\cite{CKS}. In fact one recovers the analog of~\eqref{ext en est} for data $(f, 0)$ in dimension $d \equiv 0\mod 4$ and for data $(0, g)$ in dimensions $d \equiv 2 \mod 4$. This is different from the odd dimensional case, where the analog of~\eqref{ext en est} holds for general radial data $(f, g)$ for either all positive, or all negative times, see~\cite{DKM1}.
\end{rem} 

 We have the following vanishing of the energy away from the forward light cone proved in~\cite{CKS}. 

\begin{prop}[{\cite[Theorem $4$]{CKS}}] \label{prop:6}
Let $(f,g) \in (\HH)(\R^{d})$ be radial. Then we have the following vanishing of the energy away from the forward light-cone $\{|x|=t\ge0\}$:
\begin{equation*}
\lim_{T \to +\infty} \limsup_{t \to +\infty} \| \nabla_{t,x} S(t)(f,g) \|_{L^2(||x|-t| \ge T )}  =0.
\end{equation*}
\end{prop}

\subsection{Profile Decomposition} 
Another essential tool in our analysis will be the linear and nonlinear profile decompositions of Bahouri-Gerard, \cite{BG}. We begin with a profile decomposition for a bounded sequence $\vec u_n$ in the energy space in terms of free waves. The statement below was proved in $3$ space dimensions in \cite{BG} and extended to other dimensions, including $4$ space dimensions in~\cite{Bu}. 

\subsubsection{Linear profile decomposition}

\begin{thm}[{\cite[Main Theorem]{BG}, \cite[Theorem $1.1$]{Bu}}] \label{BG}
Consider a sequence  $\vec u_n =(u_{n, 0}, u_{n,1 }) \in \HH:= \dot{H}^1 \times L^2( \R^4)$, that is radial, and such that $ \|u_n\|_{\HH} \le C$. Then, up to extracting a subsequence,  there exists a sequence of free radial waves $\vec U_L^j  \in \HH$,  a sequence of times $\{t_{j,n}\}\subset \R$, and sequence of scales $\{\la_{j,n}\}\subset (0, \infty)$, and free wave $\vec w_n^k \in \q C(\R,\HH)$ (i.e., solution to \eqref{free wave}) such that\EQ{ \label{free prof} 
&u_{n, 0}(r) = \sum_{j=1}^k \frac{1}{\la_{j,n}} U_L^j \left( - \frac{t_{j,n}}{ \la_{j,n}}, \frac{r}{ \la_{j,n}} \right) + w_{n}^k(0,r)\\
& u_{n, 1}(r)= \sum_{j=1}^k \frac{1}{(\la_{j,n})^2} \partial_t U_L^j \left( - \frac{t_{j,n}}{\la_{j,n}}, \frac{r}{\la_{j,n}} \right) + \partial_t w_{n}^k(0,r)
}
and for any $j \le k$, that 
\begin{align} \label{w weak}
(\la_{j,n} w_n^k( \la_{j,n} t_{j,n},  \la_{j,n}\cdot) , \la_{j,n}^2 \partial_t w_n^k( \la_{j,n} t_{j,n},  \la_{j,n}\cdot)) \rightharpoonup 0\quad \textrm{weakly in} \quad \HH. 
\end{align}
In addition, for any $j\neq k$ we have
\EQ{ \label{scales}
\frac{\la_{j,n}}{\la_{k,n}} + \frac{\la_{k,n}}{\la_{j,n}} + \frac{\abs{t_{j,n}-t_{k,n}}}{\la_{j,n}} + \frac{\abs{t_{j,n}-t_{k,n}}}{\la_{k,n}} \to \infty \quad \mas n \to \infty.
}
Moreover, the errors $\vec w_n^k$ vanish asymptotically in the Strichartz space,  we have 
\EQ{ \label{w in strich}
\limsup_{n \to \infty} \left\| w_{n}^k \right\|_{L^{\infty}_tL^4_x \cap S( \R \times \R^4)}  \to 0 \quad \textrm{as} \quad k \to \infty.
}
Finally, we have the almost-orthogonality of the free energy as well as of the nonlinear energy~\eqref{E} of the decomposition: 
\begin{align} \label{free en dec}
&\|\vec u_n\|_{\HH}^2 = \sum_{1 \le j \le k} \left\| \vec U_L^j \left( - \frac{t_{j,n}}{\la_{j,n}} \right) \right\|_{\HH}^2  +  \|\vec w_n^k (0) \|_{\HH}^2 + o_n(1),\\
& E( \vec u_n) = \sum_{1 \le j \le k} E \left( \vec U_L^j \left( -\frac{ t_{j,n}}{\la_{j,n}} \right) \right)  +  E(\vec w_n^k(0)) + o_n(1),\label{nonlin en dec}
\end{align} 
as  $n \to \infty$.
\end{thm}

\begin{rem} \label{prof rem}
By rescaling and time-translating each profile $ \vec U^{j}_L$ appearing in~\eqref{free prof}, and by extracting subsequences we can, without loss of generality, assume  for each fixed $j$ that either we have 
\EQ{
 \forall n, \, \, t_{j,n} = 0,  \mor  \lim_{n \to \infty} \frac{t_{j,n}}{ \la_{j,n}} = \pm \infty.
}
Moreover, we can assume that for all $j$ the sequences $\{t_{j,n}\}$ and $\{\la_{j,n}\}$ have limits in $[-\infty, + \infty]$ and $[0, + \infty]$ respectively. 
\end{rem} 

We will also need the following refinement of the almost-orthogonality of the free energy, namely that the Pythagorean decomposition~\eqref{free en dec} of the $\HH$ norm of the sequence remains valid even after a spacial localization. This was proved for dimension $3$ in~\cite{DKM3} and for even dimensions in~\cite{CKS}. 

\begin{prop}[{\cite[Corollary $8$]{CKS}}] \label{loc en dec}Consider a sequence of radial data $\vec u_n \in \HH=\dot{H}^1 \times L^2( \R^4)$ such that $ \|u_n\|_{\HH} \le C$, and a profile decomposition of this sequence as in Theorem \ref{BG}. Let $\{r_n\}\subset (0, \infty)$ be any sequence. Then we have 
\ant{
 \| \vec u_n\|_{\HH(r \ge r_n)}^2  = \sum_{1 \le j \le k} \left\| \vec U_L^j \left( - \frac{t_n^j}{ \la_n^j} \right) \right\|_{\HH(r \ge r_n/ \la_n^j)}^2  +  \|\vec w_n^k(0)\|_{\HH(r \ge r_n)}^2 + o_n(1)
}
as  $n \to \infty$.
\end{prop}

We also require the following technical lemmas for free waves proved in \cite{DKM1} in odd dimensions and in \cite{CKS} in even dimensions.  

\begin{lem}[{\cite[Lemma 4.1]{DKM1},\cite[Lemma 9]{CKS}}] \label{lem:scaled_lin_disp}
Let $\vec v(t)$ be a radial solution to the linear wave equation \eqref{free wave}, and $\{t_n\} \subset \R$, $\{\lambda_n\} \subset \R^*_+$ be two sequences. Define the sequence 
\EQ{\label{vn}
 v_n(t,x) = \frac{1}{\lambda_n} v \left( \frac{t}{\lambda_n}, \frac{x}{\lambda_n} \right). 
 }
 Assume that $\ds \frac{t_n}{\lambda_n} \to \ell \in \overline{\R}$. Then
\begin{align*} 
& \textrm{If } \ell \in \{ \pm \infty \}, & \quad \limsup_{n \to \infty}  \| \nabla_{x,t} v_n(t_n) \|^2_{L^2(||x| - |t_n|| \ge R \lambda_n)} & \to 0 \quad \textrm{as} \quad R \to +\infty, \\
& \textrm{If } \ell \in \R, & \quad 
\limsup_{n \to \infty} \| \nabla_{x,t} v_n(t_n) \|^2_{L^2(|\log (\abs{x}/\lambda_n)| \ge \log R)} & \to 0 \quad \textrm{as} \quad R \to +\infty.
\end{align*}
\end{lem}

\begin{lem}[{\cite[Lemma 2.5]{DKM1}}] \label{lem:2.5} 
Let $\vec v_n$ be defined as in~\eqref{vn} and assume it has a profile decomposition as in Theorem~\ref{BG}. 
If 
\[ \lim_{R \to +\infty} \limsup_{n \to \infty} \int_{|x| \ge R \mu_n} |\nabla v_{0,n}|^2 + |v_{1,n}|^2 =0, \]
 then for all $j$ the sequences $\ds \left\{ \frac{\lambda_{j,n}}{\mu_n} \right\}_n$,  $\ds \left\{ \frac{t_{j,n}}{\mu_n} \right\}_n$ are bounded. Moreover, there exists at most one $j$ such that $\ds \left\{ \frac{\lambda_{j,n}}{\mu_n} \right\}_n$ does not converge to 0.
\end{lem}

We will also need the following result about sequences of radial free waves with vanishing Strichartz norms established in \cite{CKS} for even dimensions and which is the analog \cite[Claim $2.11$]{DKM1}, where the result was proved in odd dimensions only.

\begin{lem}[{\cite[Lemma $11$]{CKS}, \cite[Claim $2.11$]{DKM1}}] \label{local scat lem} 
Let $\vec w_n(0)= (w_{n,0}, w_{n,1})$ be a radial uniformly bounded sequence in $\HH= \dot{H}^1 \times L^2(\R^4)$ and let $\vec w_n(t)\in \HH$ be the corresponding sequence of radial $4d$ free waves. Suppose that 
\ant{
\|w_n \|_{S(\R)} \to 0 \mas  n \to \infty,
}
where $S(I)$ is as in~\eqref{SI}.  Let $\chi \in C^{\I}_0(\R^4)$ be radial so that $\chi \equiv 1$ on $\abs{x} \le 1$ and $\textrm{supp} \chi \subset \{ \abs{x} \le 2\}$. Let $\{\la_n\} \subset (0, \infty)$ and consider the truncated  data 
\ant{
\vec v_n(0):=  \fy(r/ \la_n) \vec w_n(0),
}
where either $\fy= \chi$ or $\fy= 1- \chi$. Let $\vec v_n(t)$ be the corresponding sequence of free waves. Then 
\ant{
\|v_n \|_{S(R)} \to 0 \quad \textrm{as} \quad n \to \infty.
}
\end{lem}

\subsubsection{Nonlinear profiles}\label{nonlinear prof}

\begin{defn}\label{nonlin prof def}
Let $\vec U_{\mr L}$ be a linear solution to \eqref{free wave}, and $\ell \in [-\infty, +\infty]$. We define the nonlinear profile associated to $(\vec U_{\mr L}, \ell)$ as the unique nonlinear solution $\vec U(t)$ to \eqref{u eq}, defined on a neighborhood of $\ell$, and such that
\[ \| \vec U(t) - \vec U_{\mr L}(t) \|_{\HH} \to 0 \quad \text{as} \quad t \to \ell. \]
\end{defn}
Existence and uniqueness of $\vec U(t)$ are consequences of the local Cauchy theory for \eqref{u eq},~\cite{KM08,SS93, SS94}, and more precisely of the existence of wave operators if $\ell$ is infinite. It is important to note that in the latter case $\ell \in \{ +\pm \infty \}$, the nonlinear profile $\vec U$ scatters at $\ell$: for example if $\ell=+\infty$,
\EQ{
s_0> T_{-}( \vec U) \Longrightarrow  \| U \|_{S((s_0, \infty))}< \infty.
}
A similar statement holds for $\ell=-\infty$.

In the case of a profile decomposition as in \eqref{free prof} with profiles $\{\vec U^j_L\}$ and parameters $\{t_{j,n}, \la_{j,n}\}$ we will denote by $\{\vec U^j\}$ the non-linear profiles associated to $\ds \left( \vec U^j_L, \lim_{n \to +\infty} - \frac{t_{j,n}}{ \la_{j,n}} \right)$ (we recall that this limit exists by assumption, as explained in Remark \ref{prof rem}). For convenience, we will often use the notation 
\EQ{ \label{V}
U_{L, n}^j(t, r) & :=  \frac{1}{\la_{j,n}} U_L^j\left( \frac{t-t_{j,n}}{\la_{j,n}}, \frac{r}{ \la_{j,n}}\right),\\
U_{n}^j(t, r) & := \frac{1}{\la_{j,n}} U^j\left( \frac{t-t_{j,n}}{\la_{j,n}}, \frac{r}{ \la_{j,n}}\right).
}

\begin{prop}[Nonlinear profile decomposition] \cite{CKLS1, DKM1}\label{nonlin profile} Let $(u_{n, 0}, u_{n, 1}) \in \HH$ be a bounded sequence together with its profile decomposition as in~\eqref{free prof}. Let $ \{\vec U^j\}$, be the associated nonlinear profiles. Let $\{s_n\}  \subset (0, \infty)$ be any sequence of times so that for all $j \ge 1$,
\EQ{
\forall n, \, \, \frac{s_n - t_{j,n}}{ \la_{j,n}} < T_+( \vec U^{j}) \mand   \limsup_{n \to \infty}  \|U^j\|_{S\left( \frac{-t_{j,n}}{ \la_{j,n}},  \frac{ s_n- t_{j,n}}{ \la_{j,n}} \right)} < \infty, \\
  }
  If $\vec u_n(t) \in \HH$ is the solution to~\eqref{u eq} with initial data $\vec u_n(0) = (u_{n, 0}, u_{n, 1})$ then $\vec u_n(t)$ is defined on $[0, s_n)$ and 
\[    \limsup_{n \to \infty} \|u_n\|_{S([0, s_n))} < \infty.  \]
Moreover the following nonlinear profile decomposition holds: For $\eta_n^k$ defined by 
\begin{gather} \label{BG nonlin}
 \vec u_{n}(t,r) = \sum_{1\le j<k}  \vec U^j_{ n}(t, r)  + \vec w_{n}^k(t) + \vec{\eta}_{n}^k(t), 
\end{gather}
we have 
\[
\lim_{k \to \infty} \limsup_{n \to \infty}  \left(\| \eta_n^k\|_{S([0,  s_n))} + \|  \vec {\eta}_n^k\|_{L^{\infty}_t([0, s_n); \HH)} \right) = 0.
\]
Here $w_{n}^k(t) \in \HH$ is as in Proposition~\ref{BG} and  $V_{ n}^j$ is defined as in~\eqref{V}. Also, we note that an analogous statement holds for $ s_n<0$. 
\end{prop}

\begin{defn}[Ordering of the profiles, \cite{DKM6}] \label{re order}
Let $\{ \vec U^j_{\mr L}, \{t_{j,n}, \la_{j,n}\} \}$ be a profile decomposition as in~\eqref{free prof}, and let $\vec U^j$ their nonlinear profiles.
We introduce the following pre-order $\preccurlyeq$ on the profiles as follows. For $j,k \ge 1$, we say that
\[ \{\vec U_{\mr L}^j, \{t_{j,n}, \la_{j,n}\} \} \preccurlyeq \{\vec U_{\mr L}^k, \{t_{k,n}, \la_{k,n}\} \} \quad \text{(or simply } j \preccurlyeq k \text{ if there is no ambiguity)}   \]
if one of the following holds:
\begin{enumerate}
\item the nonlinear profile $\vec U^k$ scatters forward in time.
\item the nonlinear profile $\vec U^j$ does not scatter forward in time, and
\[ \forall\,  T < T_+(\vec U^j), \quad \lim_{n \to +\infty} \frac{\lambda_{j,n} T+ t_{j,n} - t_{k,n}}{\lambda_{k,n}} < T_+(\vec U^k). \]
(The above limit exists due to the arguments in~\cite[Discussion after (3.16) and Appendix A.1]{DKM6}.)
\end{enumerate}
We say that $\{\vec U_{\mr L}^j, \{t_{j,n}, \la_{j,n}\} \} \prec \{\vec U_{\mr L}^k, \{t_{k,n}, \la_{k,n}\} \}$ if 
\[ \{\vec U_{\mr L}^j, \{t_{j,n}, \la_{j,n}\} \} \preccurlyeq \{\vec U_{\mr L}^k, \{t_{k,n}, \la_{k,n}\} \} \ \  \text{and} \ \ \{\vec U_{\mr L}^k, \{t_{k,n}, \la_{k,n}\} \} \not\preccurlyeq \{\vec U_{\mr L}^j, \{t_{j,n}, \la_{j,n}\} \} . \]
\end{defn}

\begin{lem}[{\cite[Claim 3.7]{DKM6}}] \label{lem:3.7}
Let $(u_{0,n},u_{1,n}) \subset  \HH$ be a bounded sequence with profile decomposition $\{\vec U_{\mr L}^j, \lambda_{j,n},t_{j,n}\}_{j \in \N}$. Then one can assume without loss of generality that the profiles are ordered, that is
\[ \forall i \le j, \quad \{\vec U_{\mr L}^i, \lambda_{i,n}, t_{i,n}\} \preccurlyeq \{\vec U^j_{\mr L}, \lambda_{j,n}, t_{j,n}\}. \]
\end{lem}

\subsection{Classification of pre-compact solutions} Finally, we recall the following classification of finite energy solutions $\vec u(t) \in \HH$ to~\eqref{u eq} that have pre-compact trajectories in $\HH$ up to symmetries. In particular, we say that a solution $\vec u(t)$ has the {\em compactness property} on an interval $I \subset \R$ if there exists a function $\la:  I \to (0, \infty)$ so that the trajectory 
\ant{
K= \left\{  \left(\frac{1}{ \la(t)} u\left(t, \frac{\cdot}{ \la(t)}\right),  \frac{1}{ \la^2(t)} \p_t u\left(t, \frac{\cdot}{ \la(t)} \right)\right) \mid t \in I\right\} \subset \HH
}
is pre-compact in $\HH$. A complete classification of solution $\vec u(t)$ with the compactness property was obtained in \cite{DKM1}. In particular there it was shown that $\vec u(t)$ is either identically $0$ or is $W$ up to a rescaling. 
\begin{thm} \cite[Theorem $2$]{DKM1} \label{compactness} Let $\vec u(t) \in \HH$ be a nontrivial solution to~\eqref{u eq} with the compactness property on its maximal interval of existence $I_{\max}$. Then  there exists $\la_0>0$ so that 
\ant{
 u(t, r) = \pm\frac{1}{\la_0} W\left( \frac{ r}{\la_0} \right).
 }
 \end{thm}


\section{Self-similar and exterior regions: blow-up solutions}\label{self sim bu}
The goal of this section is to show that a type-II blow-up solution $\vec u(t)$, with, say, $T_+( \vec u) = 1$, cannot concentrate any energy in the self similar region $ r \in [\la (1-t), 1-t]$ for any fixed $0 < \la<1$. 

\begin{thm} \label{self sim} Let $ \la \in (0, 1)$. Then for any smooth solution $\vec u(t)$ to~\eqref{u eq} such that 
\EQ{
 \sup_{t \in [0, 1)} \| \vec u(t) \|_{\HH} < \infty,
}
we have  
\EQ{
\lim_{t \nearrow 1} \int_{ \la (1-t)}^{ (1-t)}  \left[u_t^2(t, r) + u_{r}^2(t, r) + \frac{u^2(t,r)}{r^2} \right] \, r^3 \, dr = 0.
}

\end{thm}

\subsection{Extraction of the regular part}\label{reg part}
First, we define the regular and singular parts of a solution $\vec u(t)$ which blows up at $T_+( \vec u)$ and satisfies~\eqref{type ii}, following the notation in~\cite{DKM1, DKM3}. Indeed, by~\cite[Section~$3$]{DKM1}, there exists $\vec v = (v_0, v_1) \in \HH$,  so that 
\ant{
\vec u(t) \rightharpoonup  (v_0, v_1) \mas t \to 1,
}
weakly in $\HH$. Moreover, if we denote by $\vec v(t)$ the solution to~\eqref{u eq} with initial data at time $t=1$, $\vec v(1) =(v_0, v_1)$, and maximal interval of existence $I_{\max}( \vec v) = (T_-( \vec v), T_+( \vec v))$, then for all $t \in [t_-, 1)$ with $t_-> \max(T_-( \vec u), T_- ( \vec v))$ we have
\EQ{\label{supp a}
\vec u(t, r) = \vec v(t, r)  \quad   \forall\, \,  r \ge 1-t.
}
We thus define the {\em singular part} of $\vec u(t)$ as the difference, 
\EQ{\label{a def}
\vec a(t) :=  \vec u(t) - \vec v(t),
}
and we remark  that $\vec a(t)$ is well defined for $t \in [t_-, 1)$ for all $t_-> \max(T_-( \vec u), T_- ( \vec v))$ and that $\vec a(t)$ is supported in the backwards light cone $$\{(t, r) \mid t_- \le t < 1, \, \, 0 \le r \le 1-t\}.$$ 
We call $\vec v(t)$ the {\em regular part} of $\vec u(t)$.

We will require the following simple estimates for $\vec v(t)$, which follow easily from the fact that the evolution $ t \mapsto \vec v(t)$ is continuous in $\HH$ at $t=1$,
\begin{lem} \label{v loc decay}Let $\vec v(t)$ be the regular part  of $\vec u(t)$ as defined above. Then 
\begin{align}\label{v loc en}
&\lim_{t \to 1} \int_0^{1-t} \left[ v_t^2(t, r)+ v_r^2(t, r) + \frac{v^2(t, r)}{r^2} \right] \, r^3 \, dr = 0,\\ \label{v pw}
&\sup_{0 \le r\le 1-t} \abs{r v(t, r)} \to 0 \mas t \to 1.
\end{align}
\end{lem}

\begin{proof}
Indeed, the continuity of $t \mapsto \vec v(t)\in \HH$ at $t =1$ gives the result for the first two terms in the integral~\eqref{v loc en}. The third term in~\eqref{v loc en} and~\eqref{v pw} then follow as direct consequences of the following techical lemma which we will also use in Section $4$. 
\begin{lem} \label{lem:hardy_loc}
Assume $\displaystyle \int_0^{+\infty} |\partial_r w(\rho)|^2 \rho^3 d\rho < +\infty$ and can be approximated by $C^{\infty}$ functions in this norm, i.e.,  $w \in \dot{H}^1(\R^4)$. Then $rw(r) \to 0$ as $r \to 0$ and as $r \to +\infty$, and for all $r \ge 0$,
\begin{align}
\abs{ r w(r)}^2 & \le \frac{1}{2} \int_r^{\infty} \abs{w_r( \rho)}^2\, \rho^3 \, dr, \label{bound:infty_loc_infty}\\
\int_r^\infty |w(\rho)|^2 \rho \, d\rho &  \le   \int_r^\infty |\partial_r w(\rho)|^2 \rho^3 d\rho. \label{bound:hardy_loc_infty}
\end{align}
and for $0 < s <r$ we have 
\begin{align}
\abs{ r^2 w^2(r) - s^2 w^2(s)} & \le 3\int_s^r w^2( \rho) \, \rho \, dr + \int_s^r \abs{w_r(\rho)}^2 \rho^3 \, d \rho, \label{rs}\\
 \sup_{0< s \le r}|s w(s)|^2 & \le 3\int_0^r w^2( \rho) \, \rho \, dr + \int_0^r |\partial_r w(\rho)|^2 \rho^3 d\rho . \label{bound:infty_loc}
\end{align}
\end{lem}
\begin{proof} By density, we can prove the lemma for $w \in C^{\infty}_0$. First, we note that 
\ant{
\abs{w(r)} &= \abs{ - \int_r^{\infty} \p_r w(\rho) \, d\rho} \le \left( \int_{r}^\infty w_r^2( \rho) \, \rho^3 \, d \rho\right)^{\frac{1}{2}} \left( \int_{r}^\infty \rho^{-3} \, d \rho\right)^{\frac{1}{2}}\\
& = \left( \int_{r}^\infty w_r^2( \rho) \, \rho^3 \, d \rho\right)^{\frac{1}{2}} \frac{1}{r \sqrt{2} }
}
from which~\eqref{bound:infty_loc_infty} follows. Next, we have 
\EQ{ \label{pr id}
\p_r( r^2 w^2(r)) = 2 r \, w^2(r) + 2 r^2 w(r) w_r(r).
}
Thus, 
\ant{
r^2 w^2(r) = - 2\int_r^{\infty} w^2(\rho) \, \rho\, d \rho - 2 \int_r^{\infty}w_r^2(\rho) w( \rho)  \rho^2 \, d \rho,
}
so that 
\ant{
2 \int_r^{\infty} w^2(\rho) \, \rho \, d \rho& \le - 2 \int_r^{\infty}w_r^2(\rho) w( \rho)  \rho^2 \, d \rho\\
& \le 2 \left( \int_{r}^\infty w_r^2( \rho) \, \rho^3 \, d \rho\right)^{\frac{1}{2}}\left( \int_{r}^\infty w^2( \rho) \, \rho \, d \rho\right)^{\frac{1}{2}}
}
which gives~\eqref{bound:hardy_loc_infty}. To prove~\eqref{rs}, we integrate~\eqref{pr id} to obtain
\ant{
\abs{ r^2 w^2(r) - s^2 w^2(s)} &\le 2\int_s^r w^2( \rho) \, \rho \, dr +2 \int_s^r \abs{w(r)}\abs{w_r^2(\rho)} \rho^2 \, d \rho\\
& \le 3\int_s^r w^2( \rho) \, \rho \, dr + \int_s^r w_r^2(\rho) \rho^3 \, d \rho
}
as desired. By~\eqref{bound:hardy_loc_infty} with $r=0$ we see that $\int_0^{\infty} w^2 (\rho) \,\rho\,  d  \rho < \infty$. Hence~\eqref{rs} implies that there exists $\ell \in\R$ so that 
\EQ{
\lim_{r \to 0} r^2 w^2(r) = \ell
}
exists. Assume, for contradiction that $\ell \neq 0$. Then, there exists $r_0>0$ so that $$w^{2}(r) \ge \frac{\ell}{2r^2}$$ for all $r \le r_0$. But this contradicts the fact that $\int_0^{\infty} w^2 (\rho)  \rho \, d  \rho < \infty$. Finally,~\eqref{bound:infty_loc} follows from~\eqref{rs} now that we know $sw(s) \to 0$ as $s \to 0$. 
\end{proof}

This also completes the proof of Lemma~\ref{v loc decay}.
\end{proof} 

\subsection{Reduction to a $2d$ equation} 
The proof of Theorem~\ref{self sim} relies crucially on the observation that~\eqref{u eq} can be reduced to a $2d$ wave maps-type equation on which the fundamental techniques introduced by Christodoulou, Tahvildar-Zadeh, \cite{CTZduke, CTZcpam}, and Shatah, Tahvildar-Zadeh, \cite{STZ92, STZ94} can be applied after we have localized to the light cone and identified the regular part $\vec v(t)$ of the solution $\vec u(t)$. Indeed, define 
\EQ{\label{psi def}
\vec \psi(t, r):= r \,\vec u(t, r).
}
Since $\vec u(t)$ solves~\eqref{u eq}, we see that $\vec \psi(t)$ solves 
\EQ{\label{psi eq}
 \psi_{tt}- \psi_{rr} - \frac{1}{r} \psi_r + \frac{\psi- \psi^3}{r^2} = 0.
 }
We define 
\EQ{ \label{f F def}
&f( \psi) := \psi- \psi^3,\\
&F(\psi) := \int_0^\psi f(\al) \,d \al = \frac{1}{2} \psi^2 - \frac{1}{4}\psi^4 = \frac{1}{2} \psi^2[1- \psi^2/2].
}
Similarly, for the regular part $\vec v(t)$ we define 
\EQ{ \label{phi def}
\phi(t, r) := r\,  \vec v(t, r).
}
Using Lemma~\ref{v loc decay} and the fact that $\phi_r =  r v_r +v$ we obtain the following estimates for $\vec \phi(t)$: 
\begin{lem}\label{phi decay}
Let $\vec \phi(t)$ be defined as in~\eqref{phi def}. Then 
\begin{align}
&\lim_{t \to 1}  \int_0^{1-t} \left[ \phi_t^2(t, r) + \phi_r^2(t, r)\right] r \, dr = 0,\label{phi en dec}\\
& \sup_{0 \le r \le 1-t} \abs{ \phi(t,r)}  \to 0 \mas t \to 1. \label{phi pw dec}
\end{align}

\end{lem}

We also note that by Hardy's inequality,~\eqref{bound:hardy_loc_infty}, and since $\psi_r -\phi_r =  r(u_r-v_r) + (u-v)$ we have the uniform estimate
\EQ{\label{B}
\sup_{t \in[t_-, 1)} \int_0^{\infty} \left[(\psi_t- \phi_t)^2(t, r) +  (\psi_r - \phi_r)^2(t, r)  \right] \, r\, dr \le B < \infty,
}
where again $t_-> (T_-( \vec u) , T_-( \vec v))$. 

We can now deduce Theorem~\ref{self sim} as a consequence of the following  proposition which is phrased in terms of $\psi:=r u$. 

\begin{prop}\label{prop2} Assume that there exists $\la \in (0, 1)$ and $t_0 \in [t_-, 1)$ so that for all $t \in [t_0, 1)$ we have 
\EQ{ \label{small psi}
 \sup_{\la(1-t) \le r \le 1-t} \abs{ \psi(t, r)} \le \frac{\sqrt{2}}{2}.
 }
 Then, 
 \EQ{ \label{psi en dec2}
\lim_{t \nearrow 1} \int_{\la(1-t)}^{1-t} \left[ \psi_t^2(t,r) + \psi_r^2(t, r) \right] \, r\, dr = 0.
}
\end{prop} 

 The size restriction~\eqref{small psi} will guarantee the positivity of $F( \psi) = \frac{1}{2} \psi^2[1- \psi^2/2]$ for $\la(1-t) \le r \le (1-t)$, $t  \in[t_0, 1)$. This positivity enters crucially in the methods introduced in \cite{CTZduke, CTZcpam, STZ92} as the $F$ term there is of the form $F = g^2$, and is always positive. Thus we do not prove Theorem~\ref{self sim} directly in terms of $\vec \psi$, as is done in, say \cite[Lemma $2.2$]{STZ92}, but rather deduce it as a consequence of Proposition~\ref{prop2}. Then, by assuming the  smallness assumption~\eqref{small psi} holds for a particular $\la \in (0, 1)$ we prove Proposition~\ref{prop2} using the methods in~\cite{STZ92}.

We momentarily postpone the proof of Proposition~\ref{prop2} and first use it to establish Theorem~\ref{self sim}. 
\begin{proof}[Proof that Proposition~\ref{prop2} implies Theorem~\ref{self sim}]{\ }

{\bfseries Step 1:} The main observation is that we can get rid of the $L^\infty$ assumption in Proposition~\ref{prop2} via an inductive argument, which is the content of the following:

\begin{claim}\label{prop1} Let $\vec u(t)$ be as in Theorem~\ref{self sim} and define $\vec \psi(t)$ as in~\eqref{psi def}. Then for every fixed $\la \in (0, 1)$ we have 
\EQ{ \label{psi en dec}
\lim_{t \nearrow 1} \int_{\la(1-t)}^{1-t} \left[ \psi_t^2(t,r) + \psi_r^2(t, r) \right] \, r\, dr = 0.
}
\end{claim}

\begin{proof}[Proof of the claim]
Consider the set $I \subset (0, 1)$ to be the collection of all $\la \in (0, 1)$ so that there exists $ t_0 = t_0( \la) \in [t_-, 1)$ such that
\ant{
\forall t \ge t_0, \ \forall r \in[\la (1-t), (1-t)], \quad
\abs{\psi(t, r)} \le \frac{\sqrt{2}}{2}. 
} 
Observe that if $\lambda \in I$, then $[\lambda,1) \subset I$.  Indeed, for such a $\la'$ one can take $t_0( \la') = t_0( \la')$. Also, by Proposition ~\ref{prop2}, then \eqref{psi en dec} holds this particular $\la$. Therefore, to prove the claim, it suffices to prove that $I$ contains a sequence $\la _n \to 0$. 
 
We begin by showing that $I$ is not empty. Fix $\la_0 \in (0, 1)$ to be determined below. Observe that since $\psi(t, 1-t) = \phi(t, 1-t)$ for $t \ge t_-$, we have for all $\la_0(1-t) \le r \le 1-t$ that 
 \EQ{ \label{psi-phi1}
\abs{\psi(t, r) - \phi(t, r)}  &= \abs{\int_r^{1-t}( \psi_r- \phi_r)(t, \rho) \, d\rho} \\
&\le \left( \int_r^{1-t} ( \psi_r- \phi_r)^2(t, \rho) \, \rho\, d\rho \right)^{\frac{1}{2}} \left( \int_r^{1-t} \rho^{-1} d\rho \right)^{\frac{1}{2}}\\
&\le  B(\log{\la_0^{-1}})^{\frac{1}{2}} \le C_0 B \abs{1- \la_0}^{\frac{1}{2}}
}
where the constant $B$ is fixed in~\eqref{B} and we have chosen $C_0>0$ so that for all $1/2 \le \la \le 1$ we have $\log( \la^{-1}) \le C_0^2 \abs{1- \la}$. Next, observe that by~\eqref{phi pw dec} we can find $t_0<1$ so that for all $t\ge t_0$ we have 
\EQ{
\abs{ \phi(t, r)} \le \frac{1}{3}, \, \, \forall \, 0 \le r \le 1-t.
} 
Hence for all $t \in [t_0, 1)$ we have 
\ant{
\sup_{ r \in [\la_0(1-t), \, 1-t]}\abs{\psi(t, r)} \le \frac{1}{3} + C_0 B \abs{1- \la_0}^{\frac{1}{2}}.
}
Choosing $\la_0$ close enough to $1$ so that $C_0 B \abs{1- \la_0}^{\frac{1}{2}}\le \frac{1}{3}$ we then guarantee that 
\ant{
\sup_{ r \in [\la_0(1-t), \, 1-t]}\abs{\psi(t, r)} \le \frac{2}{3} < \frac{ \sqrt{2}}{2}, \quad \forall t_0 \le t <1,
} 
which proves that $I$ is not empty, and  in fact $[\la_0, 1) \subset I$. 

Next, we need to prove that in fact $I= (0,1)$. Note that it will suffice to show that there exists a sequence $\la_n \to 0$ such that $\la_n \in I$ for all $n \in \N$. We define 
\EQ{
\la_n:= \la_0^n, \, \, \forall n \in \N.
}
Note that we have proved that $\la_1 = \la_0^1 \in I$. Now we argue by induction. Assume that $\la_n \in I$ for some $n \ge 1$ and fix this $\la_n$. We seek to prove that $\la_{n+1} \in I$.  We record a few additional consequences of our inductive hypothesis. Since $\la_n \in I$, Proposition~\ref{prop2} implies that 
\ant{
\lim_{t \nearrow 1} \int_{\la_n(1-t)}^{1-t} \left[ \psi_t^2(t,r) + \psi_r^2(t, r) \right] \, r\, dr = 0.
}
Using~\eqref{phi en dec} we in fact have that 
\ant{
\lim_{t \nearrow 1} \int_{\la_n(1-t)}^{1-t} \left( \psi_r- \phi_r \right)^2(t, r) \, r\, dr = 0.
}
Thus we can argue as in~\eqref{psi-phi} to deduce that there exists $0<t_n<1$  so that 
\EQ{
\abs{(\psi- \phi)(t, \la_n(1-t))} \le \sqrt{\log( \la_n^{-1})} \sqrt{ \int_{\la_n(1-t)}^{1-t} \left( \psi_r- \phi_r \right)^2(t, r) \, r\, dr} 
< \frac{1}{6} .
} 
for all $t_n \le t <1$. Using~\eqref{phi pw dec} we can also ensure that $t_n$ is large enough so that 
\EQ{
\abs{ \phi(t,r)} < \frac{1}{6}, \, \, \forall\,  0 \le r \le 1-t,
}
for all $t_n \le t <1$. Next for all $r \in [\la_{n+1}(1-t), \la_n(1-t)]$ we can argue as in~\eqref{psi-phi1} to bound the term 
\ant{
&\abs{(\psi(t, r) - \phi(t, r)) -( \psi(t, \la_n(1-t)) - \phi(t, \la_n(1-t))}  \le\\
& \quad\quad \le  \sqrt{ \int_{\la_{n+1}(1-t)}^{\la_n(1-t)}( \psi_r - \phi_r)^2(t, \rho)  \, \rho \, d \rho} \sqrt{\int_{\la_{n+1}(1-t)}^{\la_n(1-t)} \rho^{-1} \, d \rho}\\
& \quad \quad \le B (\log( \la_n/ \la_{n+1}))^{\frac{1}{2}} = B (\log( \la_0^{-1}))^{\frac{1}{2}}\\
&\quad \quad  \le \frac{1}{3}
}
where $B$  is as in ~\eqref{B} and since  $\la_n/ \la_{n+1} = \la_0^n/ \la_0^{n+1} = \la_0^{-1}$ and we have chosen $\la_0$ close enough to $1$ so that the last line above holds. 
Now for each $t \in [t_n, 1)$ and  $r \in [\la_{n+1}(1-t), \la_n(1-t)]$ write  
\ant{
\abs{\psi(t, r)}  &\le \abs{(\psi(t, r) - \phi(t, r)) -( \psi(t, \la_n(1-t)) - \phi(t, \la_n(1-t))} \\
&\quad+ \abs{ \phi(t, r)} + \abs{ \psi(t, \la_n(1-t)) - \phi(t, \la_n(1-t))}\\
& <\frac{1}{3} + \frac{1}{6}+ \frac{1}{6}  = \frac{2}{3} \le \frac{ \sqrt{2}}{2}.
}
As we also know that $\sup_{r \in [\la_n(1-t), 1-t]}\abs{\psi(t, r)} \le \frac{ \sqrt{2}}{2}$ for large enough $t<1$ by assumption, we have now proved that $\la_{n+1} \in I$ as well. Thus, by induction, $\la_n \in I$ for all $n$ and this completes the proof.
\end{proof}

{\bfseries Step 2:} We now transfer the result of the Claim \ref{prop1} to $\vec u$ and conclude the proof of Theorem \ref{self sim}. Since $\psi_r = r u_r+ u$ we see that 
\ant{
\int_{\la(1-t)}^{1-t}& \left[ u_t^2(t, r) + u_r^2(t, r) + \frac{u^2(t, r)}{r^2} \right] \, r^3 \, dr
 \\ &= \int_{\la(1-t)}^{1-t} \left[ \psi_t^2(t, r) + ( \psi_r(t, r) - u(t, r))^2(t, r) + u^2(t, r) \right] \, r \, dr\\
 & \le 2\int_{\la(1-t)}^{1-t} \left[ \psi_t^2(t, r) +  \psi_r^2(t, r)\right] \, r \, dr +3\int_{\la(1-t)}^{1-t}u^2(t, r) \, r \, dr.
 }
Hence it suffices to prove the vanishing of the Hardy term 
\EQ{ \label{hardy to 0}
\int_{\la(1-t)}^{1-t}u^2(t, r) \, r \, dr \to 0 \mas t \to 1.
}
To see this, we first note that~\eqref{psi en dec} together with~\eqref{phi en dec} imply that 
\EQ{
\lim_{t \nearrow 1} \int_{\la(1-t)}^{1-t} \left[ (\psi_t - \phi_t)^2(t,r) + (\psi_r - \phi_r)^2(t, r) \right] \, r\, dr = 0.
}
Next, note that~\eqref{supp a} implies that  $\psi(t, 1-t) = \phi(t, 1-t)$ for all $t \in [t_-, 1)$. From this we see that  that for every $r \in [\la(1-t),\, (1-t)]$ we have 
\EQ{ \label{psi-phi}
\abs{\psi(t, r) - \phi(t, r)}  &= \abs{\int_r^{1-t}( \psi_r- \phi_r)(t, \rho) \, d\rho} \\
&\le \left( \int_r^{1-t} ( \psi_r- \phi_r)^2(t, \rho) \, \rho\, d\rho \right)^{\frac{1}{2}} \left( \int_r^{1-t} \rho^{-1} d\rho \right)^{\frac{1}{2}}\\
&\le  (\log{\la^{-1}})^{\frac{1}{2}}\left( \int_{\la(1-t)}^{1-t} ( \psi_r- \phi_r)^2(t, \rho) \, \rho\, d\rho \right)^{\frac{1}{2}}.
}
Using~\eqref{psi en dec}  we can then conclude that 
\EQ{
  \sup_{ r \in[\la(1-t), (1-t)]}\abs{\psi(t, r) - \phi(t, r)} \to 0 \mas t \to 1.
 }
 Then by the definitions~\eqref{psi def},~\eqref{phi def} we have
 \EQ{
  \sup_{ r \in[\la(1-t), (1-t)]} r \, \abs{  u(t, r) - v(t, r)}   \to 0 \mas t \to 1.
  }
As a direct consequence we obtain, 
\EQ{ \label{u-v dec}
  \int_{\la(1-t)}^{1-t} \left[ u(t, r) - v(t,r)\right]^2 \, r \, dr \to 0 \mas t \to 1.
  }
  Combining~\eqref{u-v dec} with~\eqref{v loc en} we obtain~\eqref{hardy to 0}, which finishes the proof of Theorem \ref{self sim}. 
  \end{proof}

\subsection{Proof of Proposition~\ref{prop2}}

We have thus reduced the matter of proving  Theorem~\ref{self sim} to proving Proposition~\ref{prop2}. This will follow from the techniques introduced by Christodoulou, Tahvildar-Zadeh, \cite{CTZduke, CTZcpam}, and Shatah, Tahvildar-Zadeh, \cite{STZ92, STZ94}. 

Recall that $\vec \psi(t)$ satisfies the wave maps type equation~\eqref{psi eq} except that $f \neq g g'$. By translating in time, we can, without loss of generality assume that $T_+( \vec u) = 0$ so that $T_+( \vec \psi) = 0$ in order to adjust to the notation used in Shatah, Tahvildar-Zadeh \cite[Lemma $2.2$]{STZ92}. 

The conserved energy for~\eqref{psi eq} is given by 
\ant{
\E( \vec \psi) = \int_0^{\infty}\left( \frac{1}{2}[ \psi_t^2(t, r) + \psi_r^2(t, r) ]+ \frac{F( \psi)}{r^2}\right) \, r \, dr,
}
where $F( \psi) =  \frac{1}{2} \psi^2[ 1- \psi^2/2]$ as defined in~\eqref{f F def}. 
After translating in time so that $T_+(\vec \psi) = 0$ we see that the hypothesis of Proposition~\ref{prop2} give us a $\la_0 \in (0, 1)$ and a $t_0<0$ so that for all $t>t_0$, $t<0$, we have 
\EQ{\label{psi pw small}
\sup_{ \la \abs{t} \le r \le \abs{t}} \abs{ \psi(t, r)} \le \frac{\sqrt{2}}{2}.
} 
Note also that 
\EQ{\label{F>0}
\abs{\psi(t, r)} \le \frac{\sqrt{2}}{2} \Longrightarrow F( \psi(t,r)) \ge 0.
} 
This leads us to reduce the proof of Proposition~\ref{prop2} to the following lemma: 
\begin{lem} \label{l1}Let $\la \in (0, 1)$ be given as in Proposition~\ref{prop2} so that~\eqref{psi pw small} and~\eqref{F>0} holds. Then 
\EQ{\label{Ela}
\E_{\ext}^{\la}(t) := \int_{\la\abs{t}}^{\abs{t}}\left( \frac{1}{2}[ \psi_t^2(t, r) + \psi_r^2(t, r) ]+ \frac{F( \psi)}{r^2}\right) \, r \, dr \to 0 \mas t \nearrow 0.
}
\end{lem}
We remark that Proposition~\ref{prop2} is an immediate consequence of Lemma~\ref{l1} since~\eqref{F>0} implies that $F( \psi) \ge 0$ in the domain of integration in~\eqref{Ela}. To prove Lemma~\ref{l1} we will need a few multiplier identities 
\begin{align} \label{en id}
 &\p_t \left(  \frac{r}{2} \psi_t^2 + \frac{r}{2} \psi_r^2 + \frac{F(\psi)}{r} \right) - \p_{r}(r \psi_t, \psi_r) = 0, \\
& \p_t\left( r^2 \psi_t \psi_r\right) - \p_r\left( \frac{r^2}{2} \psi_t^2 + \frac{r^2}{2} \psi_r^2 - F( \psi)\right) + r \psi_t^2 = 0, \label{psi t id}
\end{align}
which are obtained by multiplying~\eqref{psi eq} by $\psi_t$ and $r t \psi_t$ respectively. We denote the truncated backwards light-cone emanating from $(t, r) = (0, 0)$ and its mantel by 
\begin{align}
&K( \tau, \e):= \{ (t, r) \mid  \tau \le t \le \e <0, \, \, 0 \le r \le  \abs{\tau}\},\\
&C( \tau, \e):=\{ (t, r) \mid  \tau \le t \le \e <0, \, \,  r =  \abs{\tau}\}.
\end{align}
For $\tau<0$ and $\e<0$ small with $\tau< \e< 0$, we also define the local energy and the flux:
\EQ{
&\E( \tau):=  \int_{0}^{\abs{\tau}}\left( \frac{1}{2}[ \psi_t^2(\tau, r) + \psi_r^2(\tau, r) ]+ \frac{F( \psi( \tau, r))}{r^2}\right) \, r \, dr\\
& \flux( \tau, \e):= -c_0 \int_{ \tau}^{ \e} \left[ \frac{1}{2}(  \chi'( \ell))^2 + \frac{F( \chi( \ell))}{\ell}\right] \,  \ell \, d \ell
}
where $\chi( \ell):= \psi(\ell, -\ell)$,  $\tau< \ell< \e<0$, and $c_0>0$ is a universal dimensional constant so that the following local energy identity holds: 
\EQ{\label{loc en id}
\E( \tau) = \E(\e)+ \flux( \tau, \e).
}
Note that although we don't know that $F( \psi) \ge 0$ on the entire domain of integration in $\E( \tau)$ and $\E( \e)$ above, the hypothesis of Proposition~\ref{prop2} guarantee that $F( \chi) \ge 0$, for $\tau$ small  in the Flux term and hence $\flux( \tau, \e) \ge 0$ since $\tau< \e<0$. From~\eqref{loc en id} we can then deduce that $\E( \tau) \ge \E( \e)$ for $ \tau< \e <0$. 

Next, since $\abs{ \E( \tau)} \le A <  \infty$ and since $\E( \tau)$ is monotonically decreasing as $\tau \nearrow 0$, we observe that 
\ant{ 
\lim_{\e \nearrow 0} \E( \e) =: \E( 0)
} 
exists and is finite. Using~\eqref{loc en id} again we see that 
\ant{
\flux(\tau): =  \lim_{\e \to 0} \flux( \tau, \e) \le \E(\tau) - \E(0)
} 
exists by monotone convergence and that $0 \le\flux( \tau)< \infty$ as well as  $\flux( \tau) \to 0$ as $\tau \to 0$. We can now replicate the argument in~\cite[Lemma $2.2$]{STZ92} which we include here for completeness and to show where precisely we will use the hypothesis in Proposition~\ref{prop2}. We also refer the reader to the book~\cite[Proof of Lemma~$8.2$]{SSbook}. 

Define, 
\EQ{ \label{e m L def}
&e(t, r):=  \frac{1}{2} \psi_t^2(t, r) + \frac{1}{2} \psi_r^2(t, r) + \frac{1}{r^2}F( \psi(t, r)),\\
&m(t, r):= \psi_t(t, r) \psi_r(t, r),\\
&L(t, r):=  - \frac{1}{2} \psi_t^2(t, r) + \frac{1}{2} \psi_r^2(t, r) + \frac{1}{r^2}F( \psi(t, r)) - \frac{2}{r} f( \psi(t, r)) \psi_r(t, r).
}
Then, using~\eqref{en id},~\eqref{psi t id} we see that 
\EQ{\label{re rm}
\p_t( r e) - \p_r(r m) = 0,\\
\p_t(rm) - \p_r(r e) = L.
}
We also introduce null coordinates 
\ant{
\eta = t+r, \, \, \xi = t-r
}
as well as the functions
\EQ{
&\q A^2( \eta, \xi):= r(e+m) =  \frac{r}{2} (\partial_t \psi + \partial_r \psi)^2  + \frac{F(\psi)}{r}, \\
&\q B^2(\eta, \xi):= r(e- m) =  \frac{r}{2}(\partial_t \psi - \partial_r \psi)^2  + \frac{F(\psi)}{r}. 
}
Note that the assumptions of Proposition~\ref{prop2} ensure the positivity of $F( \psi(t, r))$ in the region 
\EQ{ \label{K ext}
K_{\ext}^{\la}:= \{(t, r) \mid t_0<t<0, \, \, \la \abs{t} \le r \le \abs{t}  \},
}
and thus the interpretation of the functions $\A^2, \B^2$ as squares in $K^{\la}_{\ext}$ is justified. We can rewrite~\eqref{re rm} in terms of $\A^2, \B^2$ as 
\EQ{\label{d A B}
&\p_{\xi} \A^2 =L,\\
&\p_{\eta} \B^2 =- L.
}
We next claim the bound 
\begin{claim} On $K^\la_{\mr{ext}}$,
\EQ{ \label{L2}
L^2 \le C \frac{\A^2 \B^2}{r^2}.
}
\end{claim}

\begin{proof}
Indeed, a direct computation and simple algebra yields 
\EQ{\label{L2 1}
L^2 \le  \frac{1}{2}( \psi_r^2 - \psi_t^2)^2  + \frac{4}{r^4} F^2( \psi) + \frac{16}{r^2} f^2( \psi) \psi_r^2.
}
Next we note that the assumptions of Proposition~\ref{prop2} imply that for all $(t, r)\in K_{\ext}^\la$ we have 
\ant{
\abs{\psi(t, r)}^2 \le  \frac{1}{2}. 
}
It follows then that 
\ant{
&\abs{f ( \psi)} =  \abs{\psi( 1- \psi^2)} \le \abs{\psi},\\
&\abs{F(\psi)}  = \abs{\frac{\psi^2}{2}\left(1-\frac{ \psi^2}{2}\right)} \ge \frac{\psi^2}{4}.
}
Combining the above inequalities gives 
\EQ{\label{f2 F}
f^2 (\psi) \le \abs{\psi}^2 \le 4 F(\psi), \quad \forall \,\,(t, r) \in K_{\ext}^\la.
}
Plugging~\eqref{f2 F} into~\eqref{L2 1} we obtain
\EQ{ \label{L2 bu}
L^2 &\le \frac{1}{2}( \psi_r^2- \psi_t^2)^2 +  \frac{4}{r^4} F^2( \psi) + \frac{64}{r^2} F( \psi) \psi_r^2, \\
& \le C \left[ \frac{1}{4}( \psi_r^2- \psi_t^2)^2 +  \frac{1}{r^4} F^2( \psi) + \frac{2}{r^2} F( \psi) (\psi_r^2+ \psi_t^2)\right].
}
On the other hand, 
\ant{
\frac{ \A^2 \B^2}{r^2} = (e+m)(e-m) = \frac{1}{4}( \psi_r^2- \psi_t^2)^2 +  \frac{1}{r^4} F^2( \psi) + \frac{2}{r^2} F( \psi) (\psi_r^2+ \psi_t^2),
} 
which, together with the preceding inequality, establishes~\eqref{L2}. 
\end{proof}
Now, we can combine~\eqref{d A B} with~\eqref{L2} to see that 
\EQ{\label{A B in}
\abs{\p_{\xi} \A} \le \frac{C}{r} \B, \quad \abs{ \p_{\eta} \B} \le \frac{C}{r} \A.
}
Now consider the rectangle 
\ant{
\Ga:=\Ga( \eta, \xi):= [ \eta, 0] \times[ \xi_0, \xi] \subset K^{\la}_{\ext}.
}
Integrating on $\Ga$ and using~\eqref{A B in} we have the inequality
\EQ{
\A( \eta, \xi) \le \A( \eta, \xi_0) + C \int_{\xi_0}^{\xi} \frac{ \B(0, \xi')}{ \eta- \xi'} \, d \xi' + C^2 \int_{\xi_0}^{\xi} \int_{\eta}^0 \frac{ \A(\eta', \xi')}{ (\eta- \xi')( \eta'- \xi')} \, d \eta' \, d \xi'.
}
We estimate the $2$nd term on the right as follows 
\begin{align*}
\MoveEqLeft \int_{\xi_0}^{\xi} \frac{ \B(0, \xi')}{ \eta- \xi'} \, d \xi' \le \left( \int_{\xi_0}^{\xi}  \B^2(0, \xi') \, d \xi' \right)^{\frac{1}{2}}\left(\int_{\xi_0}^{\xi} \frac{ 1}{ (\eta- \xi')^2} \, d \xi' \right)^{\frac{1}{2}}\\
& =  \left( \int_{\xi_0}^{\xi}  \left[ \psi_{\xi}^2(0, \xi') + \frac{F( \psi(0, \xi')}{ ( \xi')^2} \right](- \xi')\, d \xi' \right)^{\frac{1}{2}} \left( \frac{1}{ \eta- \xi} - \frac{1}{ \eta- \xi_0} \right)^{\frac{1}{2}}\\
& \le C \sqrt{\frac{\flux( \xi_0)}{\eta- \xi}}.
\end{align*}
Now, define 
\ant{
h(\eta, \xi):= \sup_{\eta \le \eta' \le 0} \ \sqrt{ \eta'- \xi} \A( \eta', \xi).
}
We then have 
\ant{
\A( \eta, \xi) &\le \A( \eta, \xi_0) + C \sqrt{\frac{\flux( \xi_0)}{\eta- \xi}} + C^2 \int_{\xi_0}^{\xi} \int_{\eta}^0 \frac{ h(\eta, \xi')}{ (\eta- \xi')( \eta'- \xi')^{\frac{3}{2}}} \, d \eta' \, d \xi'\\
& \le \A( \eta, \xi_0) +\sqrt{\frac{\flux( \xi_0)}{\eta- \xi}} + C^2\int_{\xi_0}^{\xi}  \frac{ h(\eta, \xi')}{ (\eta- \xi')} \left( \frac{1}{ \sqrt{ \eta- \xi'}}- \frac{1}{ \sqrt{- \xi'}} \right) \,  d \xi'  .
}
Using the above and the fact that  $\Ga \subset K^\la_{\ext}$ we then obtain, 
\ant{
h( \eta, \xi) \le \frac{ \sqrt{- \xi}}{ \sqrt{- \xi_0}} h( \eta, \xi_0) + C \sqrt{\flux( \xi_0)} + C \int_{\xi_0}^{\xi} h( \eta, \xi')  \frac{\eta}{ \xi'( \eta- \xi')} \, d \xi'.
}
Next, define $\la':= (1- \la)/(1+ \la)<1$. Fix $\eta$ and consider $\xi \in [ \xi_0, \eta/ \la']$. Applying the integral form of Gronwall's inequality gives 
\begin{align*}
\MoveEqLeft h( \eta, \xi) \le \frac{ \sqrt{- \xi}}{ \sqrt{- \xi_0}} h( \eta, \xi_0) + C \sqrt{\flux( \xi_0)} \\
& + C \int_{\xi_0}^{\xi}  \left[ \frac{ \sqrt{- \xi'}}{ \sqrt{- \xi_0}} h( \eta, \xi_0) + C \flux^{\frac{1}{2}}( \xi_0)\right]h( \eta, \xi')  \frac{\eta}{ \xi'( \eta- \xi')} \,  e^{\left( \int_{\xi'}^{\xi} \frac{\eta}{ \xi''( \eta- \xi'')} \, d \xi'' \right)}d \xi'.
\end{align*}
Setting $\eta = \la' \xi$ with $ \xi_0 < \xi'< \xi$ we have 
\ant{
\int_{\xi'}^{\xi} \frac{\eta}{\xi''( \eta- \xi'')} \, d \xi'' =  \log\left( \frac{\xi( \la' \xi- \xi')}{ \xi'( \la' \xi- \xi)} \right) \le \log( \frac{1}{1- \la'}).
}
Note that since $\psi$ is regular away from $(0, 0)$, $\A$ is bounded at $(\eta',  \xi_0)$ for $\eta \le \eta' \le 0$  by a constant that depends on $\xi_0$ and thus 
\ant{
h( \eta, \xi) \le \sup_{\eta \le \eta' \le 0} \sqrt{ \eta'- \xi_0} \sup_{\eta \le \eta' \le 0} \A( \eta', \xi_0) \le C( \xi_0) \sqrt{- \xi_0}.
}
Let $\e>0$ be given. Fix $\xi_0$ small enough so that $C \sqrt{\flux( \xi_0)} \le \e$. Then, 
\ant{
h( \la' \xi, \xi) &\le C( \xi_0) \sqrt{- \xi} + \e + C( \xi_0) \int_{\xi_0}^\xi \frac{\la' \xi}{ \sqrt{- \xi_0}( \la' \xi- \xi')} \, d \xi' + C \e\\
& \le C\e + C( \xi_0) \sqrt{- \xi}\\
& \le 2C \e
}
provided $\xi$ is small enough. Therefore, 
\ant{
\A( \eta, \xi)  \le \frac{h( \eta, \xi)}{ \sqrt{ \eta- \xi}} \le \frac{C \e}{ \sqrt{ \eta- \xi}}
} 
for $( \eta, \xi)$ small inside $K_{\ext}^\la$. This means that 
\EQ{\label{int A}
\int^{0}_\eta \A^2( \eta', \xi)\,  d \eta' \le C \e^2 \int_{\eta}^0 \frac{d \eta'}{ \eta'- \xi} = C \e^2 \log\left( \frac{- \xi}{( \la'-1) \xi}\right) \le C \e^2.
}
With the above in hand, we can now conclude by integrating~\eqref{en id} over the triangle with vertices $( \eta, \xi), \, (0, \xi), \, (0, \eta + \xi)$ and $\eta= \la' \xi$. We obtain
\EQ{
0 &= - \int_{\la \abs{t}}^{\abs{t}} e(t, r)  \, r \, dr - \int_{\eta}^0 r \, (e+m) \, d \eta' + \int_{\eta + \xi}^{\xi}  r(e-m) \, d \xi' \\
 &= \textrm{I} + \textrm{II} +\textrm{III}.
 } 
 We note that III is the $\flux$ which tends to $0$ as $\abs{t} \to 0$ and II is exactly~\eqref{int A} which tends to $0$ as we have just shown. This means that I also tends to $0$ which proves Proposition~\ref{prop2}.


\section{Self-similar and exterior regions: global solutions}\label{self sim global}

In this section, we consider a global type II solution $\vec u(t)$: we assume that $[0,+\infty) \subset I$ and that for some $M \ge 0$,
\EQ{\label{typeII}
 \sup_{t \in [0,+\infty)} \| \vec u(t) \|_{\HH} \le M. 
 }
We identify the radiation term $\vec v_{L}(t)$  and establish the analog of Theorem~\ref{self sim} for $\vec u(t)$.

\subsection{Extraction of the radiation term} \label{rad term}

We begin by extracting the radiation term, that is, we find the unique solution $\vec v_{L}(t)$ to the linear equation~\eqref{free wave} which $\vec u(t)$ approaches outside the forward light cone. This is a somewhat more involved procedure than in the finite time blow-up case where taking a weak limit suffices. In particular, we prove:

\begin{prop}~\label{th:u-v_L}
There exists $(v_0,v_1) \in \HH$ such that $\| (v_0,v_1) \|_{\HH} \le M$ and
\[ \forall R \in \R, \quad \int_{|x| \ge t-R} | \nabla_{t,x} (u - v_{\mr{L}})(t,x)|^2 dx \to 0 \quad \text{as} \quad t \to +\infty, \]
where $v_{ L}$ is the free wave, i.e., solution to the linear equation \eqref{free wave}, with initial data $(v_0,v_1)$.
\end{prop}

The rest of this subsection is devoted to the proof of this result, which follows closely the proof  of the corresponding result in~\cite{DKM3}. For $\delta >0$, let $\varphi_\delta: \R^4 \to  \R$ be a radial smooth function , such that 
\[ 0 \le \varphi_{\delta} \le 1, \quad | \nabla \varphi_\delta| \le \frac{C}{\delta}, \quad \varphi(x) = \begin{cases}
1 & \text{for } |x| \ge 1-\delta \\
0 & \text{for } |x| \le 1 - 2\delta
\end{cases}. \]

\begin{lem} \label{lem:1}
Let $\e>0$ be given. Then there exists $t_n \uparrow +\infty$, $\delta >0$ small such that
\[ \varphi_\delta \left( \frac{x}{t_n} \right) \vec u(t_n,x) = \varphi_{\delta} \left( \frac{x}{t_n} \right) \left( u(t_n,x), \partial_t u(t_n,x) \right) \]
has a profile decomposition with profiles $(U_{\mr L}^j)$ and parameters $(\lambda_{j,n}, t_{j,n})$ such that 
\[ \| (U_0^1, U_1^1) \|_{\HH} \le \e \text{ and } \forall n, \ t_{1,n} =0, \quad \text{and} \quad  \forall j \ge 2, \ - \frac{t_{j,n}}{\lambda_{j,n}} \to +\infty. \]
\end{lem}

\begin{proof}
\textbf{Step 1:} We claim that there exists $\delta_1 >0$, $s_n \to +\infty$ such that $$\ds \left\{ \varphi_{\delta_1} ( x /s_n ) \vec u(s_n) \right\}$$ has a profile decomposition with profiles $\{\vec V_{\mr L}^j\}_j$ and parameters $\{\mu_{j,n}, s_{j,n}\}$ such that
\begin{align*}
\forall j \ge 2, \quad & \lim_{n \to +\infty} - \frac{s_{j,n}}{\mu_{j,n}} \in \{ \pm \infty \} \quad  \\
\text{and} \quad & \lim_{n \to \infty} - \frac{s_{j,n}}{s_n} \in [-1, 2\delta_1-1] \cup [1-2\delta_1,1], \\
\forall n \ge 1, \quad & s_{1,n} =0, \quad \text{and} \quad \| (V_0^1, V_1^1 ) \|_{\HH} \le \e/2.
\end{align*}
In fact, note that finite speed of propagation and small data theory imply
\begin{equation} \label{dagger}
\limsup_{t \to +\infty} \int_{|x| \ge t + R} | \nabla_{t,x} u(t,x)|^2 dx \to 0 \quad \text{as} \quad R \to +\infty.
\end{equation}
Indeed, let $\eta >0$ be given. Choose $R_0$ large enough such that 
\[  \int_{|x| \ge R_0} | \nabla u_0|^2 + u_1^2 \le \eta^2. \]
Let
\[ \tilde u_{0,R_0}(x) = \begin{cases}
u_0(R_0) & \text{if } |x| \le R_0 \\
u_0(x) & \text{if } |x| \ge R_0
\end{cases}, \qquad 
\tilde u_{1,R_0}(x) = \begin{cases}
0 & \text{if } |x| \le R_0 \\
u_1(x) & \text{if } |x| \ge R_0
\end{cases}. \]
Then
\[ \| (\tilde u_{0,R_0}, \tilde u_{1,R_0}) \|_{\HH} \le \eta. \]
If $\eta$ is small, by small data theory, the solution $\tilde u_{R_0}$ to \eqref{u eq} with data $(\tilde u_{0,R_0}, \tilde u_{1,R_0})$ exists for all time and 
\[ \sup_{t \ge 0} \| (\tilde u_{R_0}, \partial_t \tilde u_{R_0})(t) \|_{\HH} \le C \eta. \]
By finite speed of propagation, for $|x| \ge R+t$, $\tilde u_{R_0}(t,x) = u(t,x)$. The claim \eqref{dagger} follows.

Let $s_n \to \infty$, then $ \vec u(s_n)$ has a decomposition with profiles $\{\tilde V_{\mr L}^j\}_j$ (with initial data $(v_{0,j}, v_{1,j})$), parameters $\{\mu_{j,n}, s_{j,n}\}_n$, and remainder $(\tilde w_{0,n}^J,\tilde w_{1,n}^J)$. As usual for the profile decomposition, we denote
\[ \tilde V_{\mr L,n}^j(t,x) := \frac{1}{\mu_{j,n}} \tilde V_{\mr L}^j \left( \frac{t-s_{j,n}}{\mu_{j,n}}, \frac{x}{\mu_{j,n}} \right). \]
As we can always extract subsequences without loss of generality, we will systematically assume that all real valued sequences converge (in $\overline{\R}$). We next recall that we can assume 
\begin{equation} \label{sjn}
\text{either} \quad \lim_{n \to +\infty} - \frac{s_{j,n}}{\mu_{j,n}} = \pm \infty, \quad \text{or} \quad \forall n, \ s_{j,n} =0.
\end{equation}
Define
\[ \tau_j := \lim_{n \to +\infty} - \frac{s_{j,n}}{s_n}. \]
\begin{claim} \label{mu_tau_bounds} For all $j$, 
$|\tau_j| \le 1$, and $\ds \lim_{n \to +\infty} \frac{\mu_{j,n}}{s_n} =0$, except for at most one $j$, for which the limit is finite.
\end{claim}

\begin{proof}
For this, consider $v_{0,n} = u(s_n)$, $v_{1, n} = \partial_t u(s_n)$, $\mu_n = s_n$. By \eqref{dagger} 
\[ \lim_{R \to +\infty} \limsup_{n \to \infty} \int_{|x| \ge R s_n} |\nabla v_{0,n(x)}|^2 + |v_{1,n}(x)|^2 dx =0. \]
Hence we can apply Lemma \ref{lem:2.5}, and deduce that for all $j$, $\ds \lim_{n \to +\infty} \frac{\mu_{j,n}}{s_n} < +\infty$, and for all $j$ except at most one, the limit is 0. Moreover, $|\tau_j| < \infty$.

In the second case of \eqref{sjn}, $\tau_i =0$, and we are done. Now consider the first case of \eqref{sjn}. Assume $|\tau_j| = 1 +\eta$, $\eta>0$. Note first that
\[ \limsup_{n \to \infty} \int_{|x| \ge s_n + R} \left| \nabla_{t,x} \tilde V_{\mr L,n}^j (0,x) \right|^2 dx \to 0 \quad \text{as} \quad R \to +\infty.  \]
This follows from the Pythagorean expansion with cutoffs, i.e.,  Proposition~\ref{loc en dec}. We combine this with Lemma \ref{lem:scaled_lin_disp}: let $\e>0$, there exists $R$ and $N_0$ such that
\[ \forall n \ge N_0, \quad \int_{||x| - |s_{j,n}|| \ge R \mu_{j,n}} \left| \nabla_{t,x} \tilde V_{\mr L,n}^j (0,x) \right|^2 \le \e. \]
We note that for $n$ large and $\tilde R$ large,
\begin{equation} \label{subset_1}
\{ x \mid  ||x| - |s_{j,n}|| \le R \mu_{j,n} \} \subset \{ x \mid |x| \ge s_n + \tilde R \}.
\end{equation}
Indeed, if  $||x| - |s_{j,n}|| \le R \mu_{j,n}$, then $|x| \ge |s_{j,n}| -  R \mu_{j,n}$. But since $\ds - \frac{s_{j,n}}{\mu_{j,n}} \to +\pm \infty$, for any $\delta >0$ small, if $n \ge N_1$ is large enough $\mu_{j,n} \le \delta |s_{j,n}|$ so that
\[ |x| \ge (1-R\delta) |s_{j,n}| \ge (1+\delta_0/2)(1-R\delta) s_n. \]
Fix $\delta$ small enough so that $(1+\delta_0/2)(1-R\delta) >1$; thus, since $s_n \to \infty$, our claim \eqref{subset_1} follows. But then
\[ \forall n \ge N_0, \quad \int \left| \nabla_{t,x} \tilde V_{\mr L,n}^j (0,x) \right|^2 dx \le 2\e. \]
By invariance of the linear energy, $\tilde V_{\mr L}^j =0$, which is  a contradiction. Hence $|\tau_j| \le 1$, and this establishes our claim.
\end{proof}

Next, note that if $j$ is such that $\ds \lim_{n \to +\infty} \frac{\mu_{j,n}}{s_n} >0$ (and finite by Claim \ref{mu_tau_bounds}), we cannot have $\ds \lim_{n \to +\infty} \frac{|s_{j,n}|}{\mu_{j,n}} = +\infty$, hence $s_{j,n} =0$ for all $n$. This happens for at most one $j$, by Claim \ref{mu_tau_bounds}. We assume this is $j=1$, and we can also assume $\mu_{1,n} =s_n$. Now we claim that

\begin{equation}\label{Supp_V1}
\supp (\tilde V_0^1, \tilde V_1^1) \subset \{ x \mid |x| \le 1 \}.
\end{equation}
Indeed, take $\vec \theta \in \q D(\R^4)$ such that $\supp(\vec \theta) \subset \{ x \mid  |x| > 1 + \eta \}$. Then by \eqref{dagger},
\[ \int \vec \theta \left( \frac{x}{s_n} \right). \frac{1}{s_n^{d/2}} \nabla_{t,x} u(s_n,x) dx \to 0 \quad \text{as} \quad n \to +\infty. \]
But since $\ds \lim_{n \to +\infty} \frac{\mu_{j,n}}{s_n} =0$ for $j \ge 2$, by the profile decomposition and the weak convergence to 0 of the rescaled $w_n^J$, this gives
\[ \lim_{n \to +\infty} \int \frac{1}{s_n^{d/2}} \vec \theta \left( \frac{x}{s_n} \right) . \frac{1}{s_n^{d/2}} \left( \nabla_{x} \tilde V_0^1 \left(\frac{x}{s_n} \right), \tilde V_1^1 \left( \frac{x}{s_n} \right) \right) dx =0, \]
i.e.,  $\ds \int \vec \theta . (\nabla \tilde V_0^1, \tilde V_1^1) =0$ and \eqref{Supp_V1} follows. Then we define the first profile
\begin{align} \label{def:V1}
(V_0^1, V_1^1)(y) := \varphi_{\delta'}(y). ( \tilde V_0^1(y), \tilde V_1^1(y)), \\
\quad \text{with parameters} \quad \mu_{1,n} = s_n, \quad s_{1,n} =0. \nonumber
\end{align}
For $\delta'$ small, \eqref{Supp_V1} shows that $\| (V_0^1, V_1^1) \|_{\HH} < \e/2$.

Let now $j \ge 2$. Then $\ds \lim_{n \to +\infty} \frac{\mu_{j,n}}{s_n} =0$ (recall $j=1$ is the only one for which this is possibly not true by Claim \ref{mu_tau_bounds}). We distinguish two cases according to \eqref{sjn}. Let 
\begin{align*}
\mathcal{J}_1 & := \{ j \in \{ 2, \dots,  J \} \mid \forall n, \ s_{j,n} =0 \} \quad \text{and} \\
\mathcal{J}_2 & := \{ j \in \{ 2, \dots, J \} \mid \ \lim_{n \to +\infty} - \frac{s_{j,n}}{\mu_{j,n}} = \pm \infty \}.
\end{align*}
If $j \in \q J_1$, using $\mu_{j,n}/s_n \to 0$ and Lemma \ref{lem:scaled_lin_disp}, we see that:
\[ \left\| \nabla_{x,t} \tilde V^j_{\mr L,n}(0) \right\|_{L^2(s_n/2 \le |x| \le s_n)} \to 0 \quad \text{as} \quad n \to + \infty. \]
Next note that $\ds |\nabla \varphi_{\delta'} (y)| \le C_{\delta'} \frac{\varphi(y)}{|y|}$, from where we deduce, due to Hardy's inequality
\begin{align*}
 \left\| \nabla_x \left( \varphi_{\delta'} \left( \frac{x}{s_n} \right) \right) \tilde V^j_{\mr L,n}(0) \right\|_{L^2}  
& \le C_{\delta'} \left\| \varphi_{\delta'} \left( \frac{x}{s_n} \right) \frac{1}{|x|} \tilde V^j_{\mr L,n}(0) \right\|_{L^2} \\
& \le C_{\delta'}  \| \varphi_{\delta'} \|_{L^2} \left \| \frac{\tilde V^j_{\mr L,n}(0,x)}{|x|} \right\|_{L^2(s_n/2 \le |x| \le s_n)} \\
 &  \le C_{\delta'} \left \| \nabla_x \tilde V^j_{\mr L,n}(0,x) \right\|_{L^2(s_n/2 \le |x| \le s_n)} \to 0.
\end{align*}
Combining these two limits, we get
\begin{equation} \label{VJ1}
\forall j \in \q J_1, \quad \lim_{n \to +\infty} \left\| \nabla_{x,t} \left( \varphi_{\delta'} \left( \frac{x}{s_n} \right)  \tilde V^j_{\mr L,n}(0,x) \right) \right\|_{L^2} =0.
\end{equation}

If $j \in \q J_2$ (recall $\ds \tau_j := \lim_{n \to +\infty} - \frac{s_{j,n}}{s_n} \in [-1,1]$), we claim that
\begin{equation} \label{Vj}
\lim_{n \to \infty} \left\| \nabla_{x,t} \left( \varphi_{\delta'} \left( \frac{x}{s_n} \right) \tilde V^j_{\mr L,n} (0,x) \right) - \varphi_{\delta'}(|\tau_j|) \nabla_{x,t} \tilde V^j_{\mr L,n} \left( 0,x  \right) \right\|_{L^2} =0.
\end{equation}
Keeping in mind that $\ds \varphi_{\delta'}(|\tau_j|) - \varphi_{\delta'} \left( - \frac{s_{j,n}}{s_n} \right)$ is small for $n$ large, we rewrite
\begin{multline*}
 \nabla_{x,t} \left( \varphi_{\delta'} \left( \frac{x}{s_n} \right) \tilde V^j_{\mr L,n} (0,x) \right) - \varphi_{\delta'}(|\tau_j|) \nabla_{x,t} \tilde V^j_{\mr L,n} \left( 0,x  \right) \\
  = \frac{1}{s_n} \nabla_x \varphi \left( \frac{x}{s_n} \right) \tilde V^j_{\mr L,n} (0,x) 
 + \left( \varphi_{\delta'} \left( \frac{x}{s_n} \right)  - \varphi \left( - \frac{s_{j,n}}{s_n} \right) \right) \nabla_{x,t} \tilde V^j_{\mr L,n} (0,x)  \\ + \left( \varphi \left( - \frac{s_{j,n}}{s_n} \right) - \varphi \left( |\tau_j| \right) \right) \nabla_{x,t} \tilde V^j_{\mr L,n} (0,x).
\end{multline*}
We will show that the $L^2$-norm of all three terms tends to 0. For the first two terms, we use Lemma \ref{lem:scaled_lin_disp}:
\[ \limsup_{n \to \infty} \int_{||x| - |s_{j,n}|| \ge R \mu_{j,n}} \left| \nabla_{t,x} \tilde V_{\mr L,n}^j \left(0,x \right) \right|^2 dx \to 0 \quad \text{as} \quad R \to +\infty. \]
Now, for the first term: we use H\"older's inequality, the Sobolev embedding \\$\dot H^1(\R^4)~\hookrightarrow~L^4(\R^4)$ and conservation of the linear energy for $V^j_{\mr L}$ to compute
\begin{align} 
\MoveEqLeft \int_{||x| - |s_{j,n}|| \le R \mu_{j,n}} \frac{1}{s_n^2} | \nabla_x \varphi(x/s_n) |^2 |\tilde V^j_{\mr L,n} (0,x)|^2 dx  \nonumber \\
& \le \frac{\| \nabla \varphi \|_{L^\infty}^2}{s_n^2} \mu( \{ ||x| - |s_{j,n}|| \le R \mu_{j,n} \})^{\frac{1}{2}}  \| V^j_{\mr L,n}(0) \|_{L^{4}}^2 \nonumber \\
& \le \frac{R^2 \mu_{j,n}^2}{s_n^2} \| \nabla \varphi \|_{L^\infty}^2 \| \nabla_{x,t} V^j_{\mr L}(0) \|_{L^2}^2 \to 0, \label{phixVj}
\end{align}
for all $R \in \R$, because $\mu_{j,n}/s_n \to 0$. For the second term, if $||x|| - |s_{j,n}|| \le R \mu_{j,n}$, then $\ds \left| \frac{|x|}{s_n} - \frac{|s_{j,n}|}{s_n} \right| \le R \frac{\mu_{j,n}}{s_n}$. As $\ds \frac{\mu_{j,n}}{s_n} \to 0$, we see that $\ds \varphi_{\delta'} \left( \frac{x}{s_n} \right) - \varphi_{\delta'} \left( - \frac{s_{j,n}}{s_n} \right) \to 0$ (uniformly on the interval), and
\[  \int_{||x| - |s_{j,n}|| \le R \mu_{j,n}}  \left| \varphi_{\delta'} \left( \frac{x}{s_n} \right) - \varphi_{\delta'} \left( - \frac{s_{j,n}}{s_n} \right) \right|^2  \left| \nabla_{t,x} \tilde V_{\mr L,n}^j (0,x) \right|^2 dx \to 0 \]
for all $R \in \R$. Finally for the third term, $\ds \varphi_{\delta'}(|\tau_j|) - \varphi_{\delta'} \left( - \frac{s_{j,n}}{s_n} \right)$ is small for $n$ large, so it tends to 0. Thus the limit \eqref{Vj} holds.

In particular, if $|\tau_j| \le 1 -2 \delta'$, $\ds \varphi_{\delta'}\left( \frac{x}{s_n} \right)  \nabla_{t,x}  \tilde V^j_{\mr L,n} (0) \to 0$ in $L^2$. Thus we define 
\[ \q J := \left\{ j \in \{ 2, \dots, J \} \middle| \ \lim_{n \to +\infty} - \frac{s_{j,n}}{\mu_{j,n}} = \pm \infty \text{ and } 1 -2 \delta' \le |\tau| \le 1  \right\}. \]
We can define our new profiles
\begin{equation} \label{def:Vj}
\forall j \in \q J, \quad \vec V^j = \varphi_{\delta'} \left( | \tau_j| \right) \vec{\tilde V}^j, \quad \text{with parameters} \quad \{\mu_{j,n}, s_{j,n}\}_n
\end{equation}
Thus, using \eqref{def:V1}, \eqref{VJ1} and \eqref{def:Vj} we deduce the existence of a remainder term $(w_{0,n}^J, w_{1,n}^J)$ such that
\begin{gather*}
\varphi_{\delta'} \left( \frac{x}{s_n} \right) ( u(s_n), \partial_t u(s_n)) = \vec V^1_{L,n}(0) + \sum_{j \in \q J}   \vec V_{\mr L,n}^j(0) +  (w_{0,n}^J, w_{1,n}^J), \\
\text{where} \quad  \lim_{J \to +\infty} \limsup_{n \to +\infty} \| S(t) (w_{0,n}^J, w_{1,n}^J) \|_{S(\R)} =0.
\end{gather*}
This gives Step 1.

\textbf{Step 2:} Let $u_n$ be solution to \eqref{u eq} with initial data $\ds \varphi_{\delta'} \left( \frac{x}{s_n} \right) \vec u(s_n)$. Let $\vec V^j$ be the nonlinear profile associated to $\ds \left\{ \vec V^j_{\mr L},   \mu_{j,n} ,s_{j,n}  \right\}$, and
\[  V^j_{n}(s,x) =  \frac{1}{\mu_{j,n}} V^j \left( \frac{s-s_{j,n}}{\mu_{j,n}}, \frac{x}{\mu_{j,n}} \right). \]
We use Proposition \ref{nonlin profile} (with $t_n = s_n/2$): $u_n$ is defined on $[0,s_n/2]$  and
\begin{gather*}
\vec u_n(s_n/2) = \sum_{j=1}^J \vec V_n^j(s_n/2) + \vec w_n^J (s_n/2) + \vec r_n^J(s_n/2), \\
\text{where} \quad \lim_{n \to +\infty}  \limsup_{J \to +\infty} \left( \| r_n^J \|_{S([0,s_n/2])} + \sup_{0 \le t \le s_n/2} \| \vec r_n^J(t) \|_{\HH} \right) =0.
\end{gather*}
Then $\ds \frac{\theta_n - s_{j,n}}{\mu_{j,n}} = \frac{s_n/2 - s_{j,n}}{\mu_{j,n}}$, for $j=1$, $s_{1,n} =0$ and $\| (V_0^1, V_1^1) \|_{\HH} \le \e/2$, $\mu_{1,n} = s_n$; for $j \ge 2$, $\ds \lim_{n \to \infty} - \frac{s_{j,n}}{\mu_{j,n}} = \pm \infty$ and $\ds \lim_{n \to \infty} \left| \frac{s_{j,n}}{s_n} \right| = | \tau_j| \in [1-2\delta', 1]$, so that
\begin{align*}
\text{if }& \lim_{n \to \infty} - \frac{s_{j,n}}{\mu_{j,n}} = + \infty, \quad \text{then} \quad \lim_{n \to +\infty} \frac{s_n/2 - s_{j,n}}{\mu_{j,n}} = +\infty; \\
\text{and if }&  \lim_{n \to \infty} - \frac{s_{j,n}}{\mu_{j,n}} = - \infty, \quad \text{then} \quad \lim_{n \to +\infty} \frac{s_n/2 - s_{j,n}}{\mu_{j,n}} = -\infty. 
\end{align*}
The last limit follows from 
\[ \frac{s_n/2 - s_{j,n}}{\mu_{j,n}} = - \frac{s_{j,n}}{\mu_{j,n}} \left( - \frac{s_n}{2 s_{j,n}} + 1 \right), \quad \text{and} \quad -\frac{s_n}{2 s_{j,n}} +1 \to \frac{1}{2\tau_j}+1>0. \] 
(we recall $\tau_j \in [-1,-1+2\delta']$).

Let 
\[ t_n = \frac{3}{2} s_n, \quad t_{j,n} = s_{j,n} - \frac{s_n}{2}, \quad \delta = \frac{\delta'}{3}. \]
Now by definition of $\varphi_{\delta'}$,
\[ \text{if} \quad  | x | \ge (1 - \delta') s_n, \quad \text{then} \quad \vec u_n(0,x) = \varphi_{\delta'} \left( \frac{x}{s_n} \right) \vec u (s_n,x) = \vec u(s_n,x), \]
so, by finite speed of propagation 
\[  \text{if} \quad |x| \ge (3/2 -\delta')s_n = (1-2\delta)t_n, \quad \text{then} \quad \vec u_n (s_n/2, x) = \vec u (t_n,x). \]
Thus
\begin{multline*} 
 \varphi_{\delta} \left( \frac{x}{t_n} \right) \vec u(t_n) = \varphi_{\delta} \left( \frac{x}{t_n} \right) \vec u(s_n/2)= \\
 = \sum_{j=1}^J \varphi_{\delta} \left( \frac{x}{t_n} \right) \vec V_n^j (s_n/2) + \varphi_{\delta} \left( \frac{x}{t_n} \right) \vec w_n^J(s_n/2) + \varphi_{\delta} \left( \frac{x}{t_n} \right) \vec r_n^J(s_n/2).
\end{multline*}
Next note the for $n$ large, $J$ large, $\| \varphi_{\delta} (x/t_n) \vec r_n^J(s_n/2) \|_{\HH}$ is small, so we can ignore this term. Also observe that $\| S(t) (w_{0,n}^J ,w_{1,n}^J ) \|_{S(\R)}$ is small for $J$ large, $n$ large, hence the same is true for $\| S(t) (\varphi_\delta(x/t_n) \vec w_n^J(s_n/2) ) \|_{S(\R)}$ by Lemma \ref{local scat lem}.

\medskip 

Next for $j=1$, recall that $\supp (V_0^1, V_1^1) \subset \{ x | \ |x| \le 1 \}$ and $\| (V_0^1, V_1^1 ) \|_{\HH} \le \e/2$, so that by small data theory, $\| (V_n^1(t), \partial_t V_n^1(t) ) \|_{\HH} \le C \e/2$ where
\[ V_n^1(t,x) = \frac{1}{s_n} V^1 \left( \frac{t}{s_n}, \frac{x}{s_n} \right), \quad \partial_t V_n^1(t,x) = \frac{1}{s_n^2} \partial_t V^1 \left( \frac{t}{s_n}, \frac{x}{s_n} \right). \]
Let 
\[  (U_0^1, U_1^1) = \left( \varphi_{\delta} \left( \frac{x}{t_n} \right) \frac{1}{s_n} V^1 \left( \frac{1}{2} , \frac{x}{s_n} \right), \quad \varphi_{\delta} \left( \frac{x}{t_n} \right) \frac{1}{s_n^2} \partial_t V^1 \left( \frac{1}{2} , \frac{x}{s_n} \right) \right). \]
We will let $t_{1,n} =0$, $\lambda_{1,n} =s_n$ (and recall $t_n = 3s_n/2)$. Then $\| (U_0^1, U_1^1) \|_{\HH} \le C \e$.

For $j \ge 2$, consider first those $j$ such that $\ds \lim_{n \to +\infty} - \frac{s_{j,n}}{\mu_{j,n}} = -\infty$. We claim that 
\begin{equation} \label{Vj-infty}
\lim_{n \to +\infty} \left\|  \varphi_{\delta} \left( x/t_n \right) \vec V_n^j (s_n/2) \right\|_{\HH} =0. 
\end{equation}
In fact, since $\ds \frac{s_n/2 - s_{j,n}}{\mu_{j,n}} \to -\infty$ and
\[ \vec V_n^j(s_n/2,x) = \left( \frac{1}{\mu_{j,n}} V^j\left( \frac{s_n/2-s_{j,n}}{\mu_{j,n}}, \frac{x}{\mu_{j,n}} \right) , \frac{1}{\mu_{j,n}^2} \partial_t V^j \left( \frac{s_n/2-s_{j,n}}{\mu_{j,n}}, \frac{x}{\mu_{j,n}} \right) \right), \]
then $\| \vec V_n^j(s_n/2) - \vec V_{\mr L,n}^j (s_n/2) \|_{\HH} \to 0$ as $n \to +\infty$. We are left to bound $$\| \varphi_\delta(x/t_n) \vec V_{\mr L,n}^j (s_n/2) \|_{\HH}.$$ Recall
\begin{gather*}
\vec V_{\mr L, n}^j (t, x) = \frac{1}{\mu_{j,n}} \vec V^j_{\mr L} \left( \frac{t-s_{j,n}}{\mu_{j,n}}, \frac{x}{\mu_{j,n}} \right), \\
 \text{so that} \quad \vec V_{\mr L, n}^j (s_n/2, x) = \frac{1}{\mu_{j,n}} \vec V_{\mr L}^j \left( \frac{s_n/2-s_{j,n}}{\mu_{j,n}}, \frac{x}{\mu_{j,n}} \right).
\end{gather*}
Recall that $\ds \frac{s_n/2 - s_{j,n}}{\mu_{j,n}} \to -\infty$ in this case. Let $\ds -t_{j,n} = \frac{s_n}{2} - s_{j,n}$, so that $\ds -\frac{t_{j,n}}{\mu_{j,n}} \to~-\infty$.

Let $\e>0$ be given, apply Lemma \ref{lem:scaled_lin_disp} and choose $R$ large so that 
\[ \limsup_{n \to +\infty} \int_{||x| - |t_{j,n}|| \ge R \mu_{j,n}} |\nabla_{t,x} V_{\mr L,n}^j (-t_{j,n},x)|^2 dx \le \e. \]
On the other hand, on the  support of $ \varphi_\delta(x/t_n)$ we have $(1-2\delta) t_n \le |x|$. This means that $(1- 2\delta) 3s_n/2 \le |x|$, which implies that 
\[ (1-2\delta) \frac{3}{2} \frac{s_n}{\mu_{j,n}} \le \frac{|x|}{\mu_{j,n}}. \]
We claim that for $n$ large, 
\[ \{ x \mid (1-2\delta) 3s_n/2 \le |x| \} \cap \{ x | \ ||x| - |t_{j,n}|| \le R \mu_{j,n} \} = \varnothing. \]
Indeed, if $x$ lies in the intersection
\[ (1-2\delta) \frac{3s_n}{2} \le |x| \le R \mu_{j,n} + |t_{j,n}| = R \mu_{j,n} + s_n \left| \frac{1}{2} - \frac{s_{j,n}}{s_n} \right|. \]
As $\ds - \frac{s_{j,n}}{s_n} \to \tau_j \in [-1,2\delta' -1]$, if $\delta$ is small, $\ds \frac{1}{2} s_n \le R \mu_{j,n}$, but $\ds \lim_{n \to +\infty} \frac{s_n}{\mu_{j,n}} = +\infty$, a contradiction. This shows that our claim holds, and hence
\[ \limsup_{n \to +\infty} \| \varphi_\delta(x/t_n) \vec V_{\mr L,n}^j (s_n/2) \|_{\HH} \le \e. \]
This establishes \eqref{Vj-infty}.

The third case is when $j \ge 2$, $\ds \lim_{n \to +\infty} - \frac{s_{j,n}}{\mu_{j,n}} = +\infty$. We claim that in this case,  we have
\begin{equation} \label{Vj+infty}
 \lim_{n \to +\infty} \left\| \varphi_{\delta} \left( \frac{x}{t_n} \right) \vec V_n^j \left( \frac{s_n}{2} \right) - \varphi_{\delta} \left( \frac{1}{3} + \frac{2}{3} \tau_j \right) \vec V_{\mr L,n}^j \left( \frac{s_n}{2} \right) \right\|_{\HH} =0.
 \end{equation}
We proceed similarly to \eqref{Vj}:
\begin{align*}
 &\nabla_{t,x} \left( \varphi_{\delta} \left( x/t_n \right) V_n^j \left( s_n/2 \right) \right) - \varphi_{\delta} \left( \frac{1}{3} + \frac{2}{3} \tau_j \right) \nabla_{x,t} V_{\mr L,n}^j \left( \frac{s_n}{2} \right) \\
& = \frac{1}{t_n} \nabla \varphi_{\delta} (x/t_n) V_{\mr L,n}^j \left( s_n/2 \right) \\
& \qquad + \left( \frac{1}{t_n} \nabla \varphi_{\delta} (x/t_n) + \varphi_{\delta} \left( x/t_n \right) \right) \nabla_{x,t}  \left( V_n^j \left( s_n/2 \right) - V_{\mr L,n}^j (s_n/2) \right) \\
& \qquad + \left( \varphi_{\delta} \left( x/t_n \right) - \varphi_{\delta} \left( \frac{t_{j,n}}{\mu_{j,n}} \right) \right) \nabla_{x,t} V_{\mr L,n}^j (s_n/2) \\
& \qquad + \left( \varphi_{\delta} \left( \frac{t_{j,n}}{\mu_{j,n}} \right) - \varphi_{\delta} \left( \frac{1}{3} + \frac{2}{3} \tau_j \right) \right) \nabla_{x,t} V_{\mr L,n}^j (s_n/2)
\end{align*}
We now show that each of the four terms in the right hand-side tends to 0 as $n \to +\infty$.
 Let $\e>0$. Lemma \ref{lem:scaled_lin_disp} provides us with $R$ so that 
\[ \int_{||x| - |t_{j,n}|| \ge R \mu_{j,n}} |\nabla_{t,x} V_{\mr L,n}^j (s_n/2,x)|^2 dx \le \e. \] 
Proceeding as in \eqref{phixVj}, as $\mu_{j,n} /t_n \to 0$, we see that  
\[ \forall R \in \R, \quad \int_{||x| - |t_{j,n}|| \le R \mu_{j,n}} \left| \frac{1}{t_n} \nabla \varphi_{\delta} (x/t_n) V_{\mr L,n}^j \left( s_n/2,x \right) \right|^2 dx \to 0. \]
Hence the first term tends to 0. For the third term, it also suffices to consider the case when $||x| - |t_{j,n}|| \le R \mu_{j,n}$, but then
\[ \left| \varphi_{\delta} \left( \frac{x}{t_n} \right) - \varphi_{\delta} \left( - \frac{t_{j,n}}{t_n} \right) \right| \le C_\delta R \frac{\mu_{j,n}}{t_n} = \tilde C_{\delta} \frac{\mu_{j,n}}{s_n} \to 0 \quad \text{as} \quad n \to +\infty, \]
Hence the third term tends to 0. For the second term, as we saw before $$\ds \frac{s_n/2 - s_{j,n}}{\mu_{j,n}} \to~\infty$$ in this case (and $t_n \to +\infty)$), so that by the definition of nonlinear profiles
\[ \frac{1}{\mu_{j,n}^{1/2}} \vec V^j \left( \frac{s_n/2-s_{j,n}}{\mu_{j,n}} , \frac{x}{\mu_{j,n}} \right) - \frac{1}{\mu_{j,n}^{1/2}} \vec V_{\mr L}^j \left( \frac{s_n/2-s_{j,n}}{\mu_{j,n}} , \frac{x}{\mu_{j,n}} \right) \to 0 \quad \text{in} \quad \HH. \]
This shows that the second term tends to 0. Finally, note that $$\ds \frac{1}{3} + \frac{2}{3} \tau_j = \lim_{n \to \infty}  - \frac{t_{j,n}}{\mu_{j,n}},$$ thus the fourth term also tends to 0.  Claim~\ref{Vj+infty} follows.

Thus, it only remains to check the pseudo-orthogonality of $\{\mu_{j,n}, t_{j,n}\}$ for $j \ge 2$ of the second class ($\tau_j >0$). But
\[ \frac{\mu_{j,n}}{\mu_{k,n}} + \frac{\mu_{k,n}}{\mu_{j,n}} + \frac{|t_{j,n} - t_{k,n}|}{\mu_{j,n}} = \frac{\mu_{j,n}}{\mu_{k,n}} + \frac{\mu_{k,n}}{\mu_{j,n}} + \frac{|s_{j,n} - s_{k,n}|}{\mu_{j,n}} \to +\infty. \]
This finishes the proof of  Lemma~\ref{lem:1}.
\end{proof}

\begin{proof}[Proof of Proposition~\ref{th:u-v_L}]
\textbf{Step 1.} Let us prove first that for each $R \in \R$, there exists a solution $\vec v^R_{\mr L}$ of the linear equation \eqref{free wave}  such that 
\[ \int_{|x| \ge t-R} | \nabla_{t,x} (u - v_{\mr L}^R)(t,x)|^2 dx \to 0 \quad \text{as} \quad t \to +\infty. \]
Indeed, for each $n$ consider the solution $u_n$ to \eqref{u eq} with initial data $\varphi_{\delta} (x/t_n) \vec u(t_n)$ at $t=0$. Because of Lemma \ref{lem:1} and Proposition \ref{nonlin profile}, $u_n$ is globally defined and scatters for positive times. Fix $n$ large, let $w_{\mr L,n}$ be the solution of the linear equation \eqref{free wave} such that
\[ \lim_{t \to +\infty} \| \vec u_n(t) - {\vec w}_{\mr L,n} \|_{\HH} = 0. \]
By finite speed of propagation, $\vec u(t_n+t,x) = \vec u_n(t,x)$ for $|x| \ge (1-\delta) t_n +t$, and $t \ge 0$. Hence,
\[ \int_{|x| \ge -\delta t_n+t} | \nabla_{t,x} u(t,x) - \nabla_{t,x} w_{\mr{L},n}(t-t_n,x))|^2 dx \to 0 \quad \text{as} \quad t \to +\infty. \]
We choose $n$ so large that $\delta t_n \ge R$, and define $\vec v_{\mr L}^R(t,x) := \vec w_{\mr L,n}(t-t_n,x)$: this step follows.

\bigskip

\textbf{Step 2:}
Choose $t_n$ as before, $S(-t_n) \vec u(t_n)$ has weak limit $(v_{0}, v_{1})$ in $\HH$: notice that due to \eqref{typeII}, we have
\[ \| (v_{0}, v_{1}) \|_{\q H} \le M. \]
Let $\vec v_{\mr L}$ be the free wave, solution to the linear equation \eqref{free wave}, with initial data $(v_0,v_1)$. We also have a profile decomposition
\[ \vec u(t_n) = \vec v_{\mr L} (t_n) + \sum_{j=2}^J \vec U_{\mr L,n}^j (0) + \vec w_n^J. \]
(Here we choose $\vec U_{\mr L}^1 = \vec v_{\mr L}$, $\lambda_{1,n} =1$, $t_{1,n} = -t_n$, which is allowed by construction of a profile decomposition, with profiles as weak limits). Also,
\[ \vec u(t_n) - \vec v_{\mr L}^R (t_n) = \vec v_{\mr L} (t_n) - \vec v_{\mr L}^R(t_n) + \sum_{j=2}^J \vec U_{\mr L,n}^j (0) + \vec w_n^J \]
is a profile decomposition. By Proposition~\ref{loc en dec}, we have an almost orthogonality:
\begin{align*} 
\MoveEqLeft\int_{|x| \ge t_n-R} |\nabla_{x,t} u(t_n,x) - \nabla_{x,t} v_{\mr L}^R (t_n,x)|^2 dx = \int_{|x| \ge t_n-R} | \nabla_{t,x} (v_{\mr L} - v_{\mr L}^R)(t_n,x)|^2 dx \\
& + \sum_{j=2}^J \int_{|x| \ge t_n-R} |\nabla_{x,t} U_{\mr L,n}^j (0,x)|^2 dx + \int_{|x| \ge t_n-R} |\nabla_{x,t} w^J_n(0,x)|^2 dx + o_n(1).
\end{align*}
The left hand side tends to 0, and as all the terms on the right hand side are non negative, we deduce
\[ \ds \lim_{n \to +\infty} \int_{|x| \ge t_n-R} | \nabla_{t,x} (v_{\mr L} - v_{\mr L}^R)(t_n,x)|^2 dx =0. \]
Since $v_{\mr L} - v_{\mr L}^R$ is a solution to the linear wave equation \eqref{free wave}, by decay of outer free energy, we have
\[  \lim_{t \to +\infty} \int_{|x| \ge t-R} | \nabla_{t,x} (v_{\mr L} - v_{\mr L}^R)(t,x)|^2 dx =0, \]
which gives our result.
\end{proof}


\subsection{Vanishing energy in the self similar region for global solutions}

In this subsection we prove the analog of Theorem~\ref{self sim} for smooth global solutions to~\eqref{u eq}. 

\begin{thm} \label{thm:en_decay_global}
Assume that $\vec u(t)$ is a smooth finite energy solution to~\eqref{u eq}. Let $\lambda \in (0,1)$. Then  
\[ \limsup_{t \to +\infty} \int_{\lambda t}^{t-R} \left( |\nabla_{t,r} u(t,r)|^2 + \frac{|u(t,r)|^2}{r^2} \right) r^3 dr \to 0 \quad \text{as } R \to +\infty. \]
\end{thm}
We will also require the following simple consequence of Theorem~\ref{thm:en_decay_global}.

\begin{cor}\label{u-vL}
Let $\lambda \in (0,1)$. Then as $t \to +\infty$
\begin{gather*}
\| \nabla_{t,r} u(t) - \nabla_{t,r} v_{\mr L}(t) \|_{L^2(r \ge \lambda t)} \to 0 \quad \text{and} \quad \| r u(t,r) \|_{L^\infty(r \ge \lambda t)} \to 0.
\end{gather*}
\end{cor}


\begin{proof}[Proof of Corollary~\ref{u-vL}]
From Theorem \ref{thm:en_decay_global} and Proposition \ref{prop:6},
\begin{align*}
& \limsup_{t \to +\infty} \| \nabla_{t,r} u(t) - \nabla_{t,r} v_{\mr L}(t) \|_{L^2(r \in [\lambda t, t-R])}  \\
& \le \limsup_{t \to +\infty} \| \nabla_{t,r} u(t) \|_{L^2(r \in [\lambda t, t-R])} + \| \nabla_{t,r} v_{\mr L}(t) \|_{L^2(r \in [\lambda t, t-R])} \\
& \to 0 \quad \text{as } R \to +\infty.
\end{align*}
Now, we use Proposition~\ref{th:u-v_L} on the interval $[t-R,+\infty)$, and this gives the convergence
\[ \| \nabla_{t,r} u(t) - \nabla_{t,r} v_{\mr L}(t) \|_{L^2(r \ge \lambda t)} \to 0. \]
Then it follows from Lemma \ref{lem:hardy_loc} that
\[ \| |x| (u(t,x) -v_{\mr L}(t,x)) \|_{L^\infty(|x| \ge \lambda t)} \to 0. \]
But as $v_{\mr L}$ is a linear solution, $\| r v_{\mr L} (t,r) \|_{L^\infty_r} \to 0$ as $t \to +\infty$. Indeed, if $v_{\mr L}(0)$ is smooth with compact support, we have the well-known dispersive estimate 
\EQ{\label{disp est}
\forall t >0, \ \forall r \ge 0, \quad \abs{v_{L}(t, r) }\le C t^{-\frac{3}{2}}.
}
Combined with finite speed of propagation yields the result in this case. It follows in the general case via a density argument. This gives the second convergence.
\end{proof}

We use the linear solution $v_{\mr L}(t)$ constructed in the previous section, and  rely on our assumption of spherical symmetry. As in the finite time blow-up case we pass to a $2d$ formulation by introducing the functions
\[ \psi(t,r) = r u(t,r), \quad \phi(t,r) = r v_{\mr L}(t,r). \]

\begin{claim} \label{cl:phi_psi}
We have the convergences
\begin{gather*}
\limsup_{t \to +\infty} \int_{0}^{t-R} \left( |\nabla_{t,r} v_{\mr L}(t,r)|^2 + \frac{|v_{\mr L}(t,r)|^2}{r^2} \right) r^3 dr \to 0 \quad \text{as } R \to +\infty, \\
\limsup_{t \to +\infty} \int_{0}^{t-R} |\nabla_{t,r} \phi(t,r)|^2 rdr \to 0 \quad \text{as } R \to +\infty, \\
\limsup_{t \to +\infty} \sup_{r \in [0,t-R]} |\phi(t,r)| \to 0 \quad \text{as } R \to +\infty. \\
\forall R \ge 0, \quad \lim_{t \to +\infty} \int_{t-R}^\infty \left| \nabla_{t,r} \phi(t,r) - \nabla_{t,r} \psi(t,r) \right|^2  rdr \to 0,
\end{gather*}
and the bounds, for some constant $C(M)$ depending only on the constant $M$ (defined in~\eqref{typeII}), and all $t \ge 0$:
\begin{gather*}
 \int_0^\infty \left(  |\nabla_{t,r} \psi(t,r) |^2 + \frac{|\psi(t,r)|^2}{r^2} \right) rdr \le C(M)^2, \\
  \int_0^\infty \left(  |\nabla_{t,r} \phi(t,r) |^2 + \frac{|\phi(t,r)|^2}{r^2} \right) rdr \le C(M)^2.
\end{gather*}
\end{claim}

\begin{proof}
Proposition \ref{prop:6} yields that
\[ \limsup_{t \to +\infty} \int_{0}^{t-R} |\nabla_{t,r} v_{\mr L}(t,r)|^2 r^3 dr  \to 0 \quad \text{as } R \to +\infty. \]
By~\eqref{bound:hardy_loc_infty} in Lemma~\ref{lem:hardy_loc}, with $r=0$, we sees that by density and preservation of the linear energy, it suffices to establish the convergence for $\vec v_L$ with initial data that is $C_0^{\infty}$ (and radial). 

We now use \eqref{disp est} and \eqref{pr id} and the fact that by Lemma~\ref{lem:hardy_loc} we know that $s \abs{v_L (t, s)} \to 0$ as $s \to 0$, to integrate~\eqref{pr id} between $r=0$ and $r = t-R$ to obtain
\ant{
2 &\int_0^{t-R}v_L^2(t, r) \, r \, dr = (t-R)^2v_L^2(t, , t-R) -2 \int_0^{t-R} v_L(t, r) \, \p_r v_L(t, r)\, r^2 \, dr\\
 &\quad\le (t-R)^2 v_L^2(t, t-R) + \int_0^{t-R} v_L^2(t, r) \, r \, dr + \int_0^{t-R} \abs{\p_r v_L(t, r)}^2 \, r^3 \, dr.
 }
 Hence, 
 \ant{
 \int_0^{t-R}v_L^2(t, r) \, r \, dr & \le C \frac{(t-R)^2}{t^3} + \int_0^{t-R} \abs{\p_r v_L(t, r)}^2 \, r^3 \, dr,
 } 
 which combined with Proposition~\ref{prop:6} gives the first statement.  Now, $\p_r \phi = r \p_r v_{L} + v_{L}$ and the second statement follows. The third statement is then a consequence of~\eqref{bound:infty_loc} in Lemma~\ref{lem:hardy_loc}. The fourth and last convergence is a reformulation of the extraction of the linear term Proposition~\ref{th:u-v_L}, in light of Lemma~\ref{lem:hardy_loc}, \eqref{bound:hardy_loc_infty}.

 Finally, the first bound is a consequence of the type-II bound \eqref{typeII} combined with Lemma~\ref{lem:hardy_loc}, \eqref{bound:hardy_loc_infty} with $r =0$. For the second bound we also use \eqref{bound:hardy_loc_infty},  Proposition~\ref{th:u-v_L} and conservation of the linear energy. 


\end{proof}

The proof of Theorem \ref{thm:en_decay_global} follows the same general outline as for the finite time blow-up case. First, we prove desired vanishing of the energy for a particular $\lambda \in (0, 1)$, conditional to an $L^{\infty}$ bound which guarantees the positivity of the flux. We then prove that this implies the general case of the theorem via an inductive argument. 

\begin{prop} \label{prop:en_decay_global_1}
Assume that there exist $\lambda \in (0,1)$ and $A \ge 0$ and $T \ge A/(1-\lambda)$ such that
\begin{equation} \label{hypo:lambda_psi}
\forall t \ge T, \ \forall r \in [ \lambda t, t- A], \quad  |\psi(t,r)| \le \frac{\sqrt{2}}{2}.
\end{equation}
Then
\[ \limsup_{t \to +\infty} \int_{\lambda t}^{t-R} \left( |\partial_t \psi(t,r)|^2 + |\partial_r \psi(t,r)|^2 \right) rdr \to 0 \quad \text{as } R \to \infty. \]
\end{prop}

Let us postpone the proof of Proposition \ref{prop:en_decay_global_1} and use it to  prove Theorem \ref{thm:en_decay_global}. 

\begin{proof}[Proof  that Proposition \ref{prop:en_decay_global_1} implies Theorem \ref{thm:en_decay_global}] The proof follows in two steps.

\textbf{Step 1:} We begin by establishing the following claim which establishes the desired vanishing in terms of $\vec \psi$.  
\begin{claim} \label{cl:l0}
For all $\lambda \in (0,1)$, 
\begin{gather}
\limsup_{t \to +\infty} \int_{\lambda t}^{t-R} \left( |\partial_t \psi(t,r)|^2 + |\partial_r \psi(t,r)|^2 \right) rdr \to 0 \quad \text{as } R \to \infty.  \label{eq:cond_I_2} 
\end{gather}
\end{claim}

\begin{proof}[Proof of Claim~\ref{cl:l0}]
Consider the collection $I$  of the $\lambda \in (0,1)$ such that there exist $R(\lambda) \ge 0$ and $T(\lambda) \ge R(\lambda)/(1-\lambda)$ such that
\begin{gather}
\forall t \ge T(\lambda), \ \forall r \in [\lambda t, t-R(\lambda)], \quad | \psi(t,r) | \le \frac{\sqrt{2}}{2}.
\label{eq:cond_I_1}
\end{gather}

Observe that if $\lambda \in I$ then $[\lambda,1) \subset I$ (for any $\lambda' \in [\lambda,1)$, notice that $R(\lambda') = R(\lambda)$ and $T(\lambda') = \max (R(\lambda)/(1-\lambda'), T(\lambda))$ work, because $\lambda' t \ge \lambda t$). Also, in view of Proposition \ref{prop:en_decay_global_1}, if $\lambda \in I$, then
\eqref{eq:cond_I_2} holds for this particular $\la$.

Hence it is enough to prove that $I$ contains a sequence $\{ \la_n \} \subset (0,1)$ which converges to 0: this is our goal from now on.


Let us first show that $I$ is non empty. First observe that there exists $R_0>0$ and $T_0>0$ such that for all $t \ge T_0$,
\begin{gather*}
|\phi(t,t-R_0) - \psi(t,t-R_0)| \le 1/6, \quad \text{and} \quad \sup_{r \in [0,t-R_0]} |\phi(t,r)| \le 1/6.
\end{gather*}
Indeed, we invoke Claim \ref{cl:phi_psi} (and \eqref{bound:infty_loc_infty}).

Let $\lambda_0 \in (0,1)$ to be determined later. Let $t \ge T_0$ and $r \in [\lambda_0 t, t-R_0]$. Then, (assuming that $T_0 \ge R_0/(1-\la_0)$, using Claim \ref{cl:phi_psi} repeatedly,
\begin{align*}
& | \psi(t,r) - \phi(t,r)| \\
 & \le | \psi(t,t-R_0) - \phi(t,t-R_0)|  +  \int_r^{t-R_0} \sqrt{r'} |\partial_r \psi(t,r') - \partial_r \phi(t,r')| \frac{dr'}{\sqrt{r'}} \\
&  \le \frac{1}{6} + \left( \int_r^{t-R_0} |\partial_r \psi(t,r') -\partial_r \phi(t,r')|^2 r'dr' \right)^{1/2}  \left( \int_r^{t-R_0} \frac{dr'}{r'} \right)^{1/2} \\
& \le 1/6 + 2 C(M) \sqrt{ \log \frac{t-R_0}{r}} \le 1/3 + 2 C(M) \sqrt{|\log \lambda_0|}
\end{align*}
Thus, for $t \ge T_0$, $r \in [\lambda_0 t, t-R_0]$, and provided that $T_0 \ge R_0/(1- \la_0)$
\[ |\psi(t,r)| \le |\phi(t,r)| + | \psi(t,r) - \phi(t,r)| \le 1/3 + 2 C(M) \sqrt{|\log \lambda_0|}. \]
Choose now $\lambda_0 \in (1/2,1)$ such that $|\lambda_0 -1 | \le 1/(144 C(M)^2) )$ and  use the fact that for such $\la_0$ we have $|\log \lambda_0| \le 2|\lambda_0 - 1|$. Now define $T(\lambda_0) := \max ( T_0, R_0/(1-\lambda_0) )$: for $t \ge T(\lambda_0)$, the interval $[\lambda t, t-R_0]$ is never empty. From the definition of $\lambda_0$, we get
\[ \sup_{t \ge T_0, \ r \in [\lambda_0 t, t-R_0]} | \psi(t,r) | \le 2/3 \le \frac{\sqrt{2}}{2}. \]
Hence condition \eqref{eq:cond_I_1} is fulfilled with $R(\lambda_0) := R_0$, and  $\lambda_0 \in I$.


Denote $\lambda_n := \lambda_0^n$, and let us now prove by induction that $\lambda_n \in I$, with $R(\la_n) = R_0$. We just proved that $\la_1 \in I$ (with $R(\la_1)=R_0$). Assume that $n \ge 1$ and $\la_n \in I$ with $R(\la_n) = R_0$. First, for all $R>0$ and $t \ge R/(1-\la_n)$
\begin{align*}
& \int_{\la_n t}^\infty |\nabla_{t,r} \phi(t,r) - \nabla_{t,r} \psi(t,r)|^2 rdr \le  \int_{t-R}^\infty |\nabla_{t,r} \phi(t,r) - \nabla_{t,r} \psi(t,r)|^2 rdr
 \\
& \qquad + 2 \int_{\la_n t}^{t-R} |\nabla_{t,r} \phi(t,r)|^2 rdr + 2 \int_{\la_n t}^{t-R} |\nabla_{t,r} \psi(t,r)|^2 rdr.\end{align*}
As $\la_n \in I$, \eqref{eq:cond_I_2} holds; using Claim \ref{cl:phi_psi}, and after taking the limsup in $t \to +\infty$ and letting $R \to +\infty$, we infer
\[ \int_{\la_n t}^\infty |\nabla_{t,r} \phi(t,r) - \nabla_{t,r} \psi(t,r)|^2 rdr \to 0 \quad \text{as } t \to +\infty. \]
Using \eqref{bound:infty_loc_infty}, there exists $T_n$ such that for all $t \ge T_n$
\[ |\phi(t,\la_n t) - \psi(t,\la_n t)| \le 1/6. \]
Then define $T(\la_{n+1}) = \max( T_n, T(\la_n))$. For $t \ge T(\la_{n+1})$ and $r \in [\la_{n+1} t, \la_n t]$ (notice that $\la_n t \le t-R_0$), there holds
\begin{align*}
& | \psi(t,r) - \phi(t,r)| \\
& \le | \psi(t,\la_n t) - \phi(t,\la_n t)| +  \int_r^{\la_n t} \sqrt{r'} |\partial_r \psi(t,r') - \partial_r \phi(t,r')| \frac{dr'}{\sqrt{r'}} \\
&  \le \frac{1}{6} + \left( \int_r^{\la_n t} |\partial_r \psi(t,r') -\partial_r \phi(t,r')|^2 r'dr' \right)^{1/2}  \left( \int_r^{\la_n t} \frac{dr'}{r'} \right)^{1/2} \\
& \le 1/6 +  2 C(M) \sqrt{ \log \frac{\la_n t}{r}} \le 1/6 + 2 C(M) \sqrt{ |\log \lambda_0|} \le 1/2.
\end{align*}
Thus, by our choice of $\lambda_0$, we have for all $t \ge T(\la_{n+1})$,
\[ \sup_{r \in [\la_{n+1} t, \la_{n} t]} | \psi(t,r) | \le \frac{\sqrt{2}}{2}. \]
As $\mu_{n} \in I$ with $R(\la_{n}) = R_0$ by assumption, we see that
\[ \sup_{t \ge T(\la_{n+1}), \ r \in [\la_{n+1} t, t-R_0]} | \psi(t,r) | \le \frac{\sqrt{2}}{2}, \]
so that $\la_{n+1} \in I$ with $R(\la_{n+1})=R_0$. This completes the induction.
 
 
Finally as $\la_n \to 0$, and $\la_n \in I$ for all $n \ge 1$, we conclude that $I = (0,1)$ and for all $\la \in (0,1)$, \eqref{eq:cond_I_2}  holds, as desired.
\end{proof}

\textbf{Step $2$:} To complete the proof we now transfer these results to $\vec u(t)$.  Let $\lambda \in (0,1)$.  Claim \ref{cl:l0} combined with the second and fourth statements of Claim \ref{cl:phi_psi} show that
\[ \int_{\lambda t}^\infty \left( |\partial_t (\psi - \phi)(t,r)|^2 + |\partial_r (\psi - \phi)(t,r)|^2 \right) rdr \to 0 \quad \text{as} \quad t \to +\infty. \]
This already gives that
\[  \int_{\lambda t}^\infty |\partial_t u(t,r) - \partial_t v_{\mr L}(t,r)|^2 r^3 dr \to 0. \]
From the fact that 
\ant{
\int_{\la t}^{\infty} \abs{ \p_r( \psi- \phi)(t, r)}^2 r \, dr \to 0,
} 
we see that for $\la t < r < t-R$, we have
\ant{
\psi(t, r) - \phi(t, r) = \psi(t, t-R) - \phi(t, t-R) - \int_{r}^{t-R}( \psi- \phi)_r(t, \rho) \, d \rho, 
}
so that 
\begin{multline*}
\abs{\psi(t, r) - \phi(t, r)} \le \abs{\psi(t, t-R) - \phi(t, t-R)} + \int_{\la t}^{t-R} \abs{ (\psi - \phi)_r(t, \rho)} \, d \rho\\
 \le \abs{\psi(t, t-R) - \phi(t, t-R)} + \left( \int_{\la t}^{t-R} \abs{ (\psi - \phi)_r(t, \rho)}^2 \, \rho \, d \rho\right)^{\frac{1}{2}} \left( \int_{\la t}^{t-R} \frac{d \rho}{\rho} \right)^{\frac{1}{2}}.
\end{multline*}
Hence, using also~\eqref{bound:infty_loc_infty}, Proposition~\ref{th:u-v_L}, we obtain that 
\ant{
\lim_{R \to \infty}  \limsup_{t \to \infty} \sup_{r \in [\la t, t-R]} \abs{ \psi(t, r) - \phi(t, r)} = 0.
}
Therefore, using that 
\ant{
\int_{\la t}^{t-R} \abs{ u(t, r) - v_L(t, r)}^2 \, r^2 \, \frac{dr}{r} \le  \sup_{r \in [\la t, t-R]} \abs{\psi(t, r) - \phi(t, r)} \log\left( \frac{1}{\la}- \frac{R}{\la t} \right), 
}
we see that, 
\ant{
\lim_{R \to \infty} \limsup_{t \to \infty} \int_{\la t}^{t-R} \abs{u(t, r) - v_L(t,r)} \, r \, dr = 0, 
}
and hence, 
\ant{
\lim_{R \to \infty} \limsup_{t \to \infty} \int_{\la t}^{t-R} \abs{u_r(t, r) - v_r(t, r)}^2 \, r^3 \, dr = 0, 
}
which combined with the second and third statements in Claim~\ref{cl:phi_psi}, gives Theorem~\ref{thm:en_decay_global}. Note that using Lemma~\ref{lem:hardy_loc}, \eqref{bound:hardy_loc_infty}, and Proposition~\ref{th:u-v_L}, we in fact have 
\ant{
\lim_{t \to \infty} \int_{\la t}^{\infty} \left\{ \abs{ \na_{t, x} (u-v_L)(t, r)}^2
+ \frac{\abs{(u-v_L)(t, r)}^2}{r^2} \right\} \, r^3 \, dr = 0
}

\end{proof}

\subsection{Proof of Proposition~\ref{prop:en_decay_global_1}}
We now turn to the proof of Proposition~\ref{prop:en_decay_global_1}. We recall the fact that $\psi$ satisfies a $2d$ equation as in~\eqref{psi eq}. 
\begin{equation} \label{eq:psi}
\partial_{tt} \psi - \partial_{rr} \psi - \frac{1}{r} \partial_r \psi + \frac{f(\psi)}{r^2} =0, \quad \text{where} \quad f(\psi) = \psi - \psi^3
\end{equation}
Again we let $$\ds F( \psi) =  \int_0^\psi f(\rho) d\rho = \frac{\psi^2}{2} - \frac{\psi^4}{4} = \frac{\psi^2}{2} (1-\psi^2/2),$$ so that if $|\psi| \le \sqrt 2$, then $F (\psi) \ge 0$.

As in the finite time blow-up case the crux of the argument will be that hypothesis~\eqref{hypo:lambda_psi} will guarantee the positivity of the flux so that the methods in~\cite{CTZduke} can be applied. We will need refinements of their results, which were developed in~\cite{CKLS2} and required in order to establish Theorem~\ref{global main} and Theorem~\ref{global 2W}. The proof below actually combines ideas of \cite{CTZduke} and \cite{STZ92}.

We re-introduce the following quantities: 
\begin{align*}
&e(t,r):= \frac{1}{2}(\psi_t^2(t, r) +  \psi_r^2(t, r)) + \frac{F(\psi(t,r))}{r^2}\\
&m(t,r):=  \psi_t (t, r)  \psi_r(t, r).
\end{align*}
And recall again for convenience the identities
\begin{align}
&  \partial_t(r e)- \partial_r(r m)  =0, \label{en id global} \\
& \partial_{t} (rm) - \partial_r (re) = -  \frac{1}{2} \psi_t^2 +  \frac{1}{2}\psi_r^2 + \frac{F(\psi)}{r^2} - \frac{2f(\psi)}{r} \psi_r  =: L, \label{en id 2}
\end{align}
We define the null coordinates 
\begin{align*}
\eta = t+r, \quad \xi =t-r.
\end{align*}
Let $\ds \la' = \frac{1-\lambda}{1+\lambda}$, and denote
\begin{align*}
\q A^2(\eta, \xi):= r (e(t,r)+m(t, r)) = \frac{r}{2} (\partial_t \psi + \partial_r \psi)^2  + \frac{F(\psi)}{r}, \\
\q B^2(\eta, \xi):= r(e(t,r)- m(t, r)) =  \frac{r}{2}(\partial_t \psi - \partial_r \psi)^2  + \frac{F(\psi)}{r}. 
\end{align*}

\begin{flushleft}\textbf{Step 1: Vanishing of the flux.} First integrate the energy identity \eqref{en id global} on the truncated cone 
\end{flushleft}
\[ \q C(T,\xi_0) := \{ (\eta,\xi) \mid \eta \ge \xi \ge \xi_0,\ \eta + \xi \le 2T \}, \]
where $t \ge \xi_0 \ge A$ (see Figure \ref{fig:1}).  
\begin{figure} \label{fig:1}
\begin{tikzpicture}[
	>=stealth',
	axis/.style={semithick,->},
	coord/.style={dashed, semithick},
	yscale = 1,
	xscale = 1]
	\newcommand{\xmin}{0};
	\newcommand{\xmax}{5};
	\newcommand{\ymin}{0};
	\newcommand{\ymax}{5};
	\newcommand{\ta}{3};
	\newcommand{\fsp}{0.2};
	\draw [axis] (\xmin-\fsp,0) -- (\xmax,0) node [right] {$r$};
	\draw [axis] (0,\ymin-\fsp) -- (0,\ymax) node [below left] {$t$};
	\draw [axis] (0,0) -- (\xmax, \xmax) node [below right] {$\eta$};
	\draw [axis] (0,0) -- (-1,1) node [above right] {$\xi$};
	\draw [dashed] (\ymax-2,\ymax-1)  -- (0,\ymax-1) node [left] {$t=T$};
	\filldraw[color=light-gray1] (\ymax-2,\ymax-1) -- (0,1)  -- (0, \ymax-1); 
	\draw [thick] (\xmax - 1.5,\xmax - 0.5) node [above] {$\xi = \xi_0$} -- (0,1)  -- (0, \ymax);
\end{tikzpicture}
\caption{The cone  {$\q C(T,\xi_0)$}  in gray.}
\end{figure}
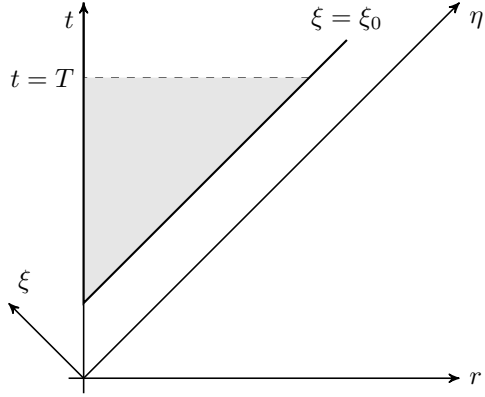
We see that
\ant{
0 & = \iint_{\q C(T,\xi_0)}  ( \partial_t(r e)- \partial_r(r m))  drdt \\
& = \int_0^{T-\xi_0} e(T,r) rdr  - \int_{\xi_0}^{2T-\xi_0}  \q A^2(\eta,\xi_0) d\eta,
} 
which implies that 
\ant{
\int_{\xi_0}^{2T-\xi_0}  \q A^2(\eta,\xi_0) d\eta = \int_0^{T-\xi_0} e(T,r) rdr.
}
Due to the type-II bound \eqref{typeII}, the term on the right-hand side is bounded by a constant depending on $M$ only as $T \to +\infty$. Now, as $\psi$ is smooth, $\ds \int_{\xi_0}^{\xi_0/\la'} \q A^2(\eta,\xi_0) d\eta$ is well defined. Also recall that for $\xi_0 \ge A$: if $\eta \ge \xi_0/\la'$, then $\q A^2(\eta,\xi_0) \ge 0$. Hence, by boundedness, $\ds \int_{\xi_0}^\infty  \q A^2 (\eta,\xi_0)d\eta$ converges. For all $\eta_0 \ge \xi_0 \ge A$, we can thus define the flux
\[ \operatorname{\mr{Flux}}(\eta_0, \xi_0) :=  \int_{\eta_0}^\infty  \q A^2  (\eta,\xi_0) d\eta. \]
Then for fixed $\xi_0$, $\operatorname{\mr{Flux}}(\eta_0, \xi_0) \ge 0$ as soon as $\la' \eta_0 \ge \xi_0$ and  
\[ \operatorname{\mr{Flux}}(\eta_0, \xi_0) \to 0 \quad \text{as} \quad \eta_0 \to +\infty. \]
Also, there exists a constant $C(M)$ such that
\begin{gather} \label{eq:flux_bd}
\forall \, \eta_0 \ge \xi_0 \ge A, \qquad |\operatorname{\mr{Flux}}(\eta_0, \xi_0)| \le C(M).
\end{gather}
Next, let $\eta_1 \ge \xi_1 \ge \xi_0 \ge A$. Integrating on the quadrangle
\[ Q(\eta_1; \xi_0,\xi_1) = \{ (\eta,\xi) \mid 0 \le \eta \le \eta_1, \ \xi_0 \le \xi \le \xi_1, \ \eta \ge \xi \} \]
with vertices $(\xi_0,\xi_0)$, $(\xi_0, \eta_1)$, $(\xi_1,\eta_1)$ and $(\xi_1,\xi_1)$ (see Figure \ref{fig:2}), we get
\begin{align*}
0 & = \iint_{Q(\eta_1; \xi_0,\xi_1)} \partial_t(r e)- \partial_r(r m) \\
& = \int_{\xi_0}^{\eta_1} \q A^2(\eta,\xi_0) d\eta + \int_{\xi_0}^{\xi_1} \q B^2 (\eta_1,\xi) d\xi - \int_{\xi_1}^{\eta_1} \q A^2 (\eta,\xi_1) d\eta
\end{align*}

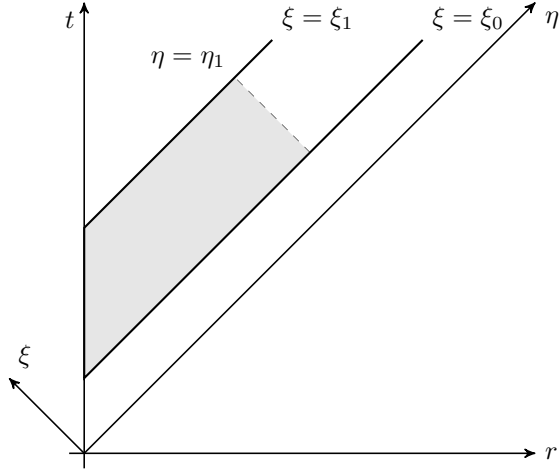
\begin{figure} \label{fig:2}
\begin{tikzpicture}[
	>=stealth',
	axis/.style={semithick,->},
	coord/.style={dashed, semithick},
	yscale = 1,
	xscale = 1]
	\newcommand{\xmin}{0};
	\newcommand{\xmax}{6};
	\newcommand{\ymin}{0};
	\newcommand{\ymax}{6};
	\newcommand{\ta}{3};
	\newcommand{\lamdba}{6};
	\newcommand{\fsp}{0.2};
	\draw [axis] (\xmin-\fsp,0) -- (\xmax,0) node [right] {$r$};
	\draw [axis] (0,\ymin-\fsp) -- (0,\ymax) node [below left] {$t$};
	\draw [axis] (0,0) -- (\xmax, \xmax) node [below right] {$\eta$};
	\draw [axis] (0,0) -- (-1,1) node [above right] {$\xi$};
	\draw [dashed] (\ymax-3,\ymax-2)  -- (\ymax-2.5- \ta/2,\ymax-2.5+\ta/2) node [above left] {$\eta=\eta_1$};
	\filldraw[color=light-gray1] (\ymax-3,\ymax-2) -- (0,1)  -- (0, \ta) -- (\ymax-2.5- \ta/2,\ymax-2.5+\ta/2); 
	\draw [thick] (\xmax - 1.5,\xmax - 0.5) node [above right] {$\xi = \xi_0$} -- (0,1)  -- (0, \ta) -- (\ymax-\ta-0.5, \ymax - 0.5) node [above right] {$\xi = \xi_1$};
\end{tikzpicture}
\caption{The quadrangle  {$Q(\eta_1; \xi_0,\xi_1)$}  in gray.}
\end{figure}

Letting $\eta_1 \to +\infty$,  we can define, for any $\xi_1 \ge \xi_0 \ge A$,
\begin{gather} \label{def:qF}
\operatorname{\q F}(\xi_0,\xi_1) := \lim_{\eta_1 \to +\infty} \int_{\xi_0}^{\xi_1} \q B^2 (\eta_1,\xi) d\xi  = \operatorname{\mr{Flux}}(\xi_0, \xi_0) - \operatorname{\mr{Flux}}(\xi_1, \xi_1).
\end{gather}
Now fix such $(\xi_0,\xi_1)$: for $\eta_1$ large enough, $\q B^2\ge 0$ on $[\xi_0,\xi_1] \times \{ \eta_1 \}$ . This proves that 
\[ \forall \xi_1 \ge \xi_0 \ge 0, \quad \operatorname{\q F}(\xi_0,\xi_1) \ge 0, \]
and so, $\xi \mapsto \operatorname{\mr{Flux}}(\xi, \xi)$ is non increasing. As it is bounded due to \eqref{eq:flux_bd}, there exists a limit as $\xi \to +\infty$, which we denote $\q E$:
\[ \q E : = \lim_{\xi \to +\infty} \operatorname{\mr{Flux}}(\xi, \xi). \]
Notice that we also have for $\xi_2 \ge \xi_1 \ge \xi_0$
\begin{gather} \label{eq:flux_add}
\operatorname{\q F}(\xi_0,\xi_2) =  \operatorname{\q F}(\xi_0,\xi_1) +  \operatorname{\q F}(\xi_1,\xi_2) \ge \operatorname{\q F}(\xi_0,\xi_1).
\end{gather}
Let us show that the map $(\xi_0,\xi_1) \mapsto \q F(\xi_0,\xi_1)$ is bounded on the set $\{ (\xi_0,\xi_1) \mid \xi_1 \ge \xi_0 \ge A \}$. Indeed, consider such $(\xi_0, \xi_1)$ and $\eta_0$ is so large that $\la' \eta_0 \ge \xi_1$, and the triangle with vertices $(\eta_0, \xi_0)$, $(\eta_0, \xi_1)$ and $(\xi_0, \eta_0 + \xi_1 - \xi_0)$:
\[ T(\eta_0,\xi_0,\xi_1) = \left\{ (\eta, \xi) \mid \ \eta_0 \le \eta,\  \xi_0 \le \xi, \ \eta+ \xi \le \eta_0 + \xi_1 \right\}. \]
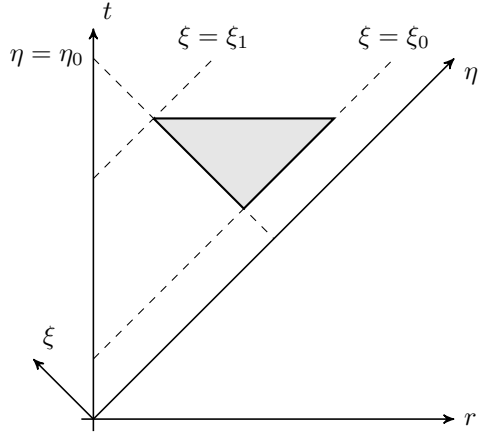
\begin{figure} \label{fig:3}
\begin{tikzpicture}[
	>=stealth',
	axis/.style={semithick,->},
	coord/.style={dashed, semithick},
	yscale = 0.8,
	xscale = 0.8]
	\newcommand{\xmin}{0};
	\newcommand{\xmax}{6};
	\newcommand{\ymin}{0};
	\newcommand{\ymax}{6};
	\newcommand{\ta}{3};
	\newcommand{\fsp}{0.2};
	\draw [axis] (\xmin-\fsp,0) -- (\xmax,0) node [right] {$r$};
	\draw [axis] (0,\ymin-\fsp) -- (0,\ymax+0.5) node [above right] {$t$};
	\draw [axis] (0,0) -- (\xmax, \xmax) node [below right] {$\eta$};
	\draw [axis] (0,0) -- (-1,1) node [above right] {$\xi$};
	\draw [dashed] (0,\ymax-2) --  (2,\ymax) node [above] {$\xi=\xi_1$};
	\draw [dashed] (0,1) --  (\ymax-1,\ymax) node [above] {$\xi=\xi_0$};
	\draw [dashed] (0,\ymax) node [left] {$\eta= \eta_0$} -- (\ymax/2, \ymax/2);
	\filldraw[color=light-gray1] (1,\ymax-1) -- (\ymax-2,\ymax-1) -- (\ymax/2-0.5, \ymax/2+0.5); 
	\draw [thick] (1,\ymax-1) -- (\ymax-2,\ymax-1) -- (\ymax/2-0.5, \ymax/2+0.5) -- (1,\ymax-1);
\end{tikzpicture}
\caption{The triangle {$T(\eta_0,\xi_0,\xi_1)$}  in gray.}
\end{figure}
(See Figure \ref{fig:3}). Observe that on $T(\eta_0; \xi_0,\xi_1)$, $\q A^2 \ge 0$ and $\q B^2 \ge 0$ and integrate the energy identity~\eqref{en id global} there:
\begin{align*}
0 & = \iint_{T(\eta_0,\xi_0,\xi_1)} \partial_t(re) - \partial_r(rm) drdt  \\
& = \int_{(\eta_0-\xi_0)/2}^{(\eta_0 - \xi_1)/2}  e(\eta_0+\xi_1,r) rdr - \int_{\eta_0}^{\eta_0+\xi_1-\xi_0} \q A^2(\eta,\xi_0) d\eta - \int_{\xi_0}^{\xi_1} \q B^2(\eta_0,\xi) d\xi. 
\end{align*}
Therefore, invoking again our type-II bound \eqref{typeII} and non-negativity of $\q A^2$:
\[ \int_{\xi_0}^{\xi_1} \q B^2(\eta_0,\xi) d\xi \le  \int_{(\eta_0-\xi_0)/2}^{(\eta_0 - \xi_1)/2} e(\eta_0+\xi_1,r) rdr \le C(M), \]
where $C(M)$ is independent of $(\eta_0,\xi_0,\xi_1)$. Letting $\eta_0 \to +\infty$ shows that
\[ \forall \xi_1 \ge \xi_0 \ge A, \quad \q F(\xi_0, \xi_1) \le C(M). \]
Hence, with the monotonicity \eqref{eq:flux_add}, for $\xi_0 \ge A$, we can define
\[  \operatorname{\q F}(\xi_0) = \lim_{\xi_1 \to + \infty} \operatorname{\q F}(\xi_0,\xi_1) \ge 0. \]
Let $\xi_1 \to +\infty$ in the definition \eqref{def:qF} of $\q F(\xi_0,\xi_1)$ and derive
\[ \q F(\xi_0) = \operatorname{Flux}(\xi_0,\xi_0) - \q E. \]
Therefore, letting $\xi_0 \to +\infty$, we finally obtain that $\q F(\xi_0) \to 0$ as $\xi_0 \to +\infty$.\\

\begin{flushleft}\textbf{Step 2. Bound on $\ds \int_{\xi_0}^{\xi_1} \q B^2(\eta_0, \xi) d\xi$.} 
We now work in the domain 
\end{flushleft}
\[ K = \left\{ (\eta,\xi) \middle| \ \eta \ge 2T-A, \ A \le \xi \le \la' \eta \right\} \subset \{ (t,r) \mid t \ge T \text{ and } \lambda t \le r \le t-A \}. \]
Notice that when $(\eta,\xi) \in K$, $\ds |\psi(\eta,\xi)| \le \frac{\sqrt{2}}{2} \le \sqrt 2$, so that $\q A^2(\eta,\xi), \q B^2(\eta,\xi) \ge 0$: for such $(\eta, \xi)$ we can then define
\[ \A(\eta,\xi) = \sqrt{\q A^2(\eta, \xi)}, \quad \B(\eta,\xi)  = \sqrt{\q B^2(\eta, \xi)}. \]
We now use the identities from \eqref{en id global}, \eqref{en id 2}: in the variables $(\eta,\xi)$, they read
\begin{align*}
\partial_\xi \A^2 =  L, \quad \partial_\eta \B^2 = -  L.
\end{align*}

\begin{claim}
On $K$, one has $\ds |L| \le \frac{C}{r} \A \B$.
\end{claim}

\begin{proof}
Recall that on $K$, $\ds |\psi| \le \frac{\sqrt{2}}{2}$, so that $F(\psi) \ge 0$ and
\ant{
\abs{f ( \psi)} =  \abs{\psi( 1- \psi^2)} \le \abs{\psi},\, \, \, \abs{F(\psi)}  = \abs{\frac{\psi^2}{2}\left(1-\frac{ \psi^2}{2}\right)} \ge \frac{\psi^2}{4}.
}
Combining the above inequalities gives 
\ant{
f^2 (\psi) \le \abs{\psi}^2 \le 4 F(\psi), \quad \forall \,\,(t, r) \in K.
}
Then (using Cauchy-Schwarz inequality) and arguing as in~\eqref{L2 bu} we have
\ant{ 
L^2 &\le \frac{1}{2}( \psi_r^2- \psi_t^2)^2 +  \frac{4}{r^4} F^2( \psi) + \frac{64}{r^2} F( \psi) \psi_r^2 \\
& \le C \left[ \frac{1}{4}( \psi_r^2- \psi_t^2)^2 +  \frac{1}{r^4} F^2( \psi) + \frac{2}{r^2} F( \psi) (\psi_r^2+ \psi_t^2)\right],
}
which gives $L^2 \le C \frac{ \A^2 \B^2}{r^2}$.
The claim follows by taking the square root.
\end{proof}
We can thus conclude that 
\[ |\partial_\xi \A| \le C\frac{\B}{r}, \quad |\partial_\eta \B| \le C\frac{\A}{r}. \]
\begin{figure} \label{fig:4}
\begin{tikzpicture}[
	>=stealth',
	axis/.style={semithick,->},
	coord/.style={dashed, semithick},
	yscale = 1,
	xscale = 1]
	\newcommand{\xmin}{0};
	\newcommand{\xmax}{8};
	\newcommand{\ymin}{0};
	\newcommand{\ymax}{8};
	\newcommand{\ta}{3};
	\newcommand{\lamdba}{6};
	\newcommand{\fsp}{0.2};
	\draw [axis] (\xmin-\fsp,0) -- (\xmax,0) node [right] {$r$};
	\draw [axis] (0,\ymin-\fsp) -- (0,\ymax) node [below left] {$t$};
	\draw [axis] (0,0) -- (\xmax, \xmax) node [below right] {$\eta$};
	\draw [axis] (0,0) -- (-1,1) node [above right] {$\xi$};
	\draw [dashed] (0,0) -- (5/7,30/7);
	\draw [dashed] (0,1) -- (\ta-1,\ta);
	\draw [dashed]  (5/7,30/7) -- (0,2*\ta-1) node [left] {$\eta = 2T-A$};
	\draw [dashed] (-0.1, \ta) node [left] {$t=T$} -- (\xmax,\ta);
	\filldraw[color=light-gray1]  (\ta-1,\ta) -- (5/7,30/7) -- (7.5/\lamdba, 7.5) -- (\ymax-1.5,\ymax-0.5); 
	\draw [thick] (\xmax - 1.5,\xmax - 0.5) node [above right] {$\xi = A$} -- (\ta-1,\ta)  -- (5/7,30/7)  -- (7.5/\lamdba, \ymax-0.5) node [above] {$\xi = \la'\eta$}; 
	\draw [thick, densely dotted] (\xmax - 2,\xmax - 0.5) node [above] {$\xi_0$} -- (\ta-1,\ta+0.5) -- (\ta-1.5,\ta+1) node [above] {$\eta_0$} -- (\xmax - 3,\xmax - 0.5) node [above] {$\xi_1$} ;
	\draw [semithick, dashdotted] (\xmax - 3,\xmax - 1.5) -- (\xmax - 3.5,\xmax - 1) node [above] {$\eta_1$};
\end{tikzpicture}
\caption{ The quadrangle $K$ in gray and the rectangle of integration {$[\eta_0,\eta_1] \times [\xi_0,\xi_1]$}.}
\end{figure}
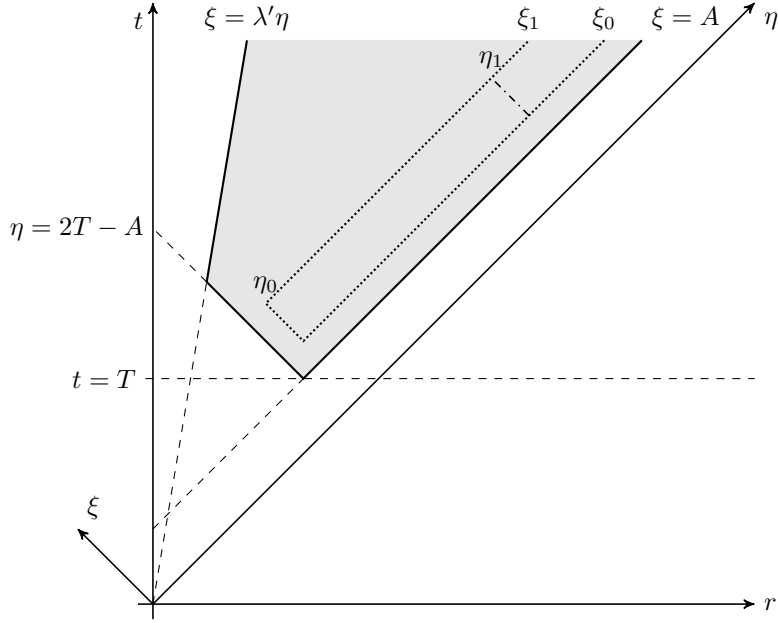
Consider a rectangle $[\eta_0, \eta_1] \times [\xi_0,\xi_1] \subset K$ (see Figure \ref{fig:4}), with $\eta_1$ meant to be large). In particular, for $0 < \la' < 1$ we have
\begin{gather} \label{rel:xi_eta}
0 \le \xi_0 \le \xi_1 \le \la' \eta_0 \le \la' \eta_1.
\end{gather}
Then for all $(\eta,\xi) \in [\eta_0, \eta_1] \times [\xi_0,\xi_1]$  we obtain 
\EQ{
\A(\eta, \xi) & \le \A(\eta,\xi_0) + C \int_{\xi_0}^{\xi} \frac{\B(\eta,\xi')}{\eta - \xi'} d\xi' ,\\
\B(\eta,\xi) & \le \B(\eta_1, \xi) + C \int_{\eta}^{\eta_1} \frac{\A(\eta', \xi)}{\eta' - \xi} d\eta' \\
 \le &\,   \B(\eta_1, \xi) + C \int_{\eta}^{\eta_1} \frac{\A(\eta',\xi_0)}{\eta' - \xi} d\eta' + C \int_{\xi_0}^{\xi} \int_{\eta}^{\eta_1} \frac{\B(\eta',\xi')}{(\eta' - \xi)(\eta' - \xi')} d\eta' d\xi' \label{eq:beta_bound}
}
Let us first evaluate the second term using the Cauchy-Schwarz inequality:
\begin{align}
\int_{\eta}^{\eta_1} \frac{\A(\eta',\xi_0)}{\eta' - \xi} d\eta' & \le \left( \int_{\eta}^{\eta_1} \A^2(\eta',\xi_0) d\eta' \right)^{1/2} \left( \int_{\eta}^{\eta_1} \frac{d\eta'}{(\eta' - \xi)^2} \right)^{1/2} \nonumber \\
& \le \sqrt{ \frac{\operatorname{\mr{Flux}}(\eta, \xi_0)}{\eta - \xi}} \le \sqrt{ \frac{\operatorname{\mr{Flux}}(\eta_0, \xi_0)}{(1- \la') \eta}} \label{eq:beta_2}
\end{align}
We now turn to the third term. It is convenient to denote
\[ \q F(\eta; \xi_0, \xi_1) = \int_{\xi_0}^{\xi_1} \q B^2(\eta,\xi) d\xi. \]
Notice that on the rectangle of integration, \eqref{rel:xi_eta} holds and we have 
\[ \frac{1}{(\eta' - \xi)(\eta'-\xi')} \le \frac{1}{(1-\la')^2 \eta'^2}. \]
Hence
\begin{align}
& \int_{\xi_0}^{\xi} \int_{\eta}^{\eta_1} \frac{\B(\eta',\xi')}{(\eta' - \xi)(\eta' - \xi')} d\eta' d\xi'  \le \frac{1}{(1-\la')^2} \int_{\eta} ^{\eta_1} \frac{1}{\eta'^2} \int_{\xi_0}^\xi \B(\eta',\xi') d\xi' d\eta' \nonumber \\
& \le \frac{\sqrt{\xi-\xi_0}}{(1-\la')^2} \int_{\eta} ^{\eta_1} \sqrt{\q F(\eta'; \xi_0,\xi)} \frac{ d\eta' }{\eta'^2}  \le \frac{\sqrt{\xi}}{(1-\la'^2)} \int_{\eta} ^{\eta_1} \sqrt{\q F(\eta'; \xi_0,\xi_1)} \frac{d\eta'}{\eta'^2}   \label{eq:beta_3}
\end{align}
Plugging the last two bounds \eqref{eq:beta_2} and \eqref{eq:beta_3} in the estimate for $\B(\eta,\xi)$ in \eqref{eq:beta_bound}, we infer that 
\[ \B(\eta,\xi)  \le \B(\eta_1, \xi) + C \sqrt{ \frac{\operatorname{\mr{Flux}}(\eta_0, \xi_0)}{(1- \la') \eta}} + \frac{C\sqrt{\xi}}{(1-\la'^2)} \int_{\eta} ^{\eta_1} \frac{\sqrt{\q F(\eta'; \xi_0,\xi_1)}}{\eta'^2} d\eta'. \]
Taking the square and integrating in $\xi$ over $[\xi_0,\xi_1]$ then yields
\begin{align*}
\q F(\eta; \xi_1,\xi_0)  & \le C \q F(\eta_1; \xi_1,\xi_0) + \frac{C}{(1-\la')} \frac{\xi_1-\xi_0}{\eta} \operatorname{\mr{Flux}}(\eta_0, \xi_0) \\
 & \qquad  + \frac{C(\xi_1-\xi_0)^2}{(1-\la')^2} \left( \int_{\eta} ^{\eta_1} \frac{\sqrt{\q F(\eta'; \xi_0,\xi_1)}}{\eta'^2} \right)^2 \\
& \le C \q F(\eta_1; \xi_1,\xi_0) + \frac{C\la' }{(1-\la')} \operatorname{\mr{Flux}}(\eta_0, \xi_0) \\
& \qquad + \frac{C \xi_1^2}{(1-\la')^2} \left( \int_{\eta}^{\eta_1} \q F(\eta'; \xi_0,\xi_1) \frac{d\eta'}{\eta'^2} \right) \left( \int_{\eta}^{\eta_1} \frac{d\eta'}{\eta'^2}  \right).
\end{align*}
We have obtained the following integral inequality for $\eta \mapsto \q F(\eta; \xi_1,\xi_0)$:
\[ \q F(\eta; \xi_1,\xi_0) \le C(\la') \left( \q F(\eta_1; \xi_1,\xi_0) + \operatorname{\mr{Flux}}(\eta_0, \xi_0) + \xi_1 \int_{\eta}^{\eta_1} \q F(\eta'; \xi_0,\xi_1) \frac{d\eta'}{\eta'^2} \right). \]
It follows from Gronwall's inequality  (in $\eta$) that
\begin{align*}  
 \q F(\eta; \xi_0, \xi_1) & \le C(\la') \left( \q F(\eta_1; \xi_0,\xi_1) + \operatorname{\mr{Flux}}(\eta_0, \xi_0) \right) \exp \left( C(\la') \xi_1  \int_\eta^{\eta_1} \frac{d\eta'}{\eta'^2} \right) \\
& \le C(\la')  \left( \q F(\eta_1; \xi_0,\xi_1) +  \operatorname{\mr{Flux}}(\eta_0, \xi_0) \right) .
\end{align*}
Letting $\eta_1 \to +\infty$ (and $\xi_0$, $\xi_1$ fixed and set $\eta = \eta_0$):
\[ \q F(\eta_0; \xi_0, \xi_1) \le  C(\la') \left( \q F(\xi_0, \xi_1) + \operatorname{\mr{Flux}}(\eta_0, \xi_0) \right). \]
As $\q F(\xi_0, \xi_1) \le \q F(\xi_0)$ we conclude that 
\[ \int_{\xi_0}^{\xi_1} \q B^2(\eta_0, \xi) d\xi = \q F(\eta_0; \xi_0, \xi_1) \le  C(\la') \left( \q F(\xi_0) + \operatorname{\mr{Flux}}(\eta_0, \xi_0) \right). \]

\begin{flushleft}\textbf{Step 3: Vanishing of the energy.}
Let $\e>0$. Let $\xi_\e \ge A$ be such that $0 \le \q F(\xi_\e)  \le \e$, then let $\eta_\e \ge \xi_\e$ be such that $\operatorname{\mr{Flux}}(\eta_\e, \xi_\e) \le \e$.
Define $R_\e = \xi_\e$ and $$\ds T_\e = \max \left( \frac{\xi_\e}{1+\lambda}, \frac{\eta_\e}{1-\lambda}, T \right).$$ 
Let $t \ge T_\e$. Denote $\xi = (1+\lambda)t \ge \xi_\e$ and $\eta = (1-\lambda)t \ge \eta_\e$. 
\end{flushleft}

We will integrate \eqref{en id global} on the triangle with vertices $(\eta, \xi_\e)$, $(2t-R_\e,\xi_\e)$ and $(\xi,\eta)$:
\[ \q T = \{ (\eta',\xi') \mid \eta' \ge \eta, \ \xi \ge \xi_\e, \eta' + \xi' \le \xi + \eta \}. \]
Then
\begin{align*}
0 & = \iint_{\q T} ( \partial_t (re) - \partial_r(rm)) drdt \\
& = \int_{\lambda t}^{t-R_\e} e(t,r) rdr - \int_{\eta}^{2t - R_\e} \q A^2(\eta',\xi_\e) d\eta' - \int_{\xi_\e}^{\xi} \q B^2(\eta, \xi') d\xi'.
\end{align*}
Therefore
\begin{align*}
 \int_{\lambda t}^{t-R_\e} e(t,r) rdr & = \int_{\eta}^{2t - R_\e} \q A^2(\eta',\xi_\e) d\eta' + \int_{\xi_\e}^{\xi} \q B^2(\eta, \xi') d\xi' \\
& \le  \operatorname{\mr{Flux}}(\eta,\xi_\e) +  \q F(\eta; \xi_\e, \xi) \\
& \le (C(\la')+1) \operatorname{\mr{Flux}}(\eta, \xi_\e) + C(\la') \q F(\xi_\e) \le (2C(\la')+1) \e.
\end{align*}
Hence,
\[ \limsup_{t \to + \infty} \int_{\lambda t}^{t-R_\e} e(t,r) rdr \le (2C(\la')+1) \e. \]
As $R_\e \ge A$ and $T_\e \ge T$, for all $R \ge R_\e$,
\[ \forall t \ge T_\e, \ r \in [\lambda t, t-R], \quad F(\psi(t,r)) \ge 0, \]
and so, for all $R \ge R_\e$,
\[  \limsup_{t \to + \infty} \int_{\lambda t}^{t-R} (|\partial_t \psi(t,r)|^2 + |\partial_r \psi(t,r)|^2) rdr \le (2C(\la')+1) \e. \]
This completes the proof of Proposition~\ref{prop:en_decay_global_1}.


\section{Proofs of Theorem~\ref{bu main} and Theorem~\ref{global main}}\label{mains}
In this section we establish Theorem~\ref{bu main} and Theorem~\ref{global main}. We remark that the details of this argument are very similar to  the arguments contained in~\cite[Section $5$]{CKLS1},~\cite[Section $3$]{CKLS2}, as well as~\cite[Section $5$]{DKM1}, \cite[Section $4$]{DKM3}  and thus we will only present a very brief sketch.  We begin with the case of finite time blow-up. 

\subsection{Proof of Theorem~\ref{bu main}}
 We assume that $\vec u(t) \in \HH$ is a smooth type-II solution to~\eqref{u eq} with $T_+( \vec u)=1$ and let $\vec v(t)$ be the regular part as defined in Section~\ref{reg part} and let $\vec a(t) = \vec u(t)- \vec v(t)$ be the singular part as defined in~\eqref{a def}. 
 
The first step in the proof of Theorem~\ref{bu main} is to use Theorem~\ref{self sim} to show that there exists a sequence of times $\{t_n\} \to 1$ so that the time derivative $a_t(t_n) = u_t(t_n)- v_t(t_n)$ tends to zero in $L^2$. We first prove this in a averaged sense for the blow-up solution $\vec u(t)$. 
\begin{lem} Let $\vec u(t) \in \HH$ be a smooth type-II solution to~\eqref{u eq} with $T_+( \vec u) = 1$. Then 
\EQ{ \label{av time dec}
 \frac{1}{1-t} \int_t^1 \int_0^{1-s} u_t^2(s, r) r^3 \, dr \to 0 \mas t \to 1.
 }
 \end{lem}
\begin{proof} This is a direct consequence of Theorem~\ref{self sim}. In fact, it is convenient to pass to the $2d$ formulation, and use Proposition~\ref{prop1}. We again set 
\ant{
\psi(t, r) = r u(t, r)
}
Writing~\eqref{av time dec} in terms of $\psi$ we need to prove that 
\EQ{\label{psi t dec}
\frac{1}{1-t} \int_t^1 \int_0^{1-s} \psi_t^2(s, r) r \, dr \to 0 \mas t \to 1.
 }
 Then~\eqref{psi t dec} can be deduced exactly as in \cite[Corollary~2.3]{STZ92}, by integrating~\eqref{psi t id} over the region $\{(s, r) \mid t \le s <1, \,,\, 0 \le r \le 1-s\}$, dividing by $1/ (1-t)$ and using Proposition~\ref{prop1} and the vanishing of the flux. We refer the reader to~\cite{STZ94} or the book,~\cite[Chapter 8]{SSbook} for the precise details. 
\end{proof}
Next, one can observe that since $\vec v(t)$ is regular at $t=1$,~\eqref{av time dec} holds for $a(t)$ as well, namely, 
\EQ{\label{a av dec}
 \frac{1}{1-t} \int_t^1 \int_0^{1-s} a_t^2(s, r) r^3 \, dr \to 0 \mas t \to 1.
 }
 Following~\cite{DKM1}, we can then deduce the following lemma as a consequence of~\eqref{a av dec}. We refer the reader to~\cite{DKM1} for the proof. 
 \begin{lem}\cite[Corollary $5.3$]{DKM1}\label{sig lem} There exist a sequence $t_n \to 1$ so that for every $n$ and for every $\s \in (0, 1-t_n)$ we have 
 \begin{align} 
 &\lim_{n \to \infty} \frac{1}{ \s} \int_{t_n}^{t_n + \s} \int_0^{\infty} a_t^2(t_n, r) \, r^3 \, dr \, dt =0,\label{sig}\\
 &\lim_{n \to \infty} \int_0^{\infty} a_t^2(t_n, r) \, r^3 \, dr =0\label{a t dec}.
 \end{align}
 \end{lem}
Now consider the bounded sequence $\vec a(t_n) \in \HH$. By Theorem~\ref{BG} and~\eqref{a t dec} we have a profile decomposition 
\ant{
 &a(t_n, r) = \sum_{j=1}^{k}  V_{L, n}^j(0, r) + w_{n, 0}^k(r),\\
 &a_t(t_n, r) = o_{\HH}(1).
 }
 By the argument in~\cite[Section $5.2$]{DKM1} we can again use~\eqref{a t dec} to conclude that any nonzero profile must in fact be either $W$ or $-W$ and there can only be finitely many of these. Indeed, we obtain the following preliminary result
 \begin{prop} \cite[Proposition~$5.1$]{DKM1} \cite[Corollary~$4.1$]{DKM3} \label{prelim dec}There exists an integer $J_0>0$,  and sequences of scales $\la_{j,n}$  for $1 \le j \le J_0$ with 
 \ant{ 
 \la_{1, n} \ll  \dots \ll \la_{ J_0,n} \ll 1- t_n,
 } 
and a sequence of signs $\io_j \in \{+1, -1\}$ so that 
 \EQ{\label{a W w}
 \vec u(t_n) =  (v_0, v_1) + \sum_{j=1}^{J_0} \left( \frac{\io_j}{  \la_{j,n}}W(\cdot/ \la_{j,n}), 0\right) + (w_n, 0),
 }
 where the linear evolution $\vec w_{L, n}(t) = S(t) (w_n, 0)$ satisfies 
 \EQ{ \label{wn strich}
 \lim_{n \to \infty} \| w_{L, n}\|_{S( \R)} = 0.
 }
 \end{prop}
In order to establish Theorem~\ref{bu main} it remains to strengthen~\eqref{wn strich} by showing that the error $(w_n, 0)$ tends to zero in the energy space $\HH$. In particular we establish the following proposition
\begin{prop}\cite[Proposition 5.6]{CKLS1} \label{prop w to 0}Let $(w_n, 0)$ be as in~\eqref{a W w},~\eqref{wn strich}. Then, 
\EQ{\label{wn to 0}
\|(w_n, 0)\|_{\HH} \to 0 \mas n \to \infty.
}
\end{prop}
It is in the proof of Proposition~\ref{prop w to 0} that the exterior energy estimates for the free equation, Proposition~\ref{lin ext estimate}  enter crucially. The proof of~\eqref{wn to 0} is identical to the argument in~\cite[Proof of Proposition $5.6$]{CKLS1} or \cite[Proposition $6.1$]{Cote13} and has its roots in the argument in  \cite[Proposition $3.4$]{DKM2}.

The argument goes by contradiction. The key idea is to use that the free wave $\vec w_n(t)$ with initial data $(w_n,0)$ actually maintains a fixed amount of energy outside the light cone (Proposition \ref{lin ext estimate}) . We prove that this forces $\vec u$ to concentrate energy on the boundary of the cone. For this, we proceed in two steps for each profile, both requiring evolving a nonlinear profile decomposition backwards in time. First, we show that the evolutions of $\vec w_n(t)$ and $\vec u(t)$ remain close on an exterior region during a time-scale on which we can control the first profiles (by means of Proposition \ref{nonlin profile}).
At this point, we focus the analysis outside the light cone: we need to evolve the decomposition past the time-scale on which we can control the first profile, but fortunately this large profile does not contribute in this exterior region. In fact, we evolve the profile decomposition with the first profile removed, exterior to the cone, up to the time scale of the \emph{second} profile, and infer that some energy remains outside the light cone. Arguing similarly for every profile, we conclude that some energy remains outside the light cone for \emph{all} times (in fact it concentrates on the boundary). Unscaling this information, we see that $\vec u(t)$ must concentrate some energy at some point $r_0 >0$ and time $t_0 = 1-r_0 <1$, which is a contradiction with our assumption that the blow-up time $T^+(\vec u)=1$.

We refer the reader to the previously mentioned references for the technical details of the argument. We also note that the energy quantization follows from the orthogonality of profiles~\eqref{nonlin en dec}. This completes our sketch of the proof of Theorem~\ref{bu main}. 

\subsection{Proof of Theorem~\ref{global main}}
We assume that $\vec u(t) \in \HH$ is a smooth,  type-II solution to \eqref{u eq} defined globally for positive times. We also assume that $\vec u(t)$ does not scatter to zero as $t \to \infty$. Let $\vec v_{L}(t) \in \HH$ be the radiation term constructed in Section~\ref{rad term} and denote by $\vec v(t)$ the nonlinear profile associated to $\vec v_{L}(t)$ as defined in Section~\ref{nonlinear prof}, i.e., $\vec v(t) \in \HH$ is the unique solution to~\eqref{u eq} so that 
\EQ{\label{v-vL}
\| \vec v(t)- \vec v_L(t)\|_{ \HH} \to 0 \mas t \to \infty.
}
We then set 
\ant{
\vec a(t) = \vec u(t) - \vec v(t).
}

The proof of Theorem~\ref{global main} follows the same general outline as the proof of Theorem~\ref{bu main} and is in fact very similar at this point to the argument in~\cite[Section $3$]{CKLS2} or~\cite[Section $4$]{DKM3}. 

Using Theorem~\ref{thm:en_decay_global} one can argue as in~\cite[Proof of Corollary $2.2$]{CKLS2} to deduce the following lemma.
\begin{lem}\label{global av dec} Let  $\vec u(t) \in \HH$ be a smooth,  type-II solution to \eqref{u eq} defined globally for positive times. Then 
\EQ{\label{ut av dec}
\limsup_{T \to \infty} \frac{1}{T} \int_0^T \int_0^{T-A} u_t^2(t, r) \, r^3 \, dr \to 0 \mas A \to \infty.
}
\end{lem}
We again refer the reader to \cite[Proof of Corollary $2.2$]{CKLS2} for the proof of Lemma~\ref{global av dec}. As in the proof of Lemma~\ref{av time dec} the argument consists of rewriting~\eqref{ut av dec} in terms of $\psi = ru$ and integrating~\eqref{psi t id} over the region of integration in~\eqref{ut av dec} and then using Theorem~\ref{thm:en_decay_global}  to conclude. 

As in the blow-up argument, the next step is to use Lemma~\ref{global av dec} to identify a sequence of times for which the $L^2$ norm of $a_t$ tends to zero. One begins by deducing the analog of Lemma~\ref{sig lem} for global solutions.  Using Corollary~\ref{u-vL} we can argue as in~\cite[Lemma~$4.4$]{DKM3} or as in~\cite[Lemma $3.3$]{CKLS2}  to prove the following result. 
\begin{lem}\cite[Lemma $4.4$]{DKM3}\label{global sig lem} 
There exists a sequence of times $t_n \to \infty$ so that 
\EQ{
&\lim_{t \to \infty} \sup_{ \s >0}\frac{1}{\s} \int_{t_n}^{t_n + \s} \int_0^{\infty} a_t^2(t, r) \, r^3 \, dr \, dt = 0,\\
&\lim_{n \to \infty}  \int_0^{\infty} a_t^2(t_n, r) \, r^3 \, dr = 0.
}
\end{lem}
We note that here we have stated Lemma~\ref{global sig lem} in terms of $\vec a(t) = \vec u(t) - \vec v(t)$ as opposed to for $\vec u(t) - \vec v_{L}(t)$ as in \cite{DKM3}. However, due to~\eqref{v-vL} this distinction makes no difference. 

Next, we can establish the global analog of Proposition~\ref{prelim dec}. 
 \begin{prop}\label{global prelim dec} \cite[Corollary~$4.2$]{DKM3} There exists and integer $J_0>0$,  and sequences of scales $\la_{j,n}$  for $1 \le j \le J_0$ with 
 \ant{ 
 \la_{1, n} \ll  \dots \ll \la_{ J_0, n} \ll  t_n,
 } 
and a sequence of signs $\io_j \in \{+1, -1\}$ so that 
 \EQ{\label{a W w global}
 \vec u(t_n) =  \vec v_L(t_n) + \sum_{j=1}^{J_0} \left( \frac{\io_j}{  \la_{j,n}}W(\cdot/ \la_{j,n}), 0\right) + (w_n, 0),
 }
 where the linear evolution $\vec w_{L, n}(t) = S(t) (w_n, 0)$ satisfies 
 \EQ{ \label{wn strich global}
 \lim_{n \to \infty} \| w_{L, n}\|_{S( \R)} = 0.
 }
 \end{prop}
Again, the main idea in the proof of Proposition~\ref{global prelim dec} is to use Lemma~\ref{global sig lem} to show that any nonzero profile must be either $W$ or $-W$ and we refer the reader to~\cite[Proof of Corollary~$4.2$]{DKM3} for the proof. 

Finally, the proof of Theorem~\ref{global main} is completed proving the analog of Proposition~\ref{prop w to 0} in the global setting. 

\begin{prop}\cite[Proposition $3.12$]{CKLS2}\label{global w to 0} Let $(w_n, 0)$ be as in~\eqref{a W w global},~\eqref{wn strich global}. Then
\begin{gather} \label{wn to 0 g}
 \|  (w_n, 0)\|_{\HH} \to 0 \mas n \to \infty.
\end{gather}
 \end{prop}

The proof of this result follows the same scheme as in Proposition \ref{prop w to 0} and the exterior linear estimates for $4d$ free waves (Proposition~\ref{lin ext estimate}) plays a crucial role here.
For the details of this compactness argument, we refer the reader to \cite[Proof of Proposition~$3.12$]{CKLS2}. 
The energy quantization again follows from the orthogonality of profiles~\eqref{nonlin en dec}. This completes our sketch of the proof of Theorem~\ref{global main}.



\section{Type-II blow-up below $2 \| \na W\|_{L^2}^2$}\label{2W bu}

This section is devoted to proving Theorem~\ref{bu 2W} and we will assume throughout that $\vec u(t)$ is a smooth type-II solution with $T_+(\vec u) = 1$. Moreover suppose that 
\EQ{\label{2W}
\sup_{0 \le t < 1-t} \|  u(t) \|_{\dot{H}^1( 0 < r < 1-t)}^2 < 2 \| \na W\|_{L^2}^2 =2 \|(W, 0)\|_{\HH}^2.
 }
 We again denote the regular part of $\vec u(t)$ by $\vec v(t)$, and the singular part by $\vec a(t):= \vec u(t)- \vec v(t)$, as defined in Section~\ref{reg part}. We also recall that $\supp \vec a(t) \in B(0, 1-t)$ and that $\vec a(t) \rightharpoonup 0$ in $\HH$. 
 
 By Theorem~\ref{bu main} we know that there exists a sequence of times $t_n \to 1$, an integer $J_0 \ge 1$ scales $\la_{j,n} \ll 1-t_n$ and signs $\io_j$ for $1 \le j \le J_0$ so that  
 \EQ{\label{a dec}
 &\vec a(t_n) =  \sum_{j = 1}^{J_0} \left(\frac{ \io_j}{ \la_{j,n}}W\left( \frac{\cdot}{ \la_{j,n}}\right), \, 0\right) + o_{\HH}(1) \mas n \to \infty,\\
 &\la_{1,n} \ll \cdots \ll \la_{ J_0, n} \ll 1-t_n.
 } 
Using~\eqref{2W}, Lemma~\ref{lem:hardy_loc}, and the definition of $\vec a(t)$ we have 
\EQ{
\| \vec a(t_n) \|^2_{\HH} < 2\|\na W\|_{L^2}^2
} 
for $n$ large. Combining this with the orthogonality of the scales $\la_{j,n}$, it is clear that there can only be one profile above, i.e., $J_0=1$.  Moreover, by replacing $u$ by $-u$ if necessary we can assume $\io =1$. Thus,~\eqref{a dec} reduces to
\EQ{\label{an dec}
&\vec a(t_n) =   \left(\frac{ 1}{ \la_{n}}W\left( \frac{\cdot}{ \la_{n}}\right), \, 0\right) + o_{\HH}(1) \mas n \to \infty,\\
& \la_n \ll 1-t_n.
}
We break up the proof of Theorem~\ref{bu 2W} into several steps. 
\subsection{Step 1: preliminary observations on a profile decomposition}\label{prof pre}
In order to prove Theorem~\ref{bu 2W} we need to show that the decomposition~\eqref{an dec} holds for \emph{any} sequence $\tau_n \to 1$. Let $\tau_n \to 1$ be any such  sequence. Up to passing to a sequence we can use Theorem~\ref{BG} to find a profile decomposition 
\EQ{\label{an prof}
\vec a(\tau_n) =  \sum_{j=1}^J  \vec U_{L, n}^j + \vec w_{n}^J,
}
where 
\ant{
\vec U_{L, n}^j(t, r)= \left( \frac{1}{\la_{j,n}} U_L\left( \frac{t-t_{j,n}}{ \la_{j,n}}, \frac{r}{\la_{j,n}}\right), \,\frac{1}{\la_{j,n}^2} \p_tU_L \left( \frac{t-t_{j,n}}{ \la_{j,n}}, \frac{r}{\la_{j,n}}\right)\right).
}
As usual we denote the nonlinear profile associated to $\vec U_{L}^j$ by $\vec U^j$. We can also assume, via an application of Lemma~\ref{lem:3.7} that the profiles are \emph{pre-ordered} as in Definition~\ref{re order} with 
\[ \forall i \le j, \quad \{\vec U_{\mr L}^i, \lambda_{i,n}, t_{i,n}\} \preccurlyeq \{\vec U^j_{\mr L}, \lambda_{j,n}, t_{j,n}\}. \]
Note that we can also view~\eqref{an prof} as a profile decomposition for $\vec u(\tau_n)$ given the definition of $(v_0, v_1)$ as the weak limit of $\vec u(t)$ in $\HH$ as $t \to1$. Indeed we can view $\vec v(\tau_n)$, up to an $o_{\HH}(1)$ error, as a profile $\vec U_{L}^0$ with initial data $(v_0, v_1)$ and parameters $\la_{n, 0} = 1$, $t_{n, 0} =0$ and nonlinear profile equal to $\vec v(t, r)$ and we write 
\begin{gather}
\vec u(\tau_n) = \vec v(\tau_n) + \sum_{j=1}^J  \vec U_{L, n}^j + \vec w_{n}^J.
\end{gather}
Note that given the support properties of $\vec a(t)$ we must have $\abs{t_n^j} \le C (1-\tau_n)$ and $\la_{j,n} \le C(1-\tau_n)$ for all $n, j$, by Lemma~\ref{lem:2.5}. 

We  observe that given the fact that $\vec u(t)$ blows up at $t=1$ and that $\vec v$ is regular at $t=1$, at least one of the nonlinear profiles $\vec U^j$ with $j \ge 1$ does not scatter in forward time.  Given our pre-ordering this means that the nonlinear profile $\vec U^1$  does not scatter in forward time. In fact, we claim that
\EQ{\label{order 1}
 \{ U_L^1, \la_{1,n}, t_{1,n}\}\prec    \{U_L^0, 1, 0\},
  }
where again $\vec U_L^0$ is the profile with initial data $(v_0, v_1)$. Indeed, since $\vec U^1$ does not scatter in forward time we would need 
\EQ{\label{time order}
T< T_+(v_0, v_1), \Longrightarrow  \lim_{n \to \infty} \frac{ T - t_{1, n}}{ \la_{1,n}} < T_+( \vec U^1) < \infty,
}
where $T_+(v_0, v_1)$ is computed from the evolution starting at $t=1$. Since $\vec v(t)$ exists in a neighborhood of $t=1$, we can simply choose any $T>0$ with $T< T_+( v_0, v_1)$. We know that  $\abs{t_{1,n}} \le C(1-\tau_n) \to 0$ and $\la_{1,n} \le C(1-\tau_n) \to 0$. This means that $T- t_{1,n}>0$ for $n$ large enough and hence 
\ant{
 \frac{ T - t_{1,n}}{ \la_{1,n}} \to \infty \mas n \to \infty,
 } 
 which renders~\eqref{time order} impossible and proves~\eqref{order 1}. 
 
 Next, note that by the orthogonality of the free energy in our decomposition, i.e.,~\eqref{free en dec}, we must have 
\begin{gather}
\| \vec U_L^j(-t_{j,n}/ \la_{j,n}) \|_{\HH}^2 < 2 \|( W, 0)\|_{\HH}^2, \quad \text{and} \quad \| \vec w_n^J\|_{\HH}^2< 2\|(W, 0)\|_{\HH}^2.
\end{gather}
for $n$ large and $j  \ge 1$. By Lemma~\ref{E<2W} we can then deduce that the nonlinear energies
\begin{gather} \label{pos en}
E( \vec U_{L, n}^j(0))  \ge 0, \quad \text{and} \quad E( \vec w_n^J) \ge 0.
\end{gather}
for $n$ large enough. Moreover, if 
\begin{gather}
E( \vec U_{L, n}^j(0)) \to 0 \mas n \to \infty,
\end{gather}
then 
\ant{
\|  \vec U_{L, n}^j(0)\|_{\HH}  \to 0 \mas n \to \infty,
}
and since the $\HH$ norm is preserved by the linear flow this means that $\vec U_{L}^j \equiv(0, 0)$. Similarly, 
\ant{
E( \vec w_n^J)  \to 0 \Longrightarrow \| \vec w_n^J\|_{\HH} \to 0 .
}
Finally, if $\vec U^j$ is the nonlinear profile associated to $\{ U_{L}^j, \la_{j,n}, t_{j,n}\}$ then either $E( \vec U^j) >0$ or $ \vec U^j =(0, 0)$. 

Note that since $\vec u( \tau_n) \rightharpoonup (v_0, v_1)$ weakly in $\HH$ (\cite[Section 3]{DKM1}), by the construction of the profiles in \cite{BG}, $(v_0, v_1)$ with parameters $t_{n, 0} =0$ for the  time translations and $\la_{n, 0} = 1$ for the scaling, and nonlinear profile (with evolution starting at $t=1$) $\vec v(t)$, occurs in the profile decomposition of $\vec u(\tau_n)$. Thus, the previous situation is the general one for a profile decomposition of $\vec u(\tau_n)$, just as in Claim~\ref{cl:v_in_prof} below. 

\subsection{Step 2: Minimization process and consequences} Here we use the minimization process for profile decompositions of $\vec u(\tau_n)$ developed in~\cite{DKM6} adapted to the current situation. We begin by introducing some of the notation from \cite[Section~$4$]{DKM6}. First let $\cS_0$ denote the set of all sequences $\{\tau_n\} \to 1$ so that $\vec u(\tau_n)$ admits a $\precc$-ordered profile decomposition. Note that up to extracting subsequences, $\cS_0$ consists of all sequences $\tau_n \to 1$. 

Let $\q T  = \{\tau_n\}_{n \in \N} \in \cS_0$. Denote by 
\EQ{
J_0( \q T) = \textrm{\# of profiles of} \,  \, \vec u( \tau_n)\, \, \textrm{that do not scatter in forward time. }
}
This means that for $j \le J_0( \q T)$, $\vec U^j$ does not scatter in forward time and for $j \ge J_0( \q T) + 1$, $\vec U^j$ scatters in forward time. 

Since $\vec u(t)$ blows up at time $t=1$ we know that for any $\q T \in \cS_0$ we must have $J_0( \q T) \ge 1$. On the other hand, by the small data theory, there is a $\de_0>0$ so that if  $\| \vec U_L^j\|_{\HH} \le \de_0$ then a nonlinear profile, $\vec U^j$ associated to $\vec U_L$ must scatter in both time directions. Since we are also assuming that $\vec u(t)$ is a type-II solution, i.e., $$\sup_{t \in [0, 1)} \| \vec u(t) \|_{\HH} \le M < \infty,$$ we can use the almost orthogonality of the $\HH$ norms of the profiles,~\eqref{free en dec} to conclude that $J_0(\q T) \le CM/\delta_0^2$ is finite and uniformly bounded on $\cS_0$.

Next, define 
\EQ{
J_1( \q T):=  \min\{ j \ge 1 \mid j \prec j+1\},
}
where $\prec$ is the \emph{strict} order introduced in Definition~\ref{re order}. Since we have $J_0(\T) \prec J_0( \T) + 1$ we can conclude that $J_1( \T) \le J_0(\T)$ and hence $J_1( \T)$ is uniformly bounded on $\cS_0$ as well. 

We also define 
\EQ{ \label{JM S1}
&J_{\mr{max}}   = \max \{ J_0(\q T) \mid \q T \in \cS_0 \}, \\
&\cS_1  = \{ \q T \in \cS_0 \mid J_0(\q T)=  J_{\mr{max}} \}.
}
For $\T \in \cS_1$ we then define the non-scattering energy $\EE(\T)$, as the sum of the energies of the nonlinear profiles that do not scatter, in particular for $\T \in \cS_1$ we set
\EQ{
\EE(\T) = \sum_{j=1}^{J_{\max}} E( \vec U^j).
}
We now recall a result proved in~\cite{DKM6}.  
\begin{claim}\cite[Corollary $4.3$ and Lemma $4.5$]{DKM6} The infimum of $\EE(\T)$ is attained (and hence is a minimum): i.e., there exists $\T_0 \in \cS_1$ so that 
\EQ{
\E(\T_0) = \inf\{ \E(\T) \mid \T \in \cS_1\} =: \E_{\min}.
}
\end{claim}
With the above claim we can define 
\begin{align} \label{S2}
&\cS_2  = \{ \q T \in  \cS_1 \mid \q E(\q T) =  \q E_{\mr{min}} \} \ne  \varnothing, \\
&J_{\mr{min}}  =  \min \{ J_1(\q T) \mid \q T \in \cS_2 \}, \\
&\cS_3  = \{ \q T \in \cS_2 \mid J_1(\q T) = J_{\mr{min}}  \} \ne \varnothing.
\end{align}
We remark that in this radial setting, we necessarily have $J_{\min} =1$. This follows from the following lemma proved in~\cite{DKM6}. 
\begin{lem}\cite[Lemma~$4.12$]{DKM6} \label{lem412}
There exists $\q T_0 \in \cS_3$ such that for all $j=1, \dots, J_{\mr{min}}$, $\vec U^j \in \{ (\pm W,0) \} $ and hence $J_{\min} =1$. 
\end{lem}
This above is a much simplified version of \cite[Lemma 4.12]{DKM6}: as we are working in the radial setting, the only stationary solutions to \eqref{u eq} are $(\pm W,0)$. The result in~\cite[Lemma~$4.12$]{DKM6} states that all of the parameters $\la_{j,n} = \la_{1,n}$ for $1 \le j \le J_{\min}$, but this forces $J_{\min}=1$ by orthogonality of the parameters. 

To proceed, we distinguish between two cases: 
\begin{itemize}
\item[$(a)$] The nonlinear profile associated to $(v_0, v_1)$, namely $\vec v(t)$, scatters in forward time.
\item[$(b)$] The nonlinear profile associated to $(v_0, v_1)$, namely $\vec v(t)$, does \emph{not} scatter in forward time.
\end{itemize}
\begin{claim} \label{a b} In case $(a)$ above we have $J_{\max} = 1$ and $\EE_{\min} \ge E(W, 0)$. In case $(b)$ we have $J_{\max} = 2$ and $\EE_{\min} \ge E(W, 0) + E(v_0, v_1)$. 
\end{claim}
\begin{proof} Choose the sequence $\T_0 = \{ \tau_n\}_{n \in \N}$ given by Lemma~\ref{lem412}. Since $J_{\min} = 1$ we have $\vec U_L^1 = ( \pm W, 0)$. We have $\T_0 \in \cS_3 \subset \cS_2 \subset \cS_1$ and hence we have $J_{\max}$ non-scattering profiles and 
\ant{
\EE_{\min} = \sum_{j=1}^{J_{\max}} E(\vec U^j).
}
Also, recall that for the sequence $\{t_n\}$ given by Theorem~\ref{bu main} we have 
\EQ{
\vec a(t_n, r) = \left( \frac{1}{\la_n} W\left( \frac{r}{\la_n}\right), 0\right) + o_{\HH}(1) \mas n \to \infty.
}
Recalling that 
\EQ{
\lim_{t \to 1} E( \vec a(t)) = E(\vec u) - E(v_0, v_1),
}
and by considering the sequence $t_n \to 1$ we have 
\EQ{\label{u=W v}
E(\vec u) = E(W, 0) +  E(v_0, v_1) .
}
Now consider the Pythagorean expansion for the sequence $\T_0 = \{ \tau_n\}$ give by Lemma~\ref{lem412}. Using the earlier established fact~\eqref{pos en} we know that all of the nonzero profiles, as well as $\vec w_n^J$ have positive energy. By the definition of $\EE_{\min}$, and the fact that $\vec U^1 = (\pm W, 0)$, we see that in case $(a)$ we have $\EE_{\min} \ge E(W, 0)$, and in case $(b)$ we have $\EE_{\min} \ge E(W, 0) + E(v_0, v_1)$. To prove the statements about $J_{\max}$ we will use~\eqref{u=W v} and the positivity of the energies of the profiles. Indeed, 
\ant{
E( \vec u)&= \sum_{j=1}^{J_{\max}} E(\vec U^j) + \sum_{j= J_{\max} +1}^J E( \vec U^j) + E( \vec w_n^J) + o_n(1)\\
&= E(W, 0) + \sum_{j=2}^{J_{\max}} E(\vec U^j) + \sum_{j= J_{\max} +1}^J E( \vec U^j) + E( \vec w_n^J) + o_n(1).
}
Using~\eqref{u=W v} we obtain
\ant{
E(v_0, v_1) =  \sum_{j=2}^{J_{\max}} E(\vec U^j) + \sum_{j= J_{\max} +1}^J E( \vec U^j) + E( \vec w_n^J) + o_n(1).
}
In  case $(a)$ we assume that $\vec v(t)$ scatters and hence corresponds to one of the nonlinear profiles $\vec U^j$ with $j  \ge J_{\max}+1$. Canceling $E(v_0, v_1)$ from both sides we have 
\ant{
0 &=  \sum_{j=2}^{J_{\max}} E(\vec U^j) + \sum_{j= J_{\max} +1, \, U^j \neq v}^J E( \vec U^j) + E( \vec w_n^J) + o_n(1)\\
& \ge \sum_{j=2}^{J_{\max}} E(\vec U^j) + o_n(1),
}
 since all the  profiles above have positive energy. Hence $J_{\max} = 1$. In case $(b)$ one similarly shows that $J_{\max}=2$. 
\end{proof}

\subsection{Step 3: Compactness of the singular part, $\vec a(t)$} 
We prove the following result. 
\begin{lem} For any sequence $\tau_n  \to 1$, there exists a subsequence, still denoted by $\tau_n$, and scales $\la_n>0$ so that  $(\la_n a(\tau_n, \la_n r), \la_n^2 a_t(\tau_n, \la_n r))$ converges in $\HH$. 
\end{lem}
\begin{proof}
Take an arbitrary sequence  $\tau_n \to 1$ which we assume, after passing to a subsequence and reordering, that $\{\tau_n\} \in \cS_0$ so that the profile decomposition for $\vec u(\tau_n)$ is pre-ordered. We summarize what we have established in the previous subsections. We know that $(v_0, v_1)$ is a profile and that either $J_{\max} = 1$ or $J_{\max}=2$ depending on whether or not, $\vec v(t)$ scatters in forward time, i.e., whether we are in case $(a)$ or $(b)$. We also know that the first profile $\vec U_L^1$ does not scatter in forward time and that $\vec U_{L}^1 \prec (v_0, v_1)$. Further, all of the profiles other than $(v_0, v_1)$ have positive nonlinear energy and so does $\vec w_n^J$. 

\begin{claim}All of the profiles that scatter in forward time must be identically $0$ and  the error 
\EQ{
\vec w_n^J \to 0 \quad  \textrm{in } \HH .
}
\end{claim}
\begin{proof} We again rely on the positivity of the nonlinear energies. Since we know that $J_{\max} =1 $ or $J_{\max} =2$ we know that $\{\tau_n\} \in \cS_1$. Thus in case $(a)$ we have 
\ant{
E( \vec u) &= E(W, 0)+ E(v_0, v_1) = E(U^1) + \sum_{j =2}^{J}E(\vec U^j)  + E( \vec w_n^J) + o_n(1)\\
& \ge \EE_{\min} + E(v_0, v_1) + \sum_{j =2}^{J}E(\vec U^j)+E( \vec w_n^J) + o_n(1)\\
& \ge E(W, 0)+ E(v_0, v_1) + o_n(1).
}
This proves the claim in case $(a)$. The same proof applies in case $(b)$. 
\end{proof}

\begin{claim} The profile $\vec U^1$ cannot scatter in backwards time. 
\end{claim} 
\begin{proof} Suppose that $ \vec U^1$ scatters as $t \to - \infty$. Then, the nonlinear profile decomposition Proposition~\ref{nonlin profile} gives (for all $t<0$ so that $\vec v(1+t)$ is defined, .i.e., for all $t \in (-T, 0$ for some fixed $T>0$)) for $n$ large 
\ant{
\vec u( \tau_n+ t) = \vec v( \tau_n + t) + \vec U^1_n(t) + \vec w_{L, n}^J(t) + \vec \eta_n^J(t),
  }
  where both $\| \vec w_{n, L}^J(t)\|_{\HH}$ and $\| \vec \eta_n^J\|_{\HH}$ are small for $t>-T$, $t \le 0$. Note that since $\vec U^1$ does not scatter in forward time, for $t \in (-T, 0]$ we have 
  $\| \vec U^1(t)\|_{\HH} \ge \de_0>0$. Choosing $t_0$ close to $1$ we then evolve the profile decomposition for time $s_n = t_0- \tau_n$  which gives 
  \ant{
  \vec u(t_0) - \vec v(t_0) = \vec U^1_n(t_0- \tau_n) + o_n(1),
  } 
  which is a nontrivial profile decomposition  for the fixed function $\vec u(t_0) - \vec v(t_0)$. This means that necessarily we must have $t_{1, n} = 0$ and $\la_{1, n} = 1$ for all $n$. But we have already observed in Section~\ref{prof pre} that we must have $\la_{1,n} \to 0$ as $n \to \infty$. Hence we have arrived at a contradiction and thus $\vec U^1$ does not scatter in backwards time. 
\end{proof}
Since $\vec U^1$ does not scatter backwards or forwards in time, we then have that $\abs{\frac{- t_{1,n}}{ \la_{1,n}}} \le C_0< \infty$. Hence we can assume without loss of generality that $t_{1,n} = 0$ for all $n$. We now have that 
\EQ{
\vec a_n(\tau_n, r) = \left( \frac{1}{ \la_n} U^1\left(0, \frac{r}{ \la_n}\right), \frac{1}{ \la_n^2} U_t^1\left( 0, \frac{r}{\la_n} \right) \right) + o_{\HH}(1),
}
which proves the desired compactness result for $\vec a(t)$. 
\end{proof}

\subsection{Step 4: Conclusion of the proof of Theorem~\ref{bu 2W}} \label{step 4 th 1.2}

Let $\{t_n\}_n$ be any sequence with  $t_n \to +\infty$. From Step 3, we have a function $\lambda$ such that $K(\vec a, \lambda)$ has compact closure in $\HH = \dot H^1 \times L^2$. Hence, after passing to a subsequence still denoted $\{t_n\}_n$, the sequence
\[ \left( \lambda(t_n) a \left( t_n, \lambda(t_n) \cdot \right),  \lambda(t_n)^2 \partial_t a \left( t_n, \lambda(t_n) \cdot \right) \right) \]
converges in $\HH$ to some $(U_0, U_1) \in \HH$; denote $\vec U(t)$ the nonlinear solution to \eqref{u eq} with initial data $\vec U(0) = (U_0, U_1)$. By \cite[Lemma 8.5]{DKM1}, we have the following claim. 
\begin{claim}\label{U comp}\cite[Lemma 8.5]{DKM1}
$\vec U$ has the compactness property on  $(T^-(U), T^+(U))$.
\end{claim}

This Claim and Theorem~\ref{compactness} show that $\vec U = (\pm W,0)$ up to scaling. As this true for any sequence $\{ t_n \}$, a diagonal argument gives that
\begin{gather} \label{eq:a_pm_conv}
d(\vec a(t), \q O^+ \cup \q O^-) \to 0 \quad \text{as } t \to +\infty,
\end{gather}
where $d$ is the $\HH$-distance to a set and
\[ \q O^\pm = \left\{ \left( \pm \frac{1}{\lambda} W \left( \frac{\cdot}{\lambda} \right), 0 \right) \middle| \ \lambda >0 \right\}. \]
Observe that
\[ d_0: = d(\q O^+, \q O^-) >0. \]
Indeed, 
\begin{align*}
d(\q O^+, \q O^-) & = \inf_{\lambda_1>0, \lambda_2>0} \left\|  \frac{1}{\lambda_1^2} \nabla W \left( \frac{\cdot}{\lambda_1} \right) + \frac{1}{\lambda_2^2} \nabla W \left( \frac{\cdot}{\lambda_2} \right) \right\|_{L^2} \\
& = \inf_{\lambda >0} \tilde d(\lambda), \quad \text{where} \quad \tilde d(\lambda) := \left\|  \frac{1}{\lambda^2} \nabla W \left( \frac{\cdot}{\lambda} \right) + \nabla W \right\|_{L^2}.
\end{align*}
Now $\tilde d(\lambda) \to 2 \| \nabla W \|_{L^2}$ as $\lambda \to 0$ or as $\lambda \to +\infty$, hence its minimum either greater or equal to $2 \| \nabla W \|_{L^2}$ and we are done; or attained at some $\lambda_0 >0$, and as $\nabla W \ne 0$, $d(\lambda_0) >0$.
Now define the sets of time
\[ \q U_\pm = \{ t \ge 0 \mid  d( \vec {a}(t), \q O^\pm) < d_0/2) \} . \]
By definition of $d_0$, $\q U_+$ and $\q U_-$ are disjoint. We also just proved that for some $T_0$ large, $[T_0, +\infty ) \subset \q U_+ \cup \q U_-$, and by continuity of $t \mapsto \vec a(t)$, both $\q U_+$ and $\q U_-$ are open.

Now recall that $\bar t_n \in \q U_+$ and $\bar t_n \to +\infty$. Therefore, $\q U_+ \cap [T_0, +\infty)$ is not empty, and by connectedness, $[T_0, +\infty) \subset \q U_+$. In view of \eqref{eq:a_pm_conv}, we infer that there exists a function  $\lambda(t) >0$ such that
\[ \left( \lambda(t) a \left( t, \lambda(t) \cdot \right), \lambda^2(t) \partial_t a \left( t, \lambda(t) \cdot \right) \right) \to (W,0) \quad \text{in } \HH \text{ as } t \to +\infty. \]
As $d_0 >0$, we see that the assumptions of Lemma~\ref{lem:cont_act} are fulfilled (with $G = ((0,+\infty), \times)$ acting on $\HH$ by $\lambda .(v_0,v_1) = (\lambda v_0(\lambda \cdot), \lambda^2 v_1 (\lambda \cdot))$ ), so that $\lambda$ can be chosen continuous. This concludes the proof of Theorem~\ref{bu 2W}.

\begin{rem}
Note that in proving the last step (proving that the sign of $(\pm  W,0)$ does not depend on the sequence $\{ t_n \}$),  the use of  Lemma~\ref{lem:cont_act} could be avoided by introducing the explicit scaling parameter 
\EQ{\label{cont scale}
\la(t):=  \left\{ \mu >0 \ \middle| \  \int_{r \le \mu} a_r^2(t, r) + a_t^2(t, r) \, r^3 \, dr \ge \int_{ r\le 1} W_r^2(r) \, r^3 \, dr\right\},
} 
and a continuity argument as in \cite[pages~590-591, Step~3]{DKM1}. But we present it in this way as Lemma~\ref{lem:cont_act} may be useful in other settings.
\end{rem}

\section{Global type-II solutions below $2\|\na W\|_{L^2}$}\label{2W global}

This section is devoted to proving Theorem~\ref{global 2W}. We assume that $\vec u(t)$ does not scatter in forward in time, so that our goal is to prove the second case of the dichotomy, namely relaxation to $W$. As in the statement of Theorem~\ref{global 2W}, we assume that there exists an $A>0$ so that 
\EQ{ \label{2W bound1}
\limsup_{t \to \infty}  \|  u(t) \|_{ \dot{H}^1(0\le r \le t-A)}^2 < 2 \|\na W\|_{L^2}^2.
}
Recall that we have already obtained a convergence for at least one sequence of times: more precisely, there exists a sequence of times $(\bar t_n)$ with $\bar t_n \to +\infty$, an integer $\bar J$,  scales $(\lambda_{1,n})_n,  \dots, (\lambda_{n ,\bar J})_n$ where
\[ 0 \ll \lambda_{1,n} \ll \cdots \ll \lambda_{ \bar J, n} \ll \bar t_n, \]
and $\bar J$ signs $\io_1, \dots \io_{\bar J} \in \{ -1, +1 \}$, such that
\begin{gather} \label{eq:conv_tn_bar}
\vec u(\bar t_n) = \sum_{j=1}^{\bar J} \left( \frac{\io_j}{\lambda_{j,n}} W \left( \frac{x}{\lambda_{j,n}} \right) ,0 \right) + \vec v_{\mr L}(\bar t_n) + o_{\HH}(1) \quad \text{as} \quad n \to +\infty.
\end{gather}
We again divide the proof of Theorem~\ref{global 2W} into several steps.

\subsection{Step 1: Preliminaries on profiles}

Denote by $\vec v(t)$ the nonlinear profile associated to $\vec v_{\mr L}$ at $+\infty$, that is, $\vec v$ is the unique solution to \eqref{u eq} such that
\[ \| \vec v(t) - \vec v_{\mr L}(t) \|_{\HH} \to 0 \quad \text{as} \quad t \to +\infty. \]
Again, we let
\begin{gather}\label{def:a}
\vec a(t) = \vec u(t) - \vec v_{\mr L}(t).
\end{gather}
We proved in the previous section that for all $\lambda >0$
\begin{gather*} \label{eq:a_out_cusp}
\int_{\lambda t}^{+\infty} \left( |\nabla_{t,x} a(t,x)|^2 + \frac{|a(t,x)|^2}{|x|^2} \right) dx \to  0 \quad \text{as} \quad t \to +\infty.
\end{gather*}
Due to the bound \eqref{2W bound1} and recalling the first statement of Claim \ref{cl:phi_psi}, (and making $T$ larger if necessary) we have
\begin{gather} \label{eq:a_bd_2W}
\forall t \ge T, \quad \| \nabla_x a(t) \|_{L^2}^2 \le 2 \| \nabla W \|_{L^2}^2 - \delta/2.
\end{gather}
The convergence \eqref{eq:conv_tn_bar} becomes
\begin{gather} \label{eq:a_tn_bar_decomp}
\vec a(\bar t_n) =  \sum_{j=1}^{\bar J} \left( \frac{\io_j}{\lambda_{j,n}} W \left( \frac{x}{\lambda_{j,n}} \right) ,0 \right) + o_{\HH}(1) \quad \text{as} \quad n \to +\infty.
\end{gather}
By orthogonality arguments, \eqref{eq:a_bd_2W}  implies $\bar J \le 1$.

First recall that $E(\vec a(t))$ has a limit as $t \to +\infty$.

\begin{claim} \label{cl:energy_a}
As $t \to +\infty$, $E(\vec a(t))$ converges to
\[ E(\vec u) - E(\vec v) = E(\vec u) - \frac{1}{2} \| \vec v_{\mr L}(0) \|_{\HH}^2 =  \bar J E(W). \]
\end{claim}

\begin{proof}
Observe that $\| \nabla_{t,x} v_{\mr L}(t) \|_{L^2}$ is constant because $\vec v_{\mr L}$ is a linear solution. Also, $\| v_{\mr L}(t) \|_{L^4} \to 0$. Therefore, $\ds E(\vec v_{\mr L}(t)) \to \frac{1}{2} \| \nabla_{t,x} v_{\mr L}(0) \|_{L^2}^2$. Hence $\ds E(\vec v) = \frac{1}{2} \| \nabla_{t,x} v_{\mr L}(0) \|_{L^2}^2$ and (recalling Claim~\ref{cl:phi_psi}) 
\begin{align*} 
& E(\vec a(t)) = E(\vec u(t) - \vec v_{\mr L}(t)) = \int_{|x| \le t/2} \left( \frac{1}{2} |\nabla_{t,x} u(t,x) |^2 - \frac{1}{4} |u(t,x)|^4 \right) dx + o(1) \\
& \qquad + \int_{x \ge t/2}  \left( \frac{1}{2} |\nabla_{t,x} u(t,x) - \nabla_{t,x} \vec v_{\mr L} (t,x) |^2 - \frac{1}{4} |u(t,x) - v_{\mr L}(t,x)|^4 \right) dx \\
& = \int_{|x| \le t/2} \left( \frac{1}{2} |\nabla_{t,x} u(t,x) |^2 - \frac{1}{4} |u(t,x)|^4 \right) dx + o(1) \\
& = \int \left( \frac{1}{2} |\nabla_{t,x} u(t,x) |^2 - \frac{1}{4} |u(t,x)|^4 \right) dx \\
& \qquad - \int_{|x| \ge t/2} \left( \frac{1}{2} |\nabla_{t,x} v_{\mr L} (t,x) |^2 - \frac{1}{4} |v_{\mr L}(t,x)|^4 \right) dx + o(1) \\
& = E(\vec u) - \frac{1}{2} \| \nabla_{t,x} v_{\mr L}(0) \|_{L^2}^2 + o(1).
\end{align*}
Hence $E(\vec a(t)) \to E(\vec u) - \frac{1}{2} \| \nabla_{t,x} v_{\mr L}(0) \|_{L^2}^2$.

Now consider the sequence $E(\vec a(\bar t_n))$: in view of the decomposition \eqref{eq:a_tn_bar_decomp}, and orhogonality, $E(\vec u(\bar t_n)) = \bar J E(W) + E(\vec v_{\mr L}(\bar t_n)) + o(1)$ as $n \to +\infty$. Taking the limit,  
there holds $E(\vec a(\bar t_n)) \to \bar J E(W)$. As we have seen that $E(\vec a(t))$ has a limit, it is $\bar J E(W)$.
\end{proof}

\begin{claim} \label{cl:J=1}
$\bar J=1$ and up to considering $-u$ instead of $u$, we may also assume $\io_1 = +1$.
\end{claim}

\begin{proof}
Claim~\ref{cl:energy_a} and the condition \eqref{eq:a_bd_2W} show that $\bar J$ is $0$ or $1$. Assume $\bar J=0$. In this case, $E(\vec a(t)) \to 0$. Now the second part of Lemma~\ref{E<2W} together with  \eqref{eq:a_bd_2W} implies that $\| \nabla_{t,x} a(t) \|_{L^2} \to 0$.
Therefore, $\| \nabla_{t,x} u(t) - \nabla_{t,x} v_{\mr L}(t) \|_{L^2} \to 0$ and $\vec u$ scatters forward in time. But this contradicts our initial assumption and hence $\bar J=1$.
\end{proof}

We now point out some properties of the profile decomposition for any sequence $\vec a(t_n)$ for large times.

Let $\{t_n\}_n$ be any sequence such that $t_n \to +\infty$. Up to extraction, the sequence $\vec a(t_n)$ admits a profile decomposition $\{\vec U_{\mr L}^j, \lambda_{j,n}, t_{j,n}\}_{j \ge 1}$ ordered for $\preccurlyeq$ (recall Lemma \ref{lem:3.7}). Let us denote by $\vec U^j$ the associated nonlinear profiles.

Using \cite[p. 154-155]{BG}, and  \eqref{eq:a_out_cusp} there exists $C$ independent of $j$ and $n$ such that
\[ \lambda_{j,n} \le C t_n, \quad |t_{j,n}| \le C t_n. \]

\begin{claim}
Define $\vec U_{\mr L}^0 = \vec v_{\mr L}(0)$, $\lambda_{0,n} = 1$ and $t_{0,n} =t_n$ (with nonlinear profile $U^0(t) = \vec v(t)$). Then $\{\vec U_{\mr L}^j, \lambda_{j,n}, t_{j,n}\}_{j \ge 0}$ is a profile decomposition for $\vec u(t_n)$.
\end{claim}

\begin{proof}
The point is to prove the pseudo-orthogonality property: but this is a consequence of the construction of a profile decomposition and $S(-t) u(t) \tendf \vec v_{\mr L}(0)$ weakly in $\HH$.
\end{proof}
Next, observe that since $\vec u$ does not scatter in forward time, by Proposition~\ref{nonlin profile} at least one of the nonlinear profiles $\vec U^j$ does not scatter in forward time and due to the ordering, $\preccurlyeq$, this means that $\vec U^1$ does not scatter in forward time. 
Also, as $\vec U^0 =\vec v$ scatters as $t \to \infty$ and $\vec U^1$ does not, we can conclude that $0 \not\preccurlyeq 1$.

Fix $J \in \N$. Due to the Pythagorean expansion of the $\HH$ norm \eqref{free en dec} and the bound on $\vec a$ \eqref{eq:a_bd_2W}, we have
\[ \forall j \ge 1, \quad \| \nabla_x U^j(-t_{j,n}/\lambda_{j,n}) \|_{L^2}^2 \le 2 \| \nabla W \|_{L^2}^2 - \delta/2 + o_n(1). \]
and the same for $w^J_n(0)$. In particular, 
it follows from Lemma \ref{E<2W} that
\[ \forall j \ge 1, \exists \, n_0(j), \quad n \ge n_0(j) \Rightarrow  E \left( U^j \left( - \frac{t_{j,n}}{\lambda_{j,n}} \right), \frac{1}{\lambda_{j,n}} \partial_t U^j \left(- \frac{t_{j,n}}{\lambda_{j,n}} \right) \right) \ge 0. \]
and similarly,
\[ \forall j \ge 1, \quad n \ge n_0(j) \Rightarrow E( w_{n}^j(0), w_{n}^j(0)) \ge 0. \]
As 
\[ E(U^j) = \lim_{n \to +\infty} E \left( U^j_{\mr L} \left( - \frac{t_{j,n}}{\lambda_{j,n}} \right), \frac{1}{\lambda_{j,n}} \partial_t U^j_{\mr L} \left(- \frac{t_{j,n}}{\lambda_{j,n}} \right) \right), \]
using again Lemma \ref{E<2W}, one can prove: 

\begin{claim} \label{cl:En_pos}
For all $j \ge 1$:
\begin{enumerate}
\item Either $E(\vec U^j) >0$, or $\vec U^j = \vec U^j_{\mr L} =0$.
\item If $E( \vec w_{n}^j(0)) \to 0$ as $n \to +\infty$, then $\| \nabla_{t,x} w_{n}^J(0) \|_{L^2} \to 0$.
\end{enumerate}
\end{claim}

This situation is the general one, more precisely, as $S(-t_n)\vec u(t_n) \tendf \vec v_{\mr L}(0)$ as $n \to +\infty$, and from the construction of profile decomposition (see \cite{BG}), we have

\begin{claim} \label{cl:v_in_prof}
Let $\{t_n\}_n$ be any sequence tending to $+\infty$.  The sequence $\vec u(t_n)$ admits a profile decomposition $\{\vec U_{\mr L}^j, \lambda_n^j, t_n^j\}_{j \ge 1}$ ordered for $\preccurlyeq$. Then $\ds \vec v_{\mr L}$ appears in the decomposition: i.e., for some $J_L \ge 2$, 
\[ \vec U_{\mr L}^{j_L} = \vec v_{\mr L}, \quad \lambda_n^{J_L} =1, \quad t_n^{J_L} = t_n. \]
Also, $\{\vec U_{\mr L}^j, \lambda_n^j, t_n^j\}_{j \ne J_L}$ is a $\preccurlyeq$-ordered profile decomposition of $\vec a(t_n)$.
\end{claim}

\subsection{Step 2: Minimization process and consequences}
As in the finite time blow-up case we follow the scheme developed in \cite{DKM6}. We recall that we have assumed that $\vec u$ does not scatter.

We define $\cS_0$ to be set of sequences of times $\{t_n\}_n$ such that $t_n \to +\infty$ and $\vec u(t_n)$ admits a $\preccurlyeq$-ordered profile decomposition $\{\vec U_{\mr L}^j, \lambda_{j,n}, t_{j,n}\}$.



\begin{lem} \label{lem:S_0_1prof}
Let  $\{t_n\}_n \in \cS_0$, with $\preccurlyeq$-ordered profile decomposition $\{\vec U_{\mr L}^j, (\lambda_{j,n}, t_{j,n}\}$ and nonlinear profiles $\vec U^j$. Then $\vec U^1$ does not scatter forward in time and for all $j \ge 2$, $\vec U^j$ does scatter forward in time. Furthermore, $E(\vec U^1) \ge E(W,0)$.
\end{lem}

\begin{proof}
We again use many ideas from \cite[Section 4]{DKM6}, adapted to the current situation. For $\q T = \{\tau_n\}_n \in \cS_0$, let $J_0(\q T)$ be the number of nonlinear profiles that do not scatter forward in time. By definition, if $j \le J_0$,  then $\vec U^j$ does not scatter forward in time, and for $j \ge J_0+1$, $U^j$ scatters forward in time.

As $\vec u$ does not scatter in forward time, $J_0(\q T) \ge 1$. On the other hand, recall that due to the small data theory, if a linear solution $\vec U_{\mr L}$ has small $\HH$ norm (say, less that $\| \vec U_{\mr L} \|_{\HH} \le \delta_0$), any nonlinear profile $\vec U^j$ associated to it scatters in both time directions. Due to the Pythagorean expansion \eqref{free en dec}  and the bound \eqref{typeII}, $J_0(\q T) \le CM/\delta_0^2$ is (finite and) uniformly bounded on $\cS_0$.

Similarly, let
\[ J_1(\q T)  = \min \{ j \ge 1 \mid j \prec j+1 \}. \]
By the definition of  $J_0(\q T)$, we have $J_0(\q T) \prec J_0(\q T)+1$, and therefore $J_1(\q T) \le J_0(\q T)$. In particular, $J_1$ is also uniformly bounded.
Then define
\begin{align} \label{def:S1}
&J_{\mr{max}}   = \max \{ J_0(\q T) \mid \q T \in \cS_0 \}, \\
&\cS_1  = \{ \q T \in \cS_0 \mid J_0(\q T)=  J_{\mr{max}} \}.
\end{align}
For $\q T \in \cS_1$, we define the non scattering energy $\q E(\T)$ as the sum of the energies of the nonlinear profiles that do not scatter as $t \to \infty$: more precisely, denoting $U^j$ the nonlinear profiles appearing in the profile decomposition derived from $\q T$, we let
\[ \q E(\q T) := \sum_{j=1}^{J_0(\q T)} E(\vec U^j) = \sum_{j=1}^{J_{\mr{max}}} E(\vec U^j). \]
We now recall the following result from~\cite{DKM6}. 

\begin{claim}[{\cite[Lemma 4.5 and Corollary 4.3]{DKM6}}]
The infimum of $\q E(\T)$ is attained (and hence is a minimum): i.e, there exists $\bar{\q T} \in \cS_1$ such that
\[  \q E(\bar{\q T}) = \inf \{  \q E(\q T) \mid \q T \in \cS_1 \} =: \q E_{\mr{min}}. \]
\end{claim}

With the above claim, we can then define
\begin{align} \label{def:S2}
&\cS_2  = \{ \q T \in \cS_1 \mid \q E(\q T) =  \q E_{\mr{min}} \} \ne  \varnothing, \\
&J_{\mr{min}}  =  \min \{ J_1(\q T) \mid \q T \in \cS_2 \}, \\
&\cS_3  = \{ \q T \in \cS_2 \mid J_1(\q T) = J_{\mr{min}}  \} \ne \varnothing.
\end{align}

We can again use Lemma~\ref{lem412} to conclude that $J_{\min}=1$ and that there exists a sequence $\T_0 \in \cS_3$ so that $\vec U^1 =(\pm W, 0)$. We need to also show that $J_{\max}=1$. 

 Recall that $\vec v_{\mr L}$ must appear in the profile decomposition, at some index $J_L$  with $J_L > J_{\mr{max}}$ because $\vec v_{\mr L}$ has a scattering nonlinear profile.
Now write the Pythagorean expansion of the energy \eqref{nonlin en dec} for $J=J_L$, denoting by $\{\vec U^j\}_j$ the nonlinear profiles associated to the profile decomposition of $\vec u (\tau_n)$:
\[ E(\vec u) = E(W,0) +  \sum_{j=2}^{J_{\mr{max}}} E(\vec U^j) +   \sum_{j=J_{\mr{max}}+1}^{J_L-1} E(\vec U^j) + E(\vec v) + E(\vec w^{J_L}_n(0)) + o_n(1). \]
Recall that $E(\vec u) = E(W,0) + E(\vec v)$. Along with Claim \ref{cl:En_pos} this allows us to  deduce that
\[ \forall j =2, \dots, J_L-1, \quad \vec U^j=0, \quad \text{and} \quad \vec w^{J_L}_n(0)= o_{\HH}(1),  \]
so that in particular $J_{\mr{max}}=1$, and $\q E_{\mr {min}} = \q E(\q T_0) = E(W,0)$. This proves Lemma~\ref{lem:S_0_1prof}. Also notice that we obtained for some $\io \in \{ \pm 1\}$,
\[ \vec u(\tau_n) = \left( \frac{\io}{\lambda_{1,n}} W \left( \frac{\cdot}{\lambda_{1,n}} \right) ,0 \right) + \vec v_{\mr L}(\tau_n) + o_{\HH}(1). \qedhere \]
\end{proof}

\subsection{Step 3: Compactness of the singular part up to scaling}

\begin{lem}
$\vec a(t)$ has the compactness property on $[0,+\infty)$, meaning that there exists a function $\la: [0, \infty) \to [0, \infty)$ so that the set 
\EQ{
K(\vec a, \la) =  \{ (\la(t) a(t, \la(t)  \cdot), \la^2(t) a_t(t, \la(t) \cdot) ) \mid t \in [0, \infty)\}
} 
is pre-compact in $\HH$. 
\end{lem}

\begin{proof}
$t\mapsto \vec a(t)$ is continuous so one only has to check compactness up to modulation in a neighbourhood of $+\infty$. Let $\{t_n\}_n$ be any sequence tending to $+\infty$. 
After passing to a subsequence, still denoted by $t_n$,  we can ensure  that $t_n \ge 1$ for all $n$ and $\{t_n\}_n \in \cS_0$, i.e. $\vec u(t_n)$ admits a  $\preccurlyeq$-ordered profile decomposition $\{\vec U_{\mr L}, \lambda_{j,n}, t_{j,n} \}$ with nonlinear profiles $\{\vec U^j\}_j$.

By Lemma \ref{lem:S_0_1prof}, we know that $\vec U^1$ does not scatter in forward time, and that $E(\vec U^1) \ge E(W,0)$. Also by Claim \ref{cl:v_in_prof}, $\vec v_{\mr L}$ appears in the profile decomposition, say as profile $\vec U^{J_L}_{\mr L} = \vec v_{\mr L}$ (we recall that its nonlinear profile is $\vec v$).

Let us first prove that all nonlinear profiles other than $\vec U^1$ and $\vec v$ vanish, that is: $\vec U^j=0$ for all $j \ge 2$, $j \ne J_L$. Indeed, write the Pythagorean expansion of the energy \eqref{nonlin en dec} for $J \ge J_L$:
\[ E(\vec u(t_n)) = E(U^1) + \sum_{j=2}^{J_L-1} E(\vec U^j) + E(\vec v) + \sum_{j= J}^{J+1} E(\vec U^j) + E(\vec w^{J_L}_n(0)) + o_n(1). \]
Recall again that all the profiles have positive energy or are 0 according to Claim~\ref{cl:En_pos}; in view of Claim \ref{cl:energy_a}, (and letting $n \to +\infty$) we infer that
\[ \forall j \ge 2, \ j \ne J_L, \quad \vec U^j =0, \quad \text{and} \quad \vec w^{J_L}_n(0) = o_{\HH}(1). \]
From there, we see $E(\vec U^1) = \q E(\{t_n\})$ and that the profile decomposition for $\vec u(t_n)$ can be written as
\begin{align} 
 \vec u(t_n,r) & = \left( \frac{1}{\lambda_{1,n}} U^1_{\mr L} \left( - \frac{t_{1,n}}{\lambda_{1,n}}, \frac{r}{\lambda_{1,n}} \right) , \frac{1}{\lambda_{1,n}^2} \partial_t U^1_{\mr L} \left( - \frac{t_{1,n}}{\lambda_{1,n}}, \frac{r}{\lambda_{1,n}} \right) \right) \nonumber \\
& \qquad + \vec v_{\mr L}(t_n,r) + o_{\HH}(1). \label{eq:u(tn')_prof}
\end{align}
Let us now show that

\begin{claim}
$t_{1,n}=0$ for all $n$. 
\end{claim}

\begin{proof}
In view of Remark~\ref{prof rem}, it suffices to show that $ - t_{1,n}/\lambda_{1,n}$ does not converge to $\pm \infty$. If  $ - t_{1,n}/\lambda_{1,n} \to +\infty$, then $\vec U^1$ would scatter forward by definition of a nonlinear profile: this is not the case. Let us argue by contradiction and assume that $\ds - t_{1,n}/\lambda_{1,n}\to - \infty$. Then again by definition of a nonlinear profile, $\vec U^1$ scatters backwards in time. Now let $t_0 > 1+ T^-(\vec v)$ be large enough (recall $\vec v$ is the nonlinear profile of $(\vec v_{\mr L}, +\infty)$), and  evolve the profile decomposition \eqref{eq:u(tn')_prof} with Proposition~\ref{nonlin profile} backwards in times up to time $\tau_n = t_0-  1 - t_n$ (which is possible by the choice of $t_0$, in view of the lifespans of $\vec U^1$ and $\vec v$). As $t_0 \in (t_n  + \tau_n, t_n] = (t_0-1,t_n]$, we have
\begin{align*} 
 \vec u(t_0,r) & = \left( \frac{1}{\lambda_{1,n}} U^1_{\mr L} \left( \frac{ t_0 - t_n -  t_{1,n}}{\lambda_{1,n}}, \frac{r}{\lambda_{1,n}} \right) , \frac{1}{\lambda_{1,n}^2} \partial_t U^1_{\mr L} \left( \frac{t_0 - t_n - t_{1,n}}{\lambda_{1,n}}, \frac{r}{\lambda_{1,n}} \right) \right)  \\
& \qquad + \vec v(t_0,r) + o_{\dot H^1 \times L^2}(1).
\end{align*}
This is a non trivial profile decomposition for the fixed function $\vec u(t_0) - \vec v(t_0)$, hence the only possibility is $\lambda_{1,n} =1$ and $t_0 - t_n -  t_{1,n}= c_0$ for all $n$. But then $t_{1,n} \to -\infty$ like $-t_n$ and $\ds - \frac{t_{1,n}}{\lambda_{1,n}}\to + \infty$, which is a contradiction. 
\end{proof}

We have obtained that
\begin{gather} \label{eq:u_decomp}
 \vec u(t_n) =  \left( \frac{1}{\lambda_{1,n}} U^1_{\mr L} \left(0, \frac{r}{\lambda_{1,n}} \right) , \frac{1}{\lambda_{1,n}^2} \partial_t U^1_{\mr L} \left( 0, \frac{r}{\lambda_{1,n}} \right) \right) + \vec v_{\mr L}(t_n) + o_{\HH}(1),
 \end{gather}
and so, 
\[ (\lambda_{1,n} a(t_n, \lambda_{1,n} \cdot), \lambda_{1,n}^2 \partial_t a(t_n, \lambda_{1,n} \cdot)) \to \vec U^1(0) \quad \text{as} \quad n \to +\infty. \]
As this is true for a subsequence of all sequences $t_n \to \infty$, we see that there exists a function $t \mapsto \lambda(t)$ such that $K(\vec a,\lambda)$ has compact closure in $ \HH$. 
\end{proof}

\subsection{Step 4: Convergence to $(W,0)$ and conclusion of the proof of Theorem~\ref{global 2W}}

Here the argument is exactly the same as in Step 4 of the proof of Theorem~\ref{bu 2W}, and we only sketch it.

Given any sequence $\{ t_n \}$ tending to $+\infty$, by by Step 3, \cite[Lemma 8.5]{DKM1} and Theorem \ref{compactness}, we have the convergence, up to a subsequence
\[ \left( \lambda(t_n) a \left( t_n, \lambda(t_n) \cdot \right),  \lambda(t_n)^2 \partial_t a \left( t_n, \lambda(t_n) \cdot \right) \right) \to (\pm W,0). \] 
Then a continuity argument shows that the sign does not depend on the sequence $\{ t_n \}$, and so there exists a scaling parameter $\lambda$ \emph{defined for all times}, such that
\[ \left( \lambda(t) a \left( t_n, \lambda(t) \cdot \right),  \lambda(t)^2 \partial_t a \left( t_n, \lambda(t) \cdot \right) \right) \to (W,0) \quad \text{as} \quad t \to +\infty. \]
Finally Lemma \ref{lem:cont_act} shows that $\lambda$ can be chosen continuous.
(As in the blow up case, we could have used the argument in \cite[pages~590-591, Step~3]{DKM1}).
 
This concludes the proof of Theorem~\ref{global 2W}.

\appendix

\section{Continuity of scaling functions}

\begin{lem} \label{lem:cont_act}
Let $(B, \| \cdot \|)$ be a Banach space and $G$ be a group of isometries of $B$ (it is a metric space endowed by the operator norm that we stil denote $\| \cdot \|$: $\| g \| = \sup \{ \| g.v \| \mid \| v \| \le 1\}$). We assume that $G$ is locally path connected. 

Let $v \in \q C([0,+\infty), B)$, and assume that there exists $v_0 \in B$ and a function $g: [0,+\infty) \to G$ such that $g(t). v(t) \to v_0$ in $B$ as $t \to +\infty$. Also assume that $G$ acts properly on $v_0$, in the sense that if $g_n .v_0 \to v_0$ in $B$, then $g_n \to \Id$ in $G$.

Then the action can be chosen to be continuous, i.e there exist $\gamma \in \q C( [0,+\infty), G)$ such that $\gamma(t).v(t) \to v_0$ in $B$ as $t \to +\infty$.
\end{lem}

Notice that if $G$ is a Lie group, it is automatically locally path connected, and so only the proper action hypothesis is to be checked.

\begin{proof}
If $v_0 =0$, then $\| v(t) \| = \| g(t).v(t) \| \to 0$ and so $\gamma(t) = \Id$ works. Let us assume in the following that $v_0 \ne 0$.

As $G$ acts by isometries, we can assume without loss of generality that for all $t\ge 0$, $\| v(t) \| \ge 1$.
For $t \ge 1$, define an adequate modulus of continuity 
\[ d(t) = \sup \{\delta \in [0,1] \mid \forall \tau, \tau' \in [t-\delta, t+\delta], \ \| v(\tau) - v(\tau') \| \le 1/t \}. \]
Define now by induction the sequence of times $t_0=1$ and $t_{n+1} = t_{n} + d(t_{n})$ for $n \ge 0$. We claim that $t_n \to +\infty$. 

Indeed, if not, $t_n \to t_\infty \in [0,+\infty)$ and $t_n \le t_\infty$ for all $n$. Now observe that if $\tau \in [t-d(t),t]$, then for $\delta = d(t) - t+\tau \ge 0$, $[\tau - \delta, \tau + \delta] \subset [t-d(t), t+d(t)]$ so that $d(\tau) \ge d(t) -t +\tau$. Then for $n$ large enough, $t_n \ge t_\infty - d(\infty)/3$ and $d(t_n) \ge 2d(t_\infty)/3$ and $t_{n+1} \ge t_{\infty} + d(t_\infty)/3 > t_\infty$, a contradiction. Hence $t_n \to +\infty$.

\medskip

Then 
\begin{align*}
& \| g(t_{n+1})g(t_{n})^{-1} . v_0 - v_0 \|  =   \| g(t_{n+1})^{-1} . v_0 - g(t_{n})^{-1} . v_0 \| \\
& \le  \| g(t_{n+1})^{-1}  . v_0 - v(t_{n+1}) \| + \| v(t_{n+1}) - v(t_{n}) \| + \| v(t_{n})  - g(t_{n})^{-1}. v_0 \| \\
&  \le \| g(t_{n+1}) . v(t_{n+1}) - v_0 \|  + d(t_{n}) +  \| g(t_{n}) . v(t_{n}) - v_0 \| \to 0.
\end{align*}
Therefore, by proper action, $g(t_{n+1})g(t_{n})^{-1} \to \Id$ as $n \to +\infty$.

For $m \in \N$, let $\q V_m$ be a path connected open set of $G$ such that $\Id \in \q V_m \subset B_G(\Id, 1/m)$ (such a $\q V_m$ exists because $G$ is path connected). Let 
\[ m(n) = \begin{cases}
\max \{ m \mid g(t_{n+1})g(t_{n})^{-1} \in \q V_m \} & \text{if } g(t_{n+1}) \ne g(t_{n}), \\
n & \text{if } g(t_{n+1}) = g(t_{n}).
\end{cases} \]
This is constructed so that $m(n) \to +\infty$ as $n \to +\infty$.  As $\q V_{m(n)}$ is path connected, there exists a path $\gamma_n$ such that $\gamma_n(0) = \Id$, $\gamma_n(1) = g(t_{n+1})g(t_{n})^{-1}$ and $\gamma_n ([0,1]) \subset \q V_{m(n)}$.
 
Finally define $\gamma: [1,+\infty) \to G$ in the following way: let $t \ge 1$, then there exists a unique $n \in \N$ such that $t \in [t_{n-1}, t_n)$ and we set 
\[ \gamma(t) = \gamma_n \left( \frac{t - t_{n}}{t_{n+1} - t_{n}} \right) g(t_{n}). \]
Observe that $\gamma$ is continuous; for all $n \in \N$, $\gamma(t_n) = g(t_n)$; and for $t \in [t_{n}, t_{n+1})$,
\[  \| \gamma(t)g(t_n)^{-1} - \Id \| = \left\| \gamma_n\left( \frac{t - t_{n}}{t_{n+1} - t_{n}} \right) - \Id \right\| \le \frac{1}{m(n)}. \]
Therefore,
\begin{align*}
& \| \gamma(t) v(t)  - v(0) \|  =  \|v(t) - \gamma(t)^{-1}. v_0 \| \\
& \le \| v(t) - v(t_n) \| + \| v(t_n) - g(t_n)^{-1} . v_0 \| + \| g(t_n)^{-1} . v_0 - \gamma(t)^{-1}. v_0 \| \\
& \le \frac{1}{t_n} + o_n(1) + \| \gamma(t)g(t_n)^{-1} . v_0 - v_0 \| \\
& \le \frac{1}{t_n} + o_n(1) + \frac{1}{m(n)} \| v_0 \|.
\end{align*}
As $t_n \to +\infty$ and $m(n) \to +\infty$ as $n \to +\infty$, this means that $\gamma(t) v(t) \to v_0$ as $t \to +\infty$.
\end{proof}

\bibliographystyle{plain}
\bibliography{researchbib}

 \bigskip

  \centerline{\scshape Rapha\"el C\^{o}te }
\smallskip
{\footnotesize
\begin{center}
CNRS and \'Ecole Polytechnique \\
Centre de Math\'ematiques Laurent Schwartz UMR 7640 \\
Route de Palaiseau, 91128 Palaiseau cedex, France \\
\email{cote@math.polytechnique.fr}
\end{center}
} 

\medskip

\centerline{\scshape Carlos Kenig, Wilhelm Schlag}
\smallskip
{\footnotesize
 \centerline{Department of Mathematics, The University of Chicago}
\centerline{5734 South University Avenue, Chicago, IL 60615, U.S.A.}
\centerline{\email{cek@math.uchicago.edu, schlag@math.uchicago.edu}}
} 

 \medskip

\centerline{\scshape Andrew Lawrie}
\smallskip
{\footnotesize
 \centerline{Department of Mathematics, The University of California, Berkeley}
\centerline{970 Evans Hall \#3840, Berkeley, CA 94720, U.S.A.}
\centerline{\email{ alawrie@math.berkeley.edu}}
} 

\end{document}